\documentclass[leqno,12pt]{amsart}
\usepackage{amssymb}
\usepackage{amsmath}
\usepackage{enumerate}
\usepackage{amsfonts}
\usepackage{hyperref}
\usepackage{cancel}
\usepackage{multicol}
\usepackage{booktabs}
\usepackage{array, xcolor}
\usepackage{colortbl}
\usepackage{float}
\usepackage[normalem]{ulem}
\usepackage{soul}
\usepackage{tikz}

\usepackage{amsmath,amsthm}
\usepackage{graphicx}
\usepackage[T1]{fontenc}
\usepackage[utf8]{inputenc}
\usepackage{tikz, tikz-cd}
\usepackage{hyperref}
\usepackage{scalerel,stackengine}
\stackMath
\newcommand\reallywidecheck[1]{%
\savestack{\tmpbox}{\stretchto{%
  \scaleto{%
    \scalerel*[\widthof{\ensuremath{#1}}]{\kern-.6pt\bigwedge\kern-.6pt}%
    {\rule[-\textheight/2]{1ex}{\textheight}}
  }{\textheight}%
}{0.5ex}}%
\stackon[1pt]{#1}{\scalebox{-1}{\tmpbox}}%
}
\hypersetup{
    colorlinks=true,
    linkcolor= blue,
    citecolor =cyan,
    urlcolor = teal,
}

\usepackage{imakeidx}
\makeindex[name=Const, title=Constants]
\makeindex[name=Other, title=List of symbols]

\newtheorem{theorem}[equation]{Theorem}
\newtheorem{corollary}[equation]{Corollary}
\newtheorem{lemma}[equation]{Lemma}
\newtheorem{proposition}[equation]{Proposition}
\newtheorem{remark}[equation]{Remark}

\numberwithin{equation}{section}

\newcommand{\supp}{\text{\rm supp}\,}

\begin{document}

\title[]{Characterizations of $H^1$ and Fefferman-Stein decompositions of ${\rm BMO}$   functions\\
by systems of singular integrals \\
in the Dunkl setting}


\author[Jacek Dziubański]{Jacek Dziubański}
\author[Agnieszka Hejna]{Agnieszka Hejna}
\begin{abstract}  We extend the classical theorem of Uchiyama about constructive  Fefferman-Stein  decompositions of ${\rm BMO}$ functions by systems of  singular integrals to the rational Dunkl setting. On $\mathbb R^N$ equipped with a root system $R$ and a multiplicity function $k\geq 0$, let $dw(\mathbf x)=\prod_{\alpha\in R} |\langle \alpha, \mathbf x\rangle |^{k(\alpha)}\;d\mathbf x$ denote the associated measure and let $\mathcal F$ stand for the Dunkl {(Fourier-Dunkl)} transform.  Consider a system $(\theta_0,\theta_1,\theta_2,\dots,\theta_d)$  of smooth away from the origin and homogeneous of degree zero functions on $\mathbb R^N$  with $\theta_0(\xi)\equiv 1$. We prove that if 
\begin{displaymath}
\text{\rm rank} \, 
\left( \begin{array}{ccccc}
1&\theta_1(\xi) & \theta_2(\xi)& \ldots &\theta_d(\xi) \\
1&\theta_1(-\xi) & \theta_2(-\xi)& \ldots &\theta_d(-\xi)
\end{array} \right) =2 \quad \text{for all } \xi\in\mathbb R^N, \ \|\xi\|=1, 
\end{displaymath}
then any compactly supported ${\rm BMO}(\mathbb R^N, \|\mathbf x-\mathbf y\|,dw)$  function $f$ can be decomposed into  
$$ f=g_0+\sum_{j=1}^d \mathbf S^{\{j\}} g_j, \quad \Big\|\sum_{j=0}^d g_j\Big\|_{L^\infty} \leq C\|f\|_{{\rm BMO}}, $$ 
where $\mathbf S^{\{j\}} g=\mathcal F^{-1}(\theta_j\mathcal Fg)$.  As a corollary we obtain characterizations of the Hardy space $H^1_{\rm Dunkl}$  by the systems of the singular integral operators $ ({\rm Id}, \mathbf S^{\{1\}}, \mathbf S^{\{2\}},\dots, \mathbf S^{\{d\}})$. 
\end{abstract}

\address{Jacek Dziubański, Uniwersytet Wroc\l awski,
Instytut Matematyczny,
Pl. Grunwaldzki 2,
50-384 Wroc\l aw,
Poland}
\email{jdziuban@math.uni.wroc.pl}

\address{Agnieszka Hejna, Uniwersytet Wroc\l awski,
Instytut Matematyczny,
Pl. Grunwaldzki 2,
50-384 Wroc\l aw,
Poland 
\&
Department of Mathematics,
Rutgers University,
Piscataway, NJ 08854-8019, USA}
\email{hejna@math.uni.wroc.pl}

\maketitle

\enlargethispage{0.37cm}
\maketitle
\tableofcontents
\thispagestyle{empty}

\section{Introduction }\label{Sec:Intro}

{Characterizations of Hardy spaces on the  Euclidean spaces by systems of singular integrals have  their origin in description of the spaces by means of the Riesz transforms, see Stein-Weiss \cite{Stein-Weiss} and Fefferman-Stein \cite{FeSt}. 
A general  result}  of Uchiyama~\cite{Uchiyama} asserts that if a system  $(\theta_0(\xi), \theta_1(\xi),\theta_2(\xi),...,\theta_d(\xi))$ of complex valued functions on $\mathbb R^N$, which are $C^\infty$ away from the origin and homogeneous of degree zero, satisfies the condition 
\begin{displaymath}\label{eq:triangle_new}\tag{$\star$}
\text{\rm rank} \, 
\left( \begin{array}{ccccc}
\theta_0(\xi)&\theta_1(\xi) & \theta_2(\xi)& \ldots &\theta_d(\xi) \\
\theta_0(-\xi)&\theta_1(-\xi) & \theta_2(-\xi)& \ldots &\theta_d(-\xi)
\end{array} \right) =2 \quad \text{for all } \xi\in\mathbb R^N, \ \|\xi\|=1, 
\end{displaymath}
then the system of the Fourier multiplier operators $(S^{\{0\}}, S^{\{1\}},S^{\{2\}},...,S^{\{d\}})$, where 
$$ \widehat{S^{\{j\}}f}(\xi)=\theta_j(\xi)\widehat f(\xi),$$
characterizes the classical real Hardy space $H^1(\mathbb R^N)$, that is, 
an $L^1(\mathbb R^N)$-function $f$ belongs to $H^1(\mathbb R^N)$ if and only if $S^{\{j\}}f\in L^1(\mathbb R^N)$ for all $0 \leq j \leq d$  and there is a constant $C>0$ such that 
\begin{equation*}
    C^{-1} \| f\|_{H^1(\mathbb R^N)} \leq \sum_{j=0}^N \| S^{\{j\}}f\|_{L^1(\mathbb R^N)} \leq C\|f\|_{H^1(\mathbb R^N)}.
\end{equation*}

Here the actions of $S_j$ on $L^1(\mathbb R^N)$-functions are understood in the sense of distribution. 

The aim of this work is to extend the result of Uchiyama to the rational Dunkl setting. To be more precise, 
on $\mathbb R^N$ equipped with a root system $R$ and a  multiplicity function $k\geq 0$,  
 let 
 \begin{equation}\label{eq:measure}dw(\mathbf x)=\prod_{\alpha\in R} |\langle \alpha, \mathbf x\rangle |^{k(\alpha)}\;d\mathbf x \index[Other]{$dw$@$dw$}
 \end{equation}
  denote the associated measure, where here and subsequently $d\mathbf x$ means the Lebesgue measure.  Let $\mathcal Ff$ and $\mathcal F^{-1} f$ stand the Dunkl transform and the inversion of the Dunkl transform of $f$ respectively (see~\eqref{eq:Dunkl_transform}). 
  { Following the classical theory (see Fefferman-Stein \cite{FeSt}), 
  we say that an $L^1(dw)$-function $f$ belongs to the Hardy space $H^1_{\rm Dunkl}$ \index[Other]{$H^1_{\rm Dunkl}$@$H^1_{\rm Dunkl}$ - Hardy space}if the maximal function (built up from the Dunkl-heat semigroup)  }
  \begin{align*}
       \mathcal M f(\mathbf x)=\sup_{\|\mathbf x'-\mathbf x\|^2\leq t} |e^{t\Delta_k}f(\mathbf x')|
  \end{align*}
  belongs to $L^1(dw)$, where $\Delta_k$ is the Dunkl Laplace operator. Then we set 
  \begin{equation*} \| f\|_{H^1_{\rm Dunkl}}=\| \mathcal Mf\|_{L^1(dw)}.
  \end{equation*}
It was proved in \cite{DzHe-JFA} that if $\theta(\xi)$ is a homogeneous of degree zero and  $C^\infty$ away from the origin function on $\mathbb R^N$, then the Dunkl--Fourier multiplier operator 
$$ f\mapsto \mathbf Sf=\mathcal F^{-1} (\theta (\cdot)\mathcal F(\cdot)),$$
initially defined on $L^2(dw)$, has a unique extensions to a bounded operator on $L^p(dw)$ for $1<p<\infty$ and $H^1_{\rm Dunkl}$. Moreover, it is well defined in the sense of distributions on $L^1(dw)$ by the relation 
{
\begin{equation*}\langle \mathbf S f,\boldsymbol \varphi\rangle = \int_{\mathbb R^N} \theta(\xi)\mathcal Ff(\xi)(\mathcal F^{-1} \boldsymbol\varphi)(\xi)\, dw(\xi), \quad \boldsymbol\varphi\in\mathcal S(\mathbb R^N).
\end{equation*}
}

{In this paper, we consider a system $(\theta_0(\xi), \theta_1(\xi), \dots, \theta_d(\xi))$, where $\theta_0(\xi) \equiv 1$, consisting of functions that are smooth away from the origin and homogeneous of degree zero on $\mathbb{R}^N$. The presence of $\theta_0 \equiv 1$ in the system appears natural from the perspective of Hardy space characterizations, as $H^1_{\rm Dunkl} \subset L^1(dw)$. For our systems, condition \eqref{eq:triangle_new} thus takes the form  
\begin{displaymath}\label{eq:triangle}\tag{$\triangle$}
\text{\rm rank} \ 
\left( \begin{array}{ccccc}
1&\theta_1(\xi) & \theta_2(\xi)& \ldots &\theta_d(\xi) \\
1&\theta_1(-\xi) & \theta_2(-\xi)& \ldots &\theta_d(-\xi)
\end{array} \right) =2 \quad \text{for all } \xi\in\mathbb{R}^N, \ \|\xi\|=1.  
\end{displaymath}
}

We are now in a position to state our main result. 

\begin{theorem}\label{teo:main_1}
    Assume that a system $(\theta_0(\xi),\theta_1(\xi),\theta_2(\xi),...,\theta_d(\xi))$, {with $\theta_0(\xi)\equiv 1$},  of complex valued functions on $\mathbb R^N$, which are  $C^\infty$ away from the origin and  homogeneous of degree zero satisfies the condition~\eqref{eq:triangle}. Let $\mathbf S^{\{j\}}$ denote the multiplier operators   $\mathcal F(\mathbf S^{\{j\}}f)(\xi)=\theta_j(\xi)\mathcal F f(\xi)$. Then an $L^1(dw)$-function $f$ belongs to the Hardy space $H^1_{\rm Dunkl}$ if and only if $\mathbf S^{\{j\}}f$ belong to $L^1(dw)$,  $1 \leq j \leq d$. Moreover, 
    there is a constant $C>0$ such that 
    \begin{equation*}
    C^{-1} \| f\|_{H^1_{\rm Dunkl} }\leq \| f\|_{L^1(dw)}+\sum_{j=1}^d\| \mathbf S^{\{j\}}f\|_{L^1(dw)}\leq C \| f\|_{H^1_{\rm Dunkl} }.
    \end{equation*}
\end{theorem}

The characterization of $H^1_{\rm Dunkl}$ by a special system of multiplier operators, namely by $({\rm Id},\mathbf R_1,\dots, \mathbf R_N)$, where $\mathbf R_j$,  $j=1,\dots,N$, are the Dunkl-Riesz transforms,  which corresponds to the system the Dunkl multipliers: 
$$ \Big(1,-\frac{i\xi_1}{\|\xi\|},\dots, -\frac{i\xi_N}{\|\xi\|}\Big)$$
 was  proved in \cite{ADH-JFAA}. To this end, the authors of~\cite{ADH-JFAA} first proved  $\Delta_k$-subharmonicity of a function built from a vector of functions satisfying  generalized Cauchy--Riemann equations for the Dunkl operators using an action of the reflection group. This was a significant step in the proof.  Then, they adapted to the Dunkl setting  methods of \cite{FeSt} and \cite{Stein-Weiss} to complete the proof of the Dunkl-Riesz transform characterization of $H^1_{\rm Dunkl}$.   Since the space $H^1_{\rm Dunkl}$ admits  decomposition into Coifman-Weiss type atoms (see \cite{DH-Studia}), its dual can be identified with the space of ${\rm BMO}(\mathbf{X})$  functions on the space of  homogeneous type ${\mathbf{X}}=(\mathbb R^N,\|\mathbf x-\mathbf y\|, dw)$ (see Section \ref{sec:BMO}). Consequently, as a corollary of the  characterization of the Hardy space $H^1_{\rm Dunkl}$ by the Dunkl-Riesz transforms, one  obtains an analogue of the Fefferman--Stein decomposition (see \cite{FeSt}), namely, every ${\rm BMO}(\mathbf{X})$-function $f$ can be decomposed into  
$$ f=g_0+\sum_{j=1}^NR_j^*g_j,  
\text{   with   } \sum_{j=0}^N \| g_j\|_{L^\infty}\leq C\| f\|_{{\rm BMO}(\mathbf{X})}.$$

{The proof of the characterization of the classical $H^1(\mathbb R^N)$ spaces by the systems of singular integrals presented in~\cite{Uchiyama} does not rely on any variant of subharmonic majorization approach, but on complex, carefully considered, multi-step inductive constructions of various functions.}
So,   to prove Theorem \ref{teo:main_1} we apply the  ideas of Uchiyama \cite{Uchiyama} which in some sense inverse the order. First we give a detailed and self-contained proof of a counterpart   of the Fefferman-Stein decomposition of compactly supported ${\rm BMO}(\mathbf{X})$ functions which is stated below as  our  second main result.

\begin{theorem}\label{UchiyamaDecomposition0} Assume that a system $( \theta_0(\xi),  \theta_1(\xi),\theta_2(\xi),...,\theta_d(\xi))$, with $\theta_0(\xi)\equiv 1$, of homogeneous of degree zero and smooth away from the origin multipliers  satisfies ~\eqref{eq:triangle}.  Let $\mathbf S^{\{j\}}$ and $\mathbf S^{\{j\}*}$ denote the Fourier-Dunkl  multiplier operators   $\mathcal F(\mathbf S^{\{j\}}f)(\xi)=\theta_j(\xi)\mathcal F f(\xi)$, $\mathcal F(\mathbf S^{\{j\}*}f)(\xi)=\overline{\theta_j(\xi)}\mathcal F f(\xi)$ respectively.  Then there is a constant $C>0$ such that any  compactly supported ${\rm BMO}(\mathbf{X})$-function $f$ can be written as
\begin{equation*}  f= \boldsymbol g_0+\sum_{j=1}^d \mathbf S^{\{j\}*} \boldsymbol g_j ,
\end{equation*}
with
\begin{equation*}
\sum_{j=0}^d \| \boldsymbol g_j\|_{\infty} \leq C\| f\|_{{\rm BMO}(\mathbf{X})},
\end{equation*}
\begin{equation*}\sum_{j=0}^d \| \boldsymbol g_j\|_{L^2(dw)}<\infty.
\end{equation*}
\end{theorem}

Then Theorem~\ref{teo:main_1} follows from Theorem~\ref{UchiyamaDecomposition0} by a duality argument. 

It is worth to emphasize that the ideas of Uchiyama were adapted in Christ-Geller \cite{CHG} to characterize Hardy spaces on homogeneous nilpotent Lie groups  by relevant Riesz transforms  and then  extended, by transference methods,  to  Hardy spaces associated with some Grushin operators \cite{JFAA}.  The present paper owes a lot to~\cite{CHG}. 

{One of the key elements of the proof of Theorem \ref{UchiyamaDecomposition0} is the use of the methods developed by the authors  in~\cite{DH-heat}, \cite{DzHe-JFA},  and~\cite{DH-JFAA}, where, among other results, upper and lower estimates for the Dunkl heat kernel  were derived, as well as estimates for Dunkl translations of certain non-radial functions. We utilize the methods  for obtaining regularity and bounds  for the Dunkl translations of the generalized  convolutions of non-radial kernels. These estimates are expressed by means of volumes of balls, the Euclidean distances and distances of orbits.}

For relations concerning the dual spaces to the Hardy spaces $H^1_{\rm Dunkl}$ with the Dunkl-Riesz transforms, Dunkl-Poisson integrals, and Carleson measures we refer the reader to~\cite{HLLW,Han_et_al,JL}.

{
The condition \eqref{eq:triangle_new} is also necessary for the system of Fourier multiplier operators to characterize the classical Hardy space $H^1(\mathbb R^N)$. A proof of this, which uses the product character of the classical Fourier transform, was given by Janson \cite{Janson}. An interesting question is whether this is a necessary condition in the Dunkl context.   } 

\section{Preliminaries}

In this section we present basic facts concerning the theory of the Dunkl operators.   For more details we refer the reader to~\cite{Dunkl},~\cite{Roesle99},~\cite{Roesler3}, and~\cite{Roesler-Voit}. 

We consider the Euclidean space $\mathbb R^N$ with the scalar product $\langle\mathbf x,\mathbf y\rangle=\sum_{j=1}^N x_jy_j
$, where $\mathbf x=(x_1,...,x_N)$, $\mathbf y=(y_1,...,y_N)$, and the norm $\| \mathbf x\|^2=\langle \mathbf x,\mathbf x\rangle$. For a nonzero vector $\alpha\in\mathbb R^N$,  the reflection $\sigma_\alpha$ with respect to the hyperplane $\alpha^\perp$ orthogonal to $\alpha$ is given by
\begin{align*}
\sigma_\alpha (\mathbf x)=\mathbf x-2\frac{\langle \mathbf x,\alpha\rangle}{\| \alpha\| ^2}\alpha.\index[Other]{$\sigma_\alpha$@$\sigma_\alpha$}
\end{align*}
In this paper we fix a normalized root system in $\mathbb R^N$, that is, a finite set  $R\subset \mathbb R^N\setminus\{0\}$ such that   $\sigma_\alpha (R)=R$ and $\|\alpha\|=\sqrt{2}$ for all $\alpha\in R$. The finite group $G$\index[Other]{$G$@$G$ - reflection group} generated by the reflections $\sigma_\alpha$, $\alpha \in R$,  is called the {\it Weyl group} ({\it reflection group}) of the root system. A~{\textit{multiplicity function}} is a $G$-invariant function $k:R\to\mathbb C$\index[Other]{$k$@$k=k(\alpha)$ - multiplicity function} which will be fixed and $\geq 0$  throughout this paper. We denote by $\mathbf N=N+\sum_{\alpha \in R} k(\alpha)$ \index[Other]{$\mathb N$@$\mathbf N$} the homogeneous dimension of the system. Clearly, 

\begin{align*}
\int_{\mathbb R^N} f(\mathbf x)\, dw(\mathbf x)=\int_{\mathbb R^N} t^{-\mathbf N} f(\mathbf x\slash t)\, dw(\mathbf x)\ \ \text{for all} \ f\in L^1(dw)  \   \text{\rm and} \  t>0,
\end{align*}
\begin{equation}\label{eq:scaling_ball} w(B(t\mathbf x, tr))=t^{\mathbf N}w(B(\mathbf x,r)) \ \ \text{\rm for all } \mathbf x\in\mathbb R^N, \ t,r>0,  
\end{equation}
 where the measure $w$ is defined in \eqref{eq:measure} and $B(\mathbf x, r)=\{\mathbf y\in\mathbb R^N: \|\mathbf y-\mathbf x\|<r\}$. Further,  there is a constant $C>0$ such that 
\begin{align*} C^{-1}w(B(\mathbf x,r))\leq  r^{N}\prod_{\alpha \in R} (|\langle \mathbf x,\alpha\rangle |+r)^{k(\alpha)}\leq C w(B(\mathbf x,r)).
\end{align*}
Consequently,  there exists a constant $C\ge1$ such that,
for every $\mathbf{x}\in\mathbb{R}^N$ and for all $r_2\ge r_1>0$,
\begin{equation}\label{eq:growth}
C^{-1}\Big(\frac{r_2}{r_1}\Big)^{N}\leq\frac{{w}(B(\mathbf{x},r_2))}{{w}(B(\mathbf{x},r_1))}\leq C \Big(\frac{r_2}{r_1}\Big)^{\mathbf{N}},
\end{equation}
so $dw(\mathbf x)$ is doubling, that is, there is a constant $C>0$ such that
\begin{equation}\label{eq:doubling} w(B(\mathbf x,2r))\leq C w(B(\mathbf x,r)) \ \ \text{ for all } \mathbf x\in\mathbb R^N, \ r>0.
\end{equation}
Thus $\mathbf X=(\mathbb R^N, \|\mathbf x-\mathbf y\|, dw)$ is a space of homogeneous type in the sense of Coifman-Weiss \cite{CW}. 

For $\xi \in \mathbb{R}^N$, the {\it Dunkl operators} $T_\xi$  are the following $k$-deformations of the directional derivatives $\partial_\xi$ by a  difference operator:
\begin{align*}
     T_\xi f(\mathbf x)= \partial_\xi f(\mathbf x) + \sum_{\alpha\in R} \frac{k(\alpha)}{2}\langle\alpha ,\xi\rangle\frac{f(\mathbf x)-f(\sigma_\alpha \mathbf x)}{\langle \alpha,\mathbf x\rangle}.\index[Other]{$T_{\xi}$@$T_{\xi}$ - Dunkl operator}
\end{align*}
The Dunkl operators $T_{\xi}$, which were introduced in~\cite{Dunkl}, commute and are skew-symmetric with respect to the $G$-invariant measure $dw$.
 Suppose that $\xi\ne 0$, $f,g \in C^1(\mathbb{R}^N)$ and $g$ is $G$-invariant. Then the following Leibniz rule can be confirmed by a direct calculation:
\begin{align*}
T_{\xi}(f g)=f(T_\xi g)+g(T_\xi f). 
\end{align*}
For fixed $\mathbf y\in\mathbb R^N$ the {\it Dunkl kernel} $E(\mathbf x,\mathbf y)$ is a unique analytic solution to the system
\begin{equation*}
    T_\xi f=\langle \xi,\mathbf y\rangle f, \ \ f(0)=1.
\end{equation*}
The function $E(\mathbf x ,\mathbf y)$, which generalizes the exponential  function $e^{\langle \mathbf x,\mathbf y\rangle}$, has a unique extension to a holomorphic function on $\mathbb C^N\times \mathbb C^N$. Moreover, 
\begin{equation}\label{eq:conjE} E(\mathbf{x},\mathbf{y})=E(\mathbf{y},\mathbf{x}) \text{ and }\overline{E(\mathbf{x},\mathbf{y})}=E(\bar{\mathbf{x}},\bar{\mathbf{y}}) \text{ for all }\mathbf{x},\mathbf{y} \in \mathbb{C}^N.
 \end{equation}

Let $\{e_j\}_{1 \leq j \leq N}$ denote the canonical orthonormal basis in $\mathbb R^N$ and let $T_j=T_{e_j}$\index[Other]{$T_{j}$@$T_{j}$ - Dunkl operator}. For multi-index $I=(\beta_1,\beta_2,\ldots,\beta_N)  \in\mathbb N_0^N$, we set 
$$ |I|=\beta_1+\beta_2 +\ldots +\beta_N,\index[Other]{$|I|$@$|I|$}$$
$$\partial^{I}=\partial_1^{\beta_1} \circ \partial_2^{\beta_2}\circ \ldots \circ \partial_N^{\beta_N},\index[Other]{$\partial^{I}$@$\partial^{I}$}$$
$$ T^I = T_1^{\beta_1}\circ T_2^{\beta_2}\circ \ldots \circ T_N^{\beta_N}.\index[Other]{$T^{I}$@$T^{I}$}$$

The \textit{Dunkl  (Fourier-Dunkl) transform}  
  \begin{equation}\label{eq:Dunkl_transform}
  \mathcal F f(\xi)=c_k^{-1}\int_{\mathbb R^N} E(-i\xi, \mathbf x)f(\mathbf x)\, dw(\mathbf x),\index[Other]{$\mathcal{F}$@$\mathcal{F}$ - Dunkl transform}
  \end{equation}
  where
  $$c_k=\int_{\mathbb{R}^N}e^{-\frac{\|\mathbf{x}\|^2}{2}}\,dw(\mathbf{x})>0,$$
   originally defined for $f\in L^1(dw)$, is an isometry on $L^2(dw)$, i.e.,
   \begin{equation}\label{eq:Plancherel}
       \|f\|_{L^2(dw)}=\|\mathcal{F}f\|_{L^2(dw)} \text{ for all }f \in L^2(dw),
   \end{equation}
and preserves the Schwartz class of functions $\mathcal S(\mathbb R^N)$ (\cite{D5}, see also \cite{dJ1}). {It generalizes the Fourier transform.}  Its inverse $\mathcal F^{-1}$ has the form
  \begin{align}\label{eq:inverse} \mathcal F^{-1} g(\mathbf{x})=c_k^{-1}\int_{\mathbb R^N} E(i\xi, \mathbf x)g(\xi)\, dw(\xi).\index[Other]{$\mathcal{F}^{-1}$@$\mathcal{F}^{-1}$ - inverse Dunkl transform}
  \end{align}
{  The inversion formula
\begin{align*}
    f(\mathbf{x})=c_k^{-1}\int_{\mathbb{R}^N} \mathcal F f(\xi)E(i\xi,\mathbf{x})\, dw(\xi)
\end{align*}
  remains valid if $f,\, \mathcal Ff\in L^1(dw)$ (see~\cite[Theorem 4.2]{dJ1}). 
  }

 The {\it generalized (Dunkl) translation\/} $\tau_{\mathbf{x}}f$\index[Other]{$\tau_{\mathbf{x}}$@$\tau_{\mathbf{x}}$ - generalized translation} of a function $f\in\mathcal{S}(\mathbb{R}^N)$ by $\mathbf{x}\in\mathbb{R}^N$ is defined by
\begin{align}\label{eq:transl}
\tau_{\mathbf{x}} f(\mathbf{y})=c_k^{-1} \int_{\mathbb{R}^N}{E}(i\xi,\mathbf{x})\,{E}(i\xi,\mathbf{y})\,\mathcal{F}f(\xi)\,{dw}(\xi).
\end{align}
  It is a contraction on $L^2(dw)$, however it is an open  problem  if the Dunkl translations are bounded operators on $L^p(dw)$ for $p\ne 2$.

{{ The following specific formula  for the Dunkl translations of (reasonable) radial functions $f({\mathbf{x}})=\tilde{f}({\|\mathbf{x}\|})$ was obtained by R\"osler \cite{Roesler2003}:
\begin{equation}\label{eq:translation-radial}
\tau_{\mathbf{x}}f(-\mathbf{y})=\int_{\mathbb{R}^N}{(\tilde{f}\circ A)}(\mathbf{x},\mathbf{y},\eta)\,d\mu_{\mathbf{x}}(\eta)\text{ for }\mathbf{x},\mathbf{y}\in\mathbb{R}^N.
\end{equation}
Here
\begin{equation*}
A(\mathbf{x},\mathbf{y},\eta)=\sqrt{{\|}\mathbf{x}{\|}^2+{\|}\mathbf{y}{\|}^2-2\langle \mathbf{y},\eta\rangle}=\sqrt{{\|}\mathbf{x}{\|}^2-{\|}\eta{\|}^2+{\|}\mathbf{y}-\eta{\|}^2}
\end{equation*}
and $\mu_{\mathbf{x}}$ is a probability measure, 
which is supported in the set $\operatorname{conv}\mathcal{O}(\mathbf{x})$,  where $\mathcal O(\mathbf x) =\{\sigma(\mathbf x): \sigma \in G\}$ is the orbit of $\mathbf x$. 
}}
{

Formula~\eqref{eq:translation-radial} implies that for all radial $f \in L^1(dw)$ and $\mathbf{x} \in \mathbb{R}^N$ we have
\begin{align*}
    \|\tau_{\mathbf{x}}f\|_{L^1(dw)} \leq \|f\|_{L^1(dw)}. 
\end{align*}
Let 
$$d(\mathbf x,\mathbf y)=\min_{\sigma\in G}\| \sigma(\mathbf x)-\mathbf y\|\index[Other]{$d(\mathbf x,\mathbf y)$@$d(\mathbf x,\mathbf y)$ - distance of orbits}$$
be the distance of the orbit of $\mathbf x$ to the orbit of $\mathbf y$.
It was proved in  \cite[Theorem 1.7]{DzHe-JFA} that if $f\in L^2(dw)$ and 
  $\text{\rm supp}\, f\subseteq B(0, r)$ , then 
\begin{equation}\label{eq:supp_transl}
    \tau_{\mathbf x}f(-\mathbf y)=0 \quad \text{\rm for } d(\mathbf x,\mathbf y)>r.
\end{equation}
 }

  {The \textit{generalized (Dunkl) convolution\/} $f*g$ \index[Other]{$*$@$*$ - generalized convolution} of two reasonable functions $f$ and $g$ (for instance Schwartz functions) is defined by
$$
(f*g)(\mathbf{x})=c_k\,\mathcal{F}^{-1}[(\mathcal{F}f)(\mathcal{F}g)](\mathbf{x})=\int_{\mathbb{R}^N}(\mathcal{F}f)(\xi)\,(\mathcal{F}g)(\xi)\,E(\mathbf{x},i\xi)\,dw(\xi) \text{ for }\mathbf{x}\in\mathbb{R}^N,
$$
or, equivalently, by}
\begin{align*}
  {(f*g)(\mathbf{x})=\int_{\mathbb{R}^N}f(\mathbf{y})\,\tau_{\mathbf{x}}g(-\mathbf{y})\,{dw}(\mathbf{y})=\int_{\mathbb R^N} f(\mathbf y)g(\mathbf x,\mathbf y) \,dw(\mathbf{y}) \text{ for all } \mathbf{x}\in\mathbb{R}^N},  
\end{align*}
where, here and subsequently, $g(\mathbf x,\mathbf y)=\tau_{\mathbf x}g(-\mathbf y)$. 

{We shall use the notation $f^{*2}=f*f$ and, inductively, $f^{*{(n+1)}}= f^{*n}*f$.\index[Other]{$f^{*n}$@$f^{*n}$}} 

Unless otherwise defined, for a function $f$, we denote: 
$$f(\mathbf x,\mathbf y)=\tau_{\mathbf{x}}f(-\mathbf{y}),\quad  f_t(\mathbf x)=t^{-\mathbf N} f(\mathbf x/t), \quad f_t(\mathbf x,\mathbf y)=\tau_{\mathbf x}(f_t)(-\mathbf y)\index[Other]{$f_t$@$f_t$}.$$ 

The {\it Dunkl Laplacian} associated with $R$ and $k$  is the differential-difference operator $\Delta_k=\sum_{j=1}^N T_{j}^2$,\index[Other]{$ \Delta_k$@$ \Delta_k$ - Dunkl Laplacian} which  acts on $C^2(\mathbb{R}^N)$-functions by

\begin{align*}
    \Delta_k f(\mathbf x)=\Delta_{\rm eucl} f(\mathbf x)+\sum_{\alpha\in R} k(\alpha) \delta_\alpha f(\mathbf x),
\end{align*}
\begin{align*}
    \delta_\alpha f(\mathbf x)=\frac{\partial_\alpha f(\mathbf x)}{\langle \alpha , \mathbf x\rangle} - \frac{\|\alpha\|^2}{2} \frac{f(\mathbf x)-f(\sigma_\alpha \mathbf x)}{\langle \alpha, \mathbf x\rangle^2}.
\end{align*}
Obviously, $\mathcal F(\Delta_k f)(\xi)=-\| \xi\|^2\mathcal Ff(\xi)$. The operator $\Delta_k$ is essentially self-adjoint on $L^2(dw)$ (see for instance \cite[Theorem\;3.1]{AH}) and generates the semigroup $e^{t\Delta_k}$  of linear self-adjoint contractions on $L^2(dw)$. The semigroup has the form
  \begin{align*}
  e^{t\Delta_k} f(\mathbf x)=\mathcal F^{-1}(e^{-t\|\xi\|^2}\mathcal Ff(\xi))(\mathbf x)=\int_{\mathbb R^N} h_t(\mathbf x,\mathbf y)f(\mathbf y)\, dw(\mathbf y),
  \end{align*}
  where the heat kernel 
  \begin{equation*}
      h_t(\mathbf x,\mathbf y)=\tau_{\mathbf x}h_t(-\mathbf y), \ \ h_t(\mathbf x)={ c_k^{-1}}\mathcal F^{-1} (e^{-t\|\xi\|^2})(\mathbf x)=c_k^{-1} (2t)^{-\mathbf N\slash 2}e^{-\| \mathbf x\|^2\slash (4t)}\index[Other]{$h_t(\mathbf{x}),\mathbf{y}$@$h_t(\mathbf{x},\mathbf{y})$ - Dunkl heat kernel}
  \end{equation*}
  is a $C^\infty$-function of all variables $\mathbf x,\mathbf y \in \mathbb{R}^N$, $t>0$, and satisfies \begin{align*} 0<h_t(\mathbf x,\mathbf y)=h_t(\mathbf y,\mathbf x),
  \end{align*}
 \begin{align*} \int_{\mathbb R^N} h_t(\mathbf x,\mathbf y)\, dw(\mathbf y)=1.
 \end{align*}
Let us note that for $f\in L^1(dw)$ such that $\mathcal Ff\in L^1(dw)$, the translation $\tau_{\mathbf{x}} f:\mathbb{R}^N \longmapsto \mathbb{C}$ is a bounded function and 
\begin{equation*}\label{eq:almost_preserv}
\begin{split}
   & \lim_{t\to 0} \int_{\mathbb{R}^N} \tau_{\mathbf x}f(\mathbf y)e^{-t\|\mathbf y\|^2}\, dw({\mathbf{y}}) 
    =\lim_{t\to 0} c_k^{-1} \int_{\mathbb R^N}\int_{\mathbb R^N}E(i\xi,\mathbf x)E(i\xi,\mathbf y)\mathcal Ff(\xi)\, e^{-t\|\mathbf y\|^2} dw(\xi)\, dw(\mathbf y)\\
    &= \lim_{t\to 0} c_k\int_{\mathbb R^N} E(i\xi, \mathbf x)\mathcal Ff(\xi)h_t(\xi)\, dw(\xi) = c_kE(0,\mathbf x)\mathcal Ff(0)
    =\int_{\mathbb{R}^N} f(\mathbf y)\, dw(\mathbf y).
\end{split} \end{equation*}
Hence, {if $f,\mathcal Ff, \tau_{\mathbf x}f\in L^1(dw)$, then the integral of $\tau_{\mathbf x} f$ is preserved, that is, 
\begin{equation}
    \label{eq:int-preserv} 
    \int_{\mathbb R^N} \tau_{\mathbf x}f(\mathbf y)\, dw(\mathbf y)=\int_{\mathbb R^N} f(\mathbf y)\, dw(\mathbf y). 
\end{equation}
}

Set
$$V(\mathbf x,\mathbf y,t):=\max (w(B(\mathbf x,t)),w(B(\mathbf y, t))).$$
{Observe that for any $\varepsilon >0$ there is a constant $C>0$ such that for any cube $Q\subset \mathbb R^N$ 
\begin{equation}\label{eq:finite_integral}
\frac{1}{w(Q)} \int_{\mathbb{R}^N} \Big(1+\frac{d(\mathbf x,\mathbf z_Q)}{\ell(Q)}\Big)^{-\mathbf N-\varepsilon} \, dw(\mathbf x)\leq C,    
\end{equation}
where $\mathbf z_Q$\index[Other]{$\mathbf z_Q$@$\mathbf z_Q$}\index[Other]{$ \ell(Q)$@$ \ell(Q)$} denotes the center of $Q$ and $\ell(Q)$ is its side-length. }

In Theorem \ref{theorem:heat} we state  estimates for $h_t(\mathbf x,\mathbf y)$.  The estimate \eqref{eq:Gauss} was announced by W. Hebisch (with an outline of a proof which used a Poincar\'e inequality).  An elementary and complete  two-step proof of \eqref{eq:Gauss} and \eqref{eq:Gauss-Lipschitz}, which is based on \eqref{eq:translation-radial},  can be found in ~\cite[Theorem 4.1]{ADH-JFAA} and  ~\cite[Theorem 3.1]{DH-Studia}. For more precise upper and lower bounds for $h_t(\mathbf x,\mathbf y)$ we refer to \cite{DH-heat}. 
\begin{theorem}\label{theorem:heat}
There are constants $C,c>0$ such that for all $\mathbf{x},\mathbf{y} \in \mathbb{R}^N$ and $t>0$ we have
\begin{equation}\label{eq:Gauss}
h_t(\mathbf{x},\mathbf{y}) \leq C\Big(1+\frac{\|\mathbf{x}-\mathbf{y}\|}{\sqrt{t}}\Big)^{-2}\,V(\mathbf{x},\mathbf{y},\!\sqrt{t\,})^{-1}\,e^{-\hspace{.25mm}c\hspace{.5mm}d(\mathbf{x},\mathbf{y})^2\slash t}.
\end{equation}
Moreover, if $\|\mathbf{y}-\mathbf{y}'\| \leq \sqrt{t}$, then
\begin{equation}\label{eq:Gauss-Lipschitz}
|h_t(\mathbf{x},\mathbf{y})-h_t(\mathbf{x},\mathbf{y}')| \leq C\frac{\|\mathbf{y}-\mathbf{y}'\|}{\sqrt{t}}\Big(1+\frac{\|\mathbf{x}-\mathbf{y}\|}{\sqrt{t}}\Big)^{-2}\,V(\mathbf{x},\mathbf{y},\!\sqrt{t\,})^{-1}\,e^{-\hspace{.25mm}c\hspace{.5mm}d(\mathbf{x},\mathbf{y})^2\slash t}.
\end{equation}

\end{theorem}

\section{Homogeneous kernels and Dunkl multiplier operators} 

\subsection{Regular kernels} 
We say that a tempered distribution $\boldsymbol S$ on $ \mathcal S(\mathbb R^N) $ is a regular kernel\index[Other]{regular kernel@regular kernel}  of order zero if $\boldsymbol S$ coincides away from the  origin with  
a function of the form $S(\mathbf x) w(\mathbf x)$, where  
$S\in C^\infty (\mathbb R^N\setminus \{0\})$ and satisfies 
$$ \langle \boldsymbol S,f^{s}\rangle = \langle \boldsymbol S,f\rangle, \quad s>0, \quad f\in\mathcal S(\mathbb R^N),  $$
where $f^s(\mathbf x)=f(s\mathbf x)$. Clearly, 
 $$ S(r \mathbf x)=r^{-\mathbf N} S(\mathbf x) \quad \text{for  } \mathbf x\ne 0, \; r>0.$$  
Any tempered distribution $\boldsymbol S$, which is a regular kernel of order zero is of the form 
 \begin{equation}\label{eq:S_action}
     \begin{split}
     \langle \boldsymbol S,f\rangle &= \boldsymbol c_1{ \langle \delta_0, f\rangle } + \lim_{\varepsilon \to 0}\int_{\varepsilon}^\infty \int_{\| \bar {\mathbf{x}}\|=1} S(\bar{\mathbf x}) f(r\bar{\mathbf x})w(\bar{\mathbf x})d\bar{\mathbf  x}\frac{dr}{r}    \\
     & =\boldsymbol c_1f(0)+\lim_{\varepsilon\to 0}\int_{\| \mathbf x\|>\varepsilon} S(\mathbf x)f(\mathbf x)\;dw(\mathbf x),
     \end{split}
 \end{equation} 
 where $\int_{\| \bar{\mathbf  x}\|=1} S(\bar{\mathbf  x})\, w(\bar{\mathbf x})\, d
 \bar{\mathbf x}=0$ and $d\bar{\mathbf x}$ is the Lebesgue measure on the unit sphere. For the proof, see Christ \cite[Lemma 2.4]{Christ-inversion}.  

Set \index[Other]{$\reallywidecheck f$@$\reallywidecheck f(\mathbf{x})=f(-\mathbf{x})$}
$$ \langle \reallywidecheck{\boldsymbol S},f\rangle =\langle \boldsymbol S,\reallywidecheck f\rangle, \quad \reallywidecheck f(\mathbf x)=f(-\mathbf x).$$ 
 
Let $\varphi$ be a $C^\infty$ non-negative radial function supported in the set $\{\mathbf{x}:1/4<\|\mathbf x\|<1\}$ such that $\sum_{j\in\mathbb Z} \varphi (2^j \mathbf x)=1$ for $\mathbf x\ne 0$. 
 Set 
 $$ S_j(\mathbf x)=\varphi(2^{-j} \mathbf x)S(\mathbf x).$$
 Then $S_j(\mathbf x)=2^{-j\mathbf N} S_0(2^{-j}\mathbf x)$ and 
 \begin{equation}\label{eq:distr_S}
      \langle \boldsymbol S,f\rangle = \boldsymbol c_1f(0)+\sum_{j=-\infty}^\infty \int_{\mathbb{R}^N} S_j(\mathbf x)f(\mathbf x)\, dw(\mathbf x).
 \end{equation}

The density with respect to the measure $dw$ of the Dunkl  transform $\mathcal F\boldsymbol S$ of $\boldsymbol S$ coincides with a homogeneous of degree zero, $C^\infty$  away from the origin function $\theta$. Indeed, 

\begin{equation*}
\begin{split}
    \langle \mathcal F\boldsymbol S,f\rangle &= \langle \boldsymbol S,\mathcal F f\rangle = \boldsymbol c_1\mathcal F f(0)+\sum_{j=-\infty}^\infty \int_{\mathbb{R}^N} S_j({\mathbf{x}})\mathcal Ff({\mathbf{x}})\, dw({\mathbf{x}})\\
    &=\boldsymbol c_1\int_{\mathbb{R}^N} f(\xi)\, dw(\xi)+\sum_{j=-\infty}^\infty \int_{\mathbb{R}^N} (\mathcal FS_0)(2^j\xi)f(\xi)\, dw(\xi).\\
\end{split}\end{equation*}
Observe that $\mathcal F S_0$ is a Schwartz class function such that $\mathcal FS_0(0)=0$. Hence $\mathcal |\mathcal FS_0(\xi )|\leq C \|\xi \|$ for $\|\xi \|\leq 1.$ Consequently, it is not difficult to see that $\boldsymbol c_1+\sum_{j=-\infty}^\infty \mathcal FS_0(2^j\xi)$ converges absolutely  to a $C^\infty$ function $\theta (\xi)$ for $\xi \ne 0$. Moreover,  $|\partial^J \theta (\xi)|\leq C_J \| S_0\|_{C^{|J|+\mathbf N}}$ {for all $J \in \mathbb{N}_0^N$ and $\xi \in \mathbb{R}^N$, $\|\xi\|=1$}. Further, to see that $\theta(\xi)$ is homogeneous of degree zero, we write 
\begin{equation*}
\begin{split}
    \int_{\mathbb{R}^N} \theta (t\xi) f(\xi)\, dw(\xi) &= \int_{\mathbb{R}^N} \theta (\xi) f_t(\xi)\, dw(\xi)=\langle \boldsymbol  S, \mathcal Ff_t\rangle \\&=\langle \boldsymbol S, (\mathcal Ff)^t\rangle =\langle \boldsymbol S,\mathcal Ff\rangle =\int_{\mathbb{R}^N} \theta(\xi) f(\xi)\, dw(\xi). 
    \end{split}
\end{equation*}

 Conversely, every $C^\infty$ away from the origin homogeneous of degree zero  function $\theta(\xi)$ is the Dunkl transform of the  distribution 
 $$ \langle \boldsymbol S, f\rangle = \int_{\mathbb{R}^N} \theta (\xi) \mathcal F^{-1} f(\xi)\, dw(\xi) $$ which is a regular kernel of order 0, which means that $\boldsymbol S$ is of the form \eqref{eq:S_action}. The corresponding homogeneous of degree $-\mathbf N$  function $S({\mathbf{x}})$ is defined by 
 $$ S(\mathbf x)=\sum_{j=-\infty}^\infty 2^{j\mathbf N} \mathcal F^{-1} (\theta\cdot \varphi)(2^j \mathbf x), \quad \mathbf x\ne 0.$$ Moreover, for every $s\geq 0$,
 $$|\boldsymbol c_1|+\sup_{|I|\leq s, \|\mathbf x\|=1}|\partial^I S(\mathbf x)|\leq C_s\sup_{|J|\leq s+\mathbf N, \|\xi\|=1}|\partial^J \theta (\xi)|.$$

 With a regular kernel $\boldsymbol S$ of order 0, we associate the convolution operator (initially defined for $f\in\mathcal S(\mathbb R^N)$)  
 $$\mathbf Sf({\mathbf{x}})=\boldsymbol S*f(\mathbf x)=\langle \reallywidecheck {\boldsymbol S},\tau_{\mathbf x} f\rangle=\langle \boldsymbol S, \reallywidecheck{\tau_{\mathbf x}f}\rangle.\index[Other]{$\mathbf Sf=\boldsymbol S*f$@$\mathbf Sf=\boldsymbol S*f$}$$
 Clearly, 
 $$\mathbf Sf(\mathbf x) =\int_{\mathbb{R}^N} \theta (\xi) \mathcal F^{-1}( \reallywidecheck{\tau_{\mathbf x}f})(\xi)\, dw(\xi)=\int_{\mathbb{R}^N} \theta(\xi) \mathcal Ff(\xi)E(i\xi,\mathbf x)\, dw(\xi)=\mathcal F^{-1} (\theta(\cdot)\mathcal Ff(\cdot))(\mathbf x),$$
thus $\mathbf S$ is a Dunkl multiplier operator, which by the Plancherel's equality~\eqref{eq:Plancherel} is bounded on $L^2(dw)$. Theorem 1 of~\cite{DzHe-JFA} asserts that it is bounded on $L^p(dw)$, weak-type (1.1), and bounded on the Hardy space $H^1_{\rm Dunkl}$. Further, 
 $$\mathbf Sf(\mathbf x)= \boldsymbol c_1f(\mathbf x) + \sum_{j=-\infty}^\infty S_j*f(\mathbf x)=\boldsymbol c_1 f(\mathbf x)+ \sum_{j=-\infty}^\infty  \int_{\mathbb{R}^N} S_j(\mathbf x,\mathbf y)f(\mathbf y)\, dw(\mathbf y),$$
 where $S_j(\mathbf x,\mathbf y)=\tau_{\mathbf x}S_j(-\mathbf y)$. 

For a regular kernel $\boldsymbol S$ of order zero, we define the action of the convolution operator $\mathbf S$ on $L^1(dw)$-functions in the distributional sense, that is,
\begin{equation*}
    \langle \mathbf S f, \varphi\rangle =\int_{\mathbb{R}^N} f(\mathbf x)(\reallywidecheck{\boldsymbol S}* \varphi) (\mathbf x)\, dw(\mathbf x)\quad \text{for all} \ \varphi \in \mathcal S(\mathbb R^N). 
\end{equation*}
This is well defined, since 
$$\reallywidecheck{\boldsymbol S}*\varphi (\mathbf x)=\int_{\mathbb{R}^N} \theta(\xi) \mathcal F^{-1}\varphi(\xi)E(-i\mathbf x,\xi)\, dw(\xi)$$ 
is a bounded function of $\mathbf x$ which belongs to $L^2(dw)$ as well. Further, for a function $f\in L^1(dw)$, the distribution $\mathbf Sf$ belongs to  $L^1(dw)$ if there is $g\in L^1(dw)$   such that
\begin{equation*}
   \langle \mathbf Sf,\varphi \rangle =\int_{\mathbb{R}^N} g(\mathbf{x})\varphi (\mathbf{x})\, dw(\mathbf{x}) \quad \text{for all} \ \varphi \in \mathcal S(\mathbb R^N). 
\end{equation*}
Then we simply write $\mathbf Sf=g$. It is not difficult to prove that in this case 
\begin{equation*}
   \mathbf S(f*h_t)=(\mathbf Sf)*h_t,
\end{equation*} where the action on the right-hand side can be understood  in  both: the $L^2(dw)$ and  distributional sense. 
{
It turns out that the kernels $S_j(\mathbf x,\mathbf y)$ satisfy the following estimates (see Theorem \ref{teo:translations_kernels} below) which are crucial in our investigations: for  all multi-indexes $I$ and $J$ and any even integer $\kappa>|I|+|J|+\mathbf N$ there is a constant $C_{I,J}>0$ such that 
 $$ |T^I_{\mathbf x}T^J_{\mathbf y} S_j(\mathbf x,\mathbf y)|\leq C_{I,J}\| S_0\|_{C^{\kappa }} 2^{j(|I|+|J|)}w(B(\mathbf x, 2^{j}))^{-1} \Big(1+\frac{\| \mathbf x-\mathbf y\|}{2^{j}}\Big)^{-1}\chi_{[0,2^{j}]}(d(\mathbf x,\mathbf y)).  $$
 Consequently, summing up the estimates above, we obtain the following bounds for the  associated kernel $S(\mathbf x,\mathbf y)=\sum_{j} S_j(\mathbf x,\mathbf y)$: 
\begin{equation}\label{eq:S-bound}
    |T_{\mathbf x}^IT_{\mathbf y}^JS(\mathbf x,\mathbf y)|\leq  C_{I,J}\|S_0\|_{C^{\kappa}}  \frac{d(\mathbf x,\mathbf y)}{\|\mathbf x-\mathbf y\|} d(\mathbf x,\mathbf y)^{-|I|-|J|} w(B(\mathbf x,d(\mathbf x,\mathbf y)))^{-1}. 
\end{equation}
 }

Let us also note  that for regular kernels $\boldsymbol S$ and $\boldsymbol  Z$  of order zero,    the corresponding operators $\mathbf S$ and $\mathbf Z$ commute and its composition is represented as a convolution with a  regular kernel of order zero. 
Moreover, the adjoint operator $\mathbf S^*$ to $\mathbf S$  corresponds to the 
 Dunkl multiplier $\overline{\theta (\xi)}$.

\subsection{Condition \texorpdfstring{$(\star)$}{(star)}  }

 Let { $\overrightarrow S=(\boldsymbol S^{\{0\}},\boldsymbol S^{\{1\}},\boldsymbol S^{\{2\}},...,\boldsymbol S^{\{d\}})$}  be a system of regular kernels of order $0$
on $\mathbb R^N$. We denote: 
\begin{align*}
    \mathbf S^{\{j\}} f=\boldsymbol S^{\{j\}}*f, \ \ \mathcal F \boldsymbol S^{\{j\}}(\xi)=\theta_j(\xi).
\end{align*}
Then $\mathcal F(\mathbf S^{\{j\}*} f)(\xi)=\overline{\theta (\xi)}\mathcal Ff(\xi)$. 

{We abuse a little our terminology and  say that  { $\overrightarrow S=(\boldsymbol S^{\{0\}},\boldsymbol S^{\{1\}},\boldsymbol S^{\{2\}},...,\boldsymbol S^{\{d\}})$} satisfies ~\eqref{eq:triangle_new} if for the system of their Dunkl transforms { $(\theta_0(\xi), \theta_1(\xi),\theta_2(\xi),\dots,\theta_d(\xi))$} the  condition ~\eqref{eq:triangle_new} holds. }

Let us note that if  a system $\overrightarrow S$ of regular kernels of order zero fulfills the condition ~\eqref{eq:triangle}, then so $\overrightarrow{S^*}$ does, that is,   
{
\begin{displaymath}\label{eq:triangle_bar}\tag{$\bar\star$}
\text{\rm rank} \, 
\left( \begin{array}{ccccc}
\overline{\theta_0(\xi)} & \overline{\theta_1(\xi)} & \overline{\theta_2(\xi)}& \ldots &\overline{\theta_d(\xi)} \\
\overline{\theta_0(-\xi)} & \overline{\theta_1(-\xi)} & \overline{\theta_2(-\xi)}& \ldots &\overline{\theta_d(-\xi)}
\end{array} \right) =2.
\end{displaymath}
}

{We say that a regular kernel  $\boldsymbol S$ is real if and only if $\boldsymbol c_1 \in \mathbb{R}$ and $S(\mathbf x)$ is real valued  (see~\eqref{eq:S_action}), or, equivalently, thanks to~\eqref{eq:conjE},   $\theta(\xi)=\overline{\theta (-\xi)}$, where  $\theta(\xi)=\mathcal F\boldsymbol S(\xi)$. 
 Similarly, $\boldsymbol S$  is purely imaginary if and only if  $\theta(\xi)=-\overline{\theta (-\xi)}$, where $\theta (\xi)=\mathcal F \boldsymbol S(\xi)$.}

The following proposition proved in  Uchiyama \cite{Uchiyama} is crucial for the proofs of our results.

\begin{proposition}[{\cite[Lemma 2.2]{Uchiyama}}]
\label{propo:real_complex}
    Suppose that a system  {$(\theta_0(\xi), \theta_1(\xi),\theta_2(\xi),\dots, \theta_d(\xi))$ } of complex valued functions on $\mathbb R^N$, which are $C^\infty$ away from the origin and homogeneous of degree zero,  {satisfies \eqref{eq:triangle_new}.}  
    Then for any {$\overrightarrow{\boldsymbol \nu} \in \mathbb{C}^{d+1}$,} $\|\overrightarrow{\boldsymbol \nu}\|=1$, there is a system of homogeneous of degree zero and smooth away from the origin functions 
    $$\mathbb R^N\setminus{\{0\}} \ni \xi \mapsto \Theta_j(\xi,\overrightarrow{\boldsymbol \nu})\in \mathbb C, \quad  j=0, 1,2,...,d,$$ such that for all $\xi \in \mathbb{R}^N\setminus \{0\}$ we have 
\begin{equation}\label{eq:Uchi1}
\sum_{j=0}^d \overline{\theta_j (\xi)}\Theta_j (\xi,\overrightarrow{\boldsymbol \nu})=1,
\end{equation}
\begin{equation}\label{eq:Uchi2} {\rm  Re}\sum_{j=0}^d \overline{\nu_j }\{\Theta_j(\xi,\overrightarrow{\boldsymbol \nu})+\Theta_j(-\xi,\overrightarrow{\boldsymbol \nu})\}= 
    {\rm  Im}\sum_{j=0}^d \overline{\nu_j} \{\Theta_j(\xi,\overrightarrow{\boldsymbol \nu})-\Theta_j(-\xi,\overrightarrow{\boldsymbol \nu})\}= 0, 
\end{equation}
\begin{equation*}
\begin{split}
\Big|\partial_{\xi}^{I}\Theta_j (\xi,\overrightarrow{\boldsymbol\nu}) \Big|\leq C_I \ \ \text{\rm for all } \|\xi\|=1  \ \text{\rm with } C_I>0 \ \text{\rm independent of} \ \overrightarrow{\boldsymbol{\nu}}.
\end{split}
\end{equation*}
\end{proposition}

\subsection{Estimates for translations of  kernels} We finish the section by stating estimates for  the Dunkl translations of functions associated with regular kernels. We want to emphasis the presence of the Euclidean distance in the estimates. It is crucial 
 in  proving the Fefferman-Stein decomposition of compactly supported  ${\rm BMO}(\mathbf X)$-functions (see Theorem \ref{UchiyamaDecomposition0}).

\begin{theorem}({\rm cf.} \cite[Theorem 3.6]{DH-JFAA}). \label{teo:translations_kernels}
     Let $\kappa$ and $m$ be  non-negative  even integers. 

\begin{enumerate}[(a)]
    \item{If $\kappa>\mathbf N+m$, then there is a constant $C_\kappa>0$ such that for all $\varphi\in C^{\kappa}(\mathbb R^N)$ such that $\text{\rm supp}\, \varphi\subseteq B(0,1) $ and all multi-indexes $I$ and $J$ such that $|I|+|J|\leq m$ one has 
     \begin{equation*}
         |T^I_{\mathbf x}T^J_{\mathbf y}\varphi_t(\mathbf x,\mathbf y)|\leq C_{\kappa} t^{-|I|-|J|}\| \varphi\|_{C^{\kappa}(\mathbb R^N)}\Big(1+\frac{\|\mathbf x-\mathbf y\|}{t}\Big)^{-1} w(B(\mathbf x,t))^{-1}\chi_{[0,t]}(d(\mathbf x,\mathbf y))
     \end{equation*}
 for all $\mathbf x,\mathbf y\in \mathbb R^N$.}\label{translations_kernels:a}

\item{ If $\kappa>\mathbf N+1$, then there is a constant $C_\kappa'>0$ such that 
     \begin{equation*}
     \begin{split}
     &    |\varphi_t(\mathbf x,\mathbf y)-\varphi_t(\mathbf x,\mathbf y')|\\
         &\leq C_{\kappa}' \| \varphi\|_{C^{\kappa}(\mathbb R^N)} \frac{\|\mathbf y-\mathbf y'\|}{t+\| \mathbf x-\mathbf y\|}w(B(\mathbf x,t))^{-1/2}(w(B(\mathbf y,t))^{-1/2}+w(B(\mathbf y',t))^{-1/2})
     \end{split}\end{equation*}
     for all $\mathbf x,\mathbf y,\mathbf y'\in\mathbb R^N$ and all $\varphi\in C^{\kappa}(\mathbb R^N)$, $\text{\rm supp}\, \varphi\subseteq B(0,1)$.}\label{translations_kernels:b}

     \item{If $\kappa>\mathbf N+m+1$, then there is a constant $C_{\kappa}''>0$ such that 
      \begin{equation*}
     \begin{split}
     &    |T_{\mathbf y}^I\varphi_t(\mathbf x,\mathbf y)-T^I_{\mathbf{y}}\varphi_t(\mathbf x',\mathbf y)|\\
         &\leq C_{\kappa}'' t^{-|I|} \| \varphi\|_{C^{\kappa}(\mathbb R^N)} \frac{\|\mathbf x-\mathbf x'\|}{t+\| \mathbf x-\mathbf y\|}w(B(\mathbf y,t))^{-1/2}(w(B(\mathbf x,t))^{-1/2}+w(B(\mathbf x',t))^{-1/2})
     \end{split}\end{equation*}
     for all $\mathbf x,\mathbf x',\mathbf y\in\mathbb R^N$,  all multi-indexes $I$, $|I|\leq m$, and all $\varphi\in C^{\kappa}(\mathbb R^N)$, $\text{\rm supp}\, \varphi\subseteq B(0,1)$.}\label{translations_kernels:c}
\end{enumerate}
 \end{theorem}
\begin{proof}
  {  For $t=1$ this is \cite[Theorem 3.6]{DH-JFAA}.  In order to obtain the estimates for $\varphi_t$ we apply  \eqref{eq:scaling_ball} and the  relations 
  $$ T_{\mathbf x}^{I} T_{\mathbf y}^J \varphi_t(\mathbf x,\mathbf y)= (-1)^{|J|}t^{-|I|-|J|-\mathbf N} (T^{I+J}\varphi)\Big(\frac{\mathbf x}{t},\frac{\mathbf y}{t}\Big).$$ 
  }
\end{proof}
 \begin{theorem}\label{teo:bounds_conv}
     Let $\kappa$ and $m$ be non-negative even integers. 
     \begin{enumerate}[(a)]
         \item{If $\kappa >\mathbf N+1$, then there is a constant $C_\kappa>0$ such that for all functions $\psi,\, \phi\in C^{\kappa}(\mathbb R^N)$ supported in the unit ball $B(0,1)$ such that $\int_{\mathbb{R}^N} \phi({\mathbf{x}})\, dw({\mathbf{x}})=0$ and all $0<s\leq t$ one has 
 $$|(\psi_t*\phi_s)(\mathbf x,\mathbf y)|\leq C_\kappa \|\phi\|_{C^{\kappa}(\mathbb R^N)}\|\psi\|_{C^{\kappa}(\mathbb R^N)} \frac{s}{t}w(B(\mathbf x,t))^{-1} \Big(1+\frac{\|\mathbf x-\mathbf y\|}{t}\Big)^{-1} \chi_{[0,2t]}(d(\mathbf x,\mathbf y)).$$}\label{teo_a}
 \item{If $\kappa >\mathbf N+m+1$, then there is a constant $C_\kappa>0$ such that for all functions $\psi,\, \eta, \phi \in C^{\kappa}(\mathbb R^N)$ supported in the unit ball $B(0,1)$ such that $\phi=\Delta_k^{m/2}\eta $ and all $0<s\leq t$ one has 
 $$|(\psi_t*\phi_s)(\mathbf x,\mathbf y)|\leq C_\kappa \|\eta\|_{C^{\kappa}(\mathbb R^N)}\|\psi\|_{C^{\kappa}(\mathbb R^N)}\Big( \frac{s}{t}\Big)^mw(B(\mathbf x,t))^{-1} \Big(1+\frac{\|\mathbf x-\mathbf y\|}{t}\Big)^{-1} \chi_{[0,2t]}(d(\mathbf x,\mathbf y)).$$}\label{teo_b}
 \item{If $\kappa >\mathbf N+1$, then there is a constant $C_\kappa>0$ such that for  all functions $\psi,\, \phi\in C^{\kappa}(\mathbb R^N)$ supported in the unit ball $B(0,1)$ such that $\phi$ is radial and $\int_{\mathbb{R}^N} \psi(\mathbf x)\, dw(\mathbf x)=0$ and all $0<s\leq t$ one has 
 \begin{equation*}
     \begin{split}
        |(\psi_s* & \phi_t*\phi_t)(\mathbf x,\mathbf y)  - (\psi_s*\phi_t*\phi_t)(\mathbf x',\mathbf y)|\\
        & \leq C_\kappa \|\phi\|^2_{C^{\kappa}(\mathbb R^N)}\|\psi\|_{C^{\kappa}(\mathbb R^N)} \frac{s}{t}\frac{\|\mathbf x-\mathbf x'\|}{t}w(B(\mathbf y,t))^{-1} \Big(1+\frac{\|\mathbf x-\mathbf y\|}{t}\Big)^{-1} \chi_{[0,4t]}(d(\mathbf x,\mathbf y))
     \end{split}
 \end{equation*}
 for all $\mathbf x,\mathbf x',\mathbf y\in\mathbb R^N$, $\|\mathbf x-\mathbf x'\| \leq t$.}\label{teo_c}
 \item{If $\kappa >\mathbf N+m+1$, then there is a constant $C_\kappa>0$ such that for all functions $\psi,\, \eta, \phi \in C^{\kappa}(\mathbb R^N)$ supported in the unit ball $B(0,1)$ such that $\phi=\Delta_k^{m/2}\eta $ and all $0<s\leq t$ one has 
 \begin{equation*}
     \begin{split}
         |(\psi_t & *\phi_s)(\mathbf x,\mathbf y)-(\psi_t *\phi_s)(\mathbf x',\mathbf y)|\\
         & \leq C_\kappa \|\eta\|_{C^{\kappa}(\mathbb R^N)}\|\psi\|_{C^{\kappa}(\mathbb R^N)}\Big( \frac{s}{t}\Big)^m\frac{\|\mathbf x-\mathbf x'\|}{t}w(B(\mathbf x,t))^{-1} \Big(1+\frac{\|\mathbf x-\mathbf y\|}{t}\Big)^{-1} \chi_{[0,4t]}(d(\mathbf x,\mathbf y))
     \end{split}
 \end{equation*}
 for all $\mathbf x,\mathbf x',\mathbf y\in\mathbb R^N$, $\|\mathbf x-\mathbf x'\|\leq t$. }\label{teo_d}
     \end{enumerate}
 \end{theorem}
{The proof of Theorem~\ref{teo:bounds_conv} is provided in the appendix. }
 
\section{Hardy and \texorpdfstring{${\rm BMO}$}{BMO} spaces - basic properties} The Hardy space $H^1_{\rm Dunkl}$ is
\begin{align*}
     H^1_{\rm Dunkl}: =\left\{f\in L^1(dw): \|f\|_{H^1_{\rm Dunkl}}:=\|\mathcal Mf\|_{L^1(dw)}<\infty\right\},
\end{align*}  
where
\begin{equation*}
    \mathcal{M}f(\mathbf{x})=
\sup\nolimits_{\,{\|}\mathbf{x}-\mathbf{x}'{\|}^2<t}|e^{t\Delta_k}f(\mathbf{x}')|.
\end{equation*}
{As in the classical theory of Hardy spaces \cite{FeSt}, the space $H^1_{\rm Dunkl}$ admits other characterizations (see in~\cite{ADH-JFAA} and~\cite{DH-Studia}). 
 In the present section we state  characterizations of the space $H^1_{\rm Dunkl}$ by relevant Riesz transforms and atomic decompositions proved there. Then we present basic properties of its dual, namely the ${\rm BMO}(\mathbf X)$-space. }

\subsection{Characterizations of \texorpdfstring{ $H^1_{\rm Dunkl}$}{H1}} \ 

\subsubsection{Characterization by Riesz transforms}

 The Riesz transform are defined in the Dunkl setting as convolution operators associated with regular kernels of order zero $ \boldsymbol R_j(\mathbf{x})=c'_k x_j\|\mathbf{x}\|^{-\mathbf N-1}$, hence  they are Dunkl multiplier operators:
 \begin{align*}
   \mathbf R_jf=-T_j(-\Delta_k)^{-1/2}f= \mathcal F^{-1} \Big(-i\frac{\xi_j}{\|\xi\|}\mathcal F f(\xi)\Big).
 \end{align*}
Since the action of $\mathbf R_j$ on $L^1(dw)$-functions is well defined in the sense of distribution,  we set 
$$H^1_{\rm Riesz}=\{f\in L^1(dw): \|\mathbf R_jf\|_{L^1(dw)}<\infty \quad \text{\rm for all } j=1,2,...,N\}.$$
It was proved in \cite[Theorem 2.5]{ADH-JFAA} that the spaces $H^1_{\rm Dunkl}$ and $H^1_{\rm Riesz}$ coincide and the corresponding norms $\| f\|_{H^1_{\rm Dunkl}}$ and 
\begin{align*}
    \|f\|_{\rm Riesz}:=\| f\|_{L^1(dw)}+\sum_{j=1}^N \| \mathbf R_jf\|_{L^1(dw)}
\end{align*}
 are equivalent. 

\subsubsection{Characterization by atoms of Coifman--Weiss type}\label{sec:atomic}
Fix $1<q\leq \infty$. We say that a measurable function $a:\mathbb{R}^N \longmapsto \mathbb{C}$ is a $(1,q)$-atom if there is a ball $B$ such that $\text{supp}\, a\subseteq B$, $\|a\|_{L^{\infty}}\leq w(B)^{\frac{1}{q}-1}$, $\int_{\mathbb{R}^N} a(\mathbf{x})\, dw(\mathbf{x})=0$. A function $f$ belongs to $H^1_{(1,q)}$ if there are sequences  $\{\lambda_j\}_{j \in \mathbb{N}}$ of complex numbers and $(1,q)$-atoms $\{a_j\}_{j \in \mathbb{N}}$ such that $f=\sum_{j=0}^{\infty} \lambda_j a_j$ and $\sum_{j=0}^{\infty} |\lambda_j|<\infty$. Then 
\begin{align*}
    \|f\|_{H^1_{(1,q)}}:=\inf\Big\{ \sum_{j=0}^{\infty} |\lambda_j|\Big\},
\end{align*}
where the infimum is taken over all representation of $f$ as above. Now    \cite[Theorem 1.5]{DH-Studia} asserts that for all $1<q\leq \infty$ 
the spaces $H^1_{\rm Dunkl}$ and $H^1_{(1,q)}$
  coincide and the norms $\|f\|_{H^1_{\rm Dunkl}}$ and $\| f\|_{H^1_{(1,q)}}$ are equivalent. 

\subsection{
\texorpdfstring{${\rm BMO}(\mathbf X)$}{BMO(X)}
 and \texorpdfstring{${\rm VMO}(\mathbf X)$}{VMO(X)}   spaces - duality}\label{sec:BMO} 
For a locally measurable function $f$ and a Euclidean ball $B \subset \mathbb{R}^N$, let us denote 
\begin{equation*}
    f_{B}:=\frac{1}{w(B)}\int_B f(\mathbf{x})\,dw(\mathbf{x}).
\end{equation*}
 We set 
 \begin{align*}
      {\rm BMO}(\mathbf X)=\{b\in L^1_{\rm loc} (dw): \|b\|_{{\rm BMO}({\mathbf{X}})}<\infty\}, \ \ \| b\|_{{\rm BMO}({\mathbf{X}})}:=\sup_{B}\frac{1}{w(B)}\int_B |b(\mathbf x)- b_B|\, dw(\mathbf{x}).\index[Other]{$ {\rm BMO}(\mathbf X)$@$ {\rm BMO}(\mathbf X)$}
 \end{align*}
Since the Hardy space $H^1_{\rm Dunkl}$ admits decomposition into Coifman-Weiss atoms described  in Section \ref{sec:atomic}, Theorem B in \cite{CW} states that  its dual can be identified with  ${\rm BMO}(\mathbf{X})$ (cf.  \cite[pages 142-144]{Stein-Harmonic}). To be more precise,  if $b\in {\rm BMO}(\mathbf{X})$, then the mapping 
\begin{equation*}
    H^1_{\rm Dunkl} \ni  f\longmapsto \langle b,f\rangle=\sum_{j}\lambda_j\int_{\mathbb{R}^N} b(\mathbf x)a_j(\mathbf x)\, dw(\mathbf x)
\end{equation*} does not depend on decompositions of $f=\sum_j \lambda_j a_j$ into $(1,q)$-atoms and defines a bounded  linear functional, that is, 
$$|\langle b,f\rangle |\leq C_{q} \|b\|_{{\rm BMO}(\mathbf X)}\| f\|_{H^1_{(1,q)}}.$$ 
Conversely, every bounded linear functional on $H^1_{\rm Dunkl}$ is of this type  and its norm is comparable to $\| b\|_{{\rm BMO}(\mathbf X)}$. 

Further, {keeping the terminology of \cite[page 638]{CW},} the ${\rm VMO}(\mathbf X)$ space is defined as the closure of the space $C_c(\mathbb R^N)$  of compactly supported continuous functions in the ${\rm BMO}(\mathbf X)$-norm. Theorem 4.1 of \cite{CW} asserts that the Hardy  space $H^1_{\rm Dunkl}$ is dual to ${\rm VMO}(\mathbf X)$, which means that each continuous linear functional on ${\rm VMO}(\mathbf X)$ has the form 
$$\langle f,b\rangle = \int_{\mathbb{R}^N} f(\mathbf x)b(\mathbf x)\, dw(\mathbf x) \quad \text{for all } b\in C_c(\mathbb R^N),$$
where $f\in H^1_{\rm Dunkl}$ and $\| f\|_{H^1_{\rm Dunkl}}$ is equivalent to the norm of the functional.

\subsection{Properties of
\texorpdfstring{${\rm BMO}(\mathbf X)$}{BMO(X)}
} In this subsection we collect elementary properties of ${\rm BMO}(\mathbf X)$ we shall need later on. We start by the following  well-known inequalities stated in the lemma below. 

\begin{lemma}\label{lem:log}
There is a constant $C>0$  such that for all $b\in L^1_{{\rm loc}}(dw)$, $r_1>r>0$, $\mathbf{x},\mathbf{y} \in \mathbb{R}^N$, and $\sigma \in G$, we have:
\begin{equation*}
    |b_{B(\mathbf{x},r)}-b_{B(\mathbf{x},r_1)}| \leq C\log\Big(\frac{r_1}{r}+4\Big)\|b\|_{{\rm BMO}(\mathbf{X})},
\end{equation*}

\begin{equation}\label{eq:not-far}
    |b_{B(\mathbf{x},r)}-b_{B(\mathbf{y},r)}| \leq C\|b\|_{{\rm BMO}(\mathbf{X})} \quad \text{\rm provided } \|\mathbf x-\mathbf y\|\leq 2r,
\end{equation}

\begin{equation}\label{eq:log_sigma}
    |b_{B(\mathbf{x},r)}-b_{B(\sigma(\mathbf{x}),r)}| \leq C\log\left(\frac{\|\sigma(\mathbf{x})-\mathbf{x}\|}{r}+4\right)\|b\|_{{\rm BMO}(\mathbf{X})},
\end{equation}
(John--Nirenberg inequality) for any $1\leq s<\infty$ there is $C_s>0$ such that for all positive integers $j$ we have
\begin{equation}\label{eq:John}
   \Big(\frac{1}{w(B(\mathbf x,2^jr))} \int_{B(\mathbf x,2^jr)}|b(\mathbf y)-b_{B(\mathbf x,r)}|^s\, dw(\mathbf y) \Big)^{1/s}\leq C_s j \|b\|_{\rm BMO(\mathbf{X})}.
\end{equation}
\end{lemma}

\begin{proposition}\label{propo:log_app}
Let $\delta_1>0$. There is a constant $C>0$ such that for all $r>t>0$, $\mathbf{x} \in B(\mathbf{x}_0,r)$, and any locally integrable functions $f$ we have
\begin{equation}\label{eq:log_app}
\begin{split}
    \int\limits_{\mathcal O(B(\mathbf{x}_0,2r))^c}
    &\Big(1+\frac{\|\mathbf{x}-\mathbf{y}\|}{t}\Big)^{-\delta_1}
    \frac{1}{w(B(\mathbf{x},d(\mathbf{x},\mathbf{y})+t))}|f(\mathbf{y})-f_{B(\mathbf{x}_0,2r)}|\,dw(\mathbf{y}) \\
    &\leq C \left(\frac{t}{r}\right)^{\delta_1/4}\|f\|_{{\rm BMO}(\mathbf{X})}.
    \end{split}
\end{equation}
\end{proposition}
\begin{proof} We split the integral in \eqref{eq:log_app} as follows
\begin{align*}
     \int\limits_{\mathcal O(B(\mathbf{x}_0,2r))^c}
    &\Big(1+\frac{\|\mathbf{x}-\mathbf{y}\|}{t}\Big)^{-\delta_1}
    \frac{1}{w(B(\mathbf{x},d(\mathbf{x},\mathbf{y})+t))}|f(\mathbf{y})-f_{B(\mathbf{x}_0,2r)}|\,dw(\mathbf{y})\\& \leq \sum_{n=\lfloor\log_2(r/t)\rfloor}^{\infty}\int_{V_n} \ldots,
\end{align*}
where 
\begin{equation*}
    V_n:=(\mathcal O(B(\mathbf{x}_0,2r))^c) \cap \mathcal{O}(B(\mathbf{x}_0,2^{n+1}t)) \cap (\mathcal{O}(B(\mathbf{x}_0,2^{n}t)))^c \text{ for }n \in \mathbb{N}_0, \, n \geq \lfloor \log_2(r/t)\rfloor .
\end{equation*}

If $\sigma \in G$, $\mathbf{y} \in B(\sigma(\mathbf{x}_0),2^{n+1}t) \setminus B(\mathbf{x}_0,2r)$ and $\mathbf{x} \in B(\mathbf{x}_0,r)$,  then 
\begin{equation*}
   2^{n+1}t+2\|\mathbf{x}-\mathbf{y}\|\geq  2^{n+1}t+\|\mathbf{x}-\mathbf{y}\|+r\geq \|\sigma(\mathbf{x}_0)-\mathbf{y}\|+\|\mathbf{x}-\mathbf{y}\|+\|\mathbf{x}-\mathbf{x}_0\|\geq\|\sigma(\mathbf{x}_0)-\mathbf{x}_0\|
\end{equation*}
{ and, consequently, 
\begin{equation}\label{eq:logg}
    \Big(4+\frac{\| \mathbf x-\mathbf y\|}{t}\Big)^{-\delta_1/4}\leq C 2^{n\delta_1/4} \Big(4+\frac{\| \mathbf x_0-\sigma (\mathbf x_0)\|}{t}\Big)^{-\delta_1/4}.
\end{equation}
}
Moreover, since $n \geq \lfloor\log_2(r/t) \rfloor$, for $\mathbf{y} \in V_n$ and $\mathbf{x} \in B(\mathbf{x}_0,r)$, we have
$$w(B(\mathbf{x},d(\mathbf{x},\mathbf{y})+t)) \sim w(B(\sigma(\mathbf{x}_0),2^{n}t))\quad \text{\rm for all } \sigma\in G.$$
 Therefore,  applying~\eqref{eq:logg} Lemma~\ref{lem:log}, we obtain
\begin{equation*}
    \begin{split}
         &\int_{V_n}\Big(1+\frac{\|\mathbf{x}-\mathbf{y}\|}{t}\Big)^{-\delta_1}
    \frac{1}{w(B(\mathbf{x},d(\mathbf{x},\mathbf{y})+t))}|f(\mathbf{y})-f_{B(\mathbf{x}_0,2r)}|\,dw(\mathbf{y}) \\
         &\leq  C2^{-3n\delta_1/4}\sum_{\sigma \in G}\int_{B(\sigma(\mathbf{x}_0),2^{n+1}t){\setminus B(\mathbf x_0,2r)}}\Big(1+\frac{\|\mathbf{x}-\mathbf{y}\|}{t}\Big)^{-\delta_1/4}\frac{1}{w(B(\sigma(\mathbf{x}_0),2^{n}t))}|f(\mathbf{y})-f_{B(\mathbf{x}_0,2r)}|\,dw(\mathbf{y})\\
         &\leq C2^{-3n\delta_1/4}\sum_{\sigma \in G}\int_{B(\sigma(\mathbf{x}_0),2^{n+1}t)}\frac{1}{w(B(\sigma(\mathbf{x}_0),2^{n}t))}|f(\mathbf{y})-f_{B(\sigma(\mathbf{x}_0),2^{n+1}t)}|\,dw(\mathbf{y})\\
         &+C2^{-3n\delta_1/4}2^{n\delta_1/4}\sum_{\sigma \in G}\int_{B(\sigma(\mathbf{x_{0}}),2^{n+1}t)}\left(\frac{\|\sigma(\mathbf{x}_0)-\mathbf{x}_0\|}{t}+4\right)^{-\delta_1/4}\\
         & \hskip4cm \times \frac{1}{w(B(\sigma(\mathbf{x}_0),2^{n}t))}|f_{B(\sigma (\mathbf{x}_0),2^{n+1}t)}-f_{B(\mathbf{x}_0,2^{n+1}t)}|\,dw(\mathbf{y})\\
         &+C2^{-3n\delta_1/4}\sum_{\sigma \in G}\int_{B(\sigma(\mathbf{x}_0),2^{n+1}t)}\frac{1}{w(B(\sigma(\mathbf{x}_0),2^{n}t))}|f_{B(\mathbf{x_{0}},2^{n+1}t)}-f_{B(\mathbf{x}_0,2r)}|\,dw(\mathbf{y})\\&\leq C2^{-n\delta_1/2}(n+1)\|f\|_{{\rm BMO}(\mathbf{X})} \leq C2^{-n\delta_1/4}\|f\|_{{\rm BMO}(\mathbf{X})}. 
    \end{split}
\end{equation*}
Finally, 
\begin{equation*}
    \sum_{n=\lfloor \log_{2}(r/t)\rfloor}^{\infty}2^{-n\delta_1/4}\|f\|_{{\rm BMO}(\mathbf{X})} \leq C'' \left(\frac{{t}}{r}\right)^{\delta_1/4}\|f\|_{{\rm BMO}(\mathbf{X})},  
\end{equation*}
so~\eqref{eq:log_app} is proved.
\end{proof}

{The following simple lemma will be used in our further considerations. For the convenience of the reader, we present the proof.

\begin{lemma}\label{lem:BMO_compactly_supported}
    Let $1\leq p<\infty$. There is a constant $C_{J{\text{-}}N,p}>0$ such that for all $\mathbf{x}_0 \in \mathbb{R}^N$, $r>0$, and $f \in L^p(dw)$ such that $\supp f \subseteq B(\mathbf{x}_0,r)$ we have
    \begin{equation*}
        \|f\|_{L^p(dw)} \leq C_{J{\text{-}}N,p} w(B(\mathbf{x}_0,r))^{1/p}\| f\|_{{\rm BMO}(\mathbf X)}.
    \end{equation*}
\end{lemma}
\begin{proof}
    The fact that $\supp f \subseteq B(\mathbf{x}_0,r)$ and \eqref{eq:not-far} give $|f_{B(\mathbf x_0,r)}|\leq C\|f\|_{{\rm BMO}(\mathbf X)}$. Hence, for $1\leq p<\infty$,  
\begin{equation*}
\begin{split}
\| f\|_{L^p(dw)} &\leq \Big(\int_{B(\mathbf x_0,r)}(|f(\mathbf x)-f_{B(\mathbf x_0,r)}|+|f_{B(\mathbf x_0,r)}|)^p\, dw(\mathbf x)\Big)^{1/p}\\
&\leq \Big(\int_{B(\mathbf x_0,r)}(|f(\mathbf x)-f_{B(\mathbf x_0,r)}|^p\, dw(\mathbf x)\Big)^{1/p} + C\| f\|_{{\rm BMO}(\mathbf X)} w(B(\mathbf x_0, r))^{1/p}\\
&\leq C_{J{\text{-}}N,p} w(B(\mathbf x_0, r))^{1/p} \| f\|_{{\rm BMO}(\mathbf X)}, 
\end{split} 
\end{equation*}
where in the last inequality we have used the John-Nirenberg inequality \eqref{eq:John}. 
\end{proof}
}

\subsection{\texorpdfstring{${\rm BMO}(\mathbf X)$}{BMO(X)} and Carleson measure}

\begin{theorem}\label{teo:Carleson}
Assume that
 $\varphi \in L^1(dw)$ is such that $\mathcal{F}\varphi \in L^1(dw)$ and the functions satisfy  the following properties:
\begin{equation*}
    \int_{\mathbb{R}^N}\varphi(\mathbf{x})\,dw(\mathbf{x})=0,
\end{equation*}
\begin{equation}\label{prop2}
    \sup_{\xi \in \mathbb{R}^N, \xi \neq 0}\int_0^{\infty}|\mathcal{F}\varphi(t\xi)|^2\frac{dt}{t}<\infty,
\end{equation}
there are constants $C>0$ and $\delta>0$ such that for all $\mathbf{x},\mathbf{y} \in \mathbb{R}^N$ and $t>0$ we have
\begin{equation}\label{prop3}
    |\varphi_t(\mathbf{x},\mathbf{y})| \leq C\Big(1+\frac{\|\mathbf{x}-\mathbf{y}\|}{t}\Big)^{-\delta}\frac{1}{w(B(\mathbf{x},d(\mathbf{x},\mathbf{y})+t))}.
\end{equation}
Then there is a constant $C>0$ such that for all $b\in {\rm BMO}(\mathbf{X})$ and all balls $B\subset \mathbb R^N$ we have 
\begin{equation*}
     \int_0^r\int_B |b*\varphi_t (\mathbf{x})|^2\frac{dw(\mathbf{x})\,dt}{t}\leq C^2\| b\|_{{\rm BMO}(\mathbf{X})}^2 w(B).
\end{equation*}
\end{theorem}

\begin{remark}\normalfont
    The inequality \eqref{eq:carleson1} means that the measure  
\begin{equation*}
   d \mu(\mathbf{x},t):=|b*\varphi_t (\mathbf{x})|^2\frac{dw(\mathbf{x})\,dt}{t} 
\end{equation*}
is a Carleson measure (related to $dw$) with its Carleson norm bounded by $C^2\| f\|_{{\rm BMO}(\mathbf X)}^2 $. 
\end{remark}

\begin{proof}[Proof of Theorem \ref{teo:Carleson}]
Fix a ball  $B=B(\mathbf{x}_0, r)$.  
Set $B^*=B(\mathbf{x}_0,6r)$.
We write all the elements of the group $G$ in a sequence  $\sigma_{0}={\text{\rm Id}}, \sigma_1,\sigma_2,\ldots,\sigma_{|G|-1}$. We inductively define  a partition of $\mathcal O(B^*)$ by the sets $U_j \subseteq \mathbb{R}^N$, $j=0,1,2,\ldots,|G|-1$:
\begin{equation*}
\begin{split}
&U_0:=B(\mathbf{x}_0,6r)=B^*,\\
   & U_1:=\{\mathbf{z}\in \mathbb{R}^N\;:\; \| \mathbf{z}-\mathbf{x}_0\|>6r, \ \|\mathbf{z}-\sigma_1(\mathbf{x}_0)\|\leq 6r\},\\
    &U_{j+1}:=\{\mathbf{z}\in \mathbb{R}^N\;:\; \| \mathbf{z}-\mathbf{x}_0\|>6r,\ \|\mathbf{z}-\sigma_{j+1}(\mathbf{x}_0)\|\leq 6r\}\setminus \left(\bigcup_{j_1=1}^jU_{j_1}\right) \text{ for }j \geq 1.
\end{split}
\end{equation*}  
We write
\begin{align*}
    b(\mathbf{y})&=b_{B^*}+(b(\mathbf{y})-b_{B^*})=b_{B^*}+(b(\mathbf{y})-b_{B^*})\chi_{(\mathcal O(B^*))^c}(\mathbf y)+\sum_{j=0}^{|G|-1} (b(\mathbf{y})-b_{B^*})\chi_{U_j}(\mathbf y)\\&=:f_1(\mathbf{y})+f_2(\mathbf{y})+\sum_{j=0}^{|G|-1} f_{\sigma_j}(\mathbf{y}).
\end{align*} 
It follows from~\eqref{eq:int-preserv} and our assumptions on $\varphi$ that $\int_{\mathbb{R}^N}\varphi_t(\mathbf{x},\mathbf{y})\,dw(\mathbf{y})=0$ for all $\mathbf{x} \in \mathbb{R}^N$.
Since $f_1$ is a constant function, $f_1*\varphi_t(\mathbf{x})=0$. 

Further, from property~\eqref{prop3} and Proposition~\ref{propo:log_app}, we conclude that for $\mathbf{x} \in B(\mathbf{x}_0,r)$ and $0<t\leq r$,  we have
\begin{align*}
&|\varphi_t*f_2(\mathbf{x})|=\left|\int_{\mathbb{R}^N}\varphi_t(\mathbf{x},\mathbf{y})f_2(\mathbf{y})\,{dw(\mathbf{y})}\right| \\&\leq C\int\limits_{\mathcal O(B(\mathbf{x}_0,2r))^c}
    \Big(1+\frac{\|\mathbf{x}-\mathbf{y}\|}{t}\Big)^{-\delta}
    \frac{1}{w(B(\mathbf{x},d(\mathbf{x},\mathbf{y})+t))}|f(\mathbf{y})-f_{B(\mathbf{x}_0,2r)}|\,dw(\mathbf{y}) \\&
\leq C \left(\frac{t}{r}\right)^{\delta/4} \| f\|_{{\rm BMO}(\mathbf X)}.
\end{align*}
So, 
\begin{equation*}
    \begin{split}
        \int_0^r\int_B |\varphi_t*f_2(\mathbf{x})|^2\frac{dw(\mathbf{x})dt}{t}\leq C^2\| b\|_{{\rm BMO}(\mathbf{X})}^2\int_0^r \int_B \Big(\frac{t}{r}\Big)^{\delta/4}\frac{dw(\mathbf{x})dt}{t}
    \end{split}\leq C^2\| b\|_{{\rm BMO}(\mathbf{X})}^2 w(B).
\end{equation*}
To deal with  $f_{\sigma_j}$, we assume that $U_j\ne \emptyset$.  We write 
$$ f_{\sigma_0}*\varphi_t (\mathbf{x})=((b-b_{B^*})\chi_{B^*})*\varphi_t(\mathbf{x}), $$
$$ f_{\sigma_j}*\varphi_t (\mathbf{x})=((b-b_{\sigma_j(B^*)})\chi_{U_j})*\varphi_t(\mathbf{x})+ ((b_{\sigma_j(B^*)}-b_{B^*})\chi_{U_j})*\varphi_t(\mathbf{x})\quad \text{for } j=1,2,..., |G|-1.$$
By Plancherel's equality~\eqref{eq:Plancherel}, the John-Nirenberg inequality~\eqref{eq:John}, and property~\eqref{prop2} ,
\begin{equation*}
    \begin{split}
        &\sum_{j=0}^{|G|-1} \int_0^r\int_B |((b-b_{\sigma_j(B^*)})\chi_{U_j})*\varphi_t(\mathbf{x})|^2\frac{dw(\mathbf{x})dt}{t}\\&\leq  \sum_{j=0}^{|G|-1} \int_0^\infty \int_{\mathbb R^N} |\mathcal F  ((b-b_{\sigma_j(B^*)})\chi_{U_j})(\xi)\mathcal F\varphi (t\xi)|^2 \frac{dw(\xi)dt}{t}\\
        &\leq  C
        \sum_{j=0}^{|G|-1} \| ((b-b_{\sigma_j(B^*)})\chi_{U_j})\|_{L^2(dw)}^2 \leq C'|G| \|b\|_{{\rm BMO}(\mathbf X)}^2 w(B).
    \end{split}
\end{equation*}

It remains to consider $(b_{\sigma_j(B^*)}-b_{B^*})\chi_{U_j}*\varphi_t$ for $j=1,2,...,|G|-1. $ To this end we note that by the definition of $U_j$  for $\mathbf y\in U_j$ and $\mathbf x\in B$, we have 
\begin{equation}\label{eq:const_term_1}
    \|\mathbf{y}-\mathbf{x}_0\|>6r, \quad \|\mathbf{x}-\mathbf{y}\|\sim (\| \sigma_j(\mathbf{x}_0)-\mathbf{x}_0\| +r).
\end{equation} So we conclude 
\begin{equation}\label{eq:const_term_2}
   | \chi_{U_j}*\varphi_t(\mathbf{x})|\leq C \int_{U_j} \Big(1+\frac{\|\mathbf{x}-\mathbf{y}\|}{t}\Big)^{-\delta}
    \frac{1}{w(B(\mathbf{x},d(\mathbf{x},\mathbf{y})+t))}\, dw(\mathbf{y})\leq C\left(\frac{t}{\| \sigma_j(\mathbf{x}_0)-\mathbf{x}_0\|+r}\right)^{\delta/2}. 
\end{equation}
Finally, using~\eqref{eq:const_term_1} and~\eqref{eq:const_term_2} together with and \eqref{eq:log_sigma} of Lemma~\ref{lem:log}, we arrive at 
\begin{equation*}
    \begin{split}
        \int_0^r\int_B |(b_{\sigma_j(B^*)}-b_{B^*})\chi_{U_j}*\varphi_t(\mathbf{x})|^2\frac{dw(\mathbf{x})dt}{t}\leq C^2\| b\|_{{\rm BMO}(\mathbf{X})}^2 w(B), 
    \end{split}
\end{equation*}
which completes the proof of Theorem \ref{teo:Carleson}.
\end{proof}

\begin{corollary}\label{coro:Carleson}
Assume that $\varphi\in C_c^\infty (B(0,1))$ and 
$ \int_{\mathbb{R}^N} \varphi (\mathbf{x})\, dw(\mathbf{x})=0.$
Then there is a constant $C>0$ such that for all $b\in {\rm BMO}(\mathbf{X})$ and all balls $B\subset \mathbb R^N$ we have 
\begin{equation}\label{eq:carleson1}
     \int_0^r\int_B |b*\varphi_t (\mathbf{x})|^2\frac{dw(\mathbf{x})\,dt}{t}\leq C^2\| b\|_{{\rm BMO}(\mathbf{X})}^2 w(B).
\end{equation}
\end{corollary}

\begin{proof}
    The corollary follows  from  Theorem \ref{teo:Carleson} and part~\eqref{translations_kernels:a} of Theorem \ref{teo:translations_kernels}.
    
\end{proof}

Let 
\begin{equation*}
    \mathcal P_tf(\mathbf{x})=e^{-t\sqrt{-\Delta_k}} f(\mathbf{x})=f*p_t(\mathbf{x}),
\end{equation*}
     be the Dunkl-Poisson semigroup, where 
     $$ p(\mathbf{x})=c_{N,k}(1+\|\mathbf{x}\|^2)^{-(\mathbf  N+1)/2}$$
     is the $k$-Cauchy kernel 
     (see~\cite[Section 5]{ThangaveluXu},~\cite{Roesler-Voit}). Set 
     \begin{equation*}
         \tilde \nabla_k\mathcal P_tf(\mathbf{x})=(\partial_t \mathcal P_t f(\mathbf{x}), T_1\mathcal P_t f(\mathbf{x}),\dots,T_N\mathcal P_tf(\mathbf{x})).
     \end{equation*}

 \begin{corollary}\label{coro:Carleson_Poisson}
    There is a constant $C>0$ such that for all balls $B(\mathbf{x}_0,r)$ one has 
     $$ \int_0^r\int_{B(\mathbf{x}_0,r)} \|t\tilde\nabla_k \mathcal P_tf(\mathbf{x})\|^2 \, dw(\mathbf{x})\, \frac{dt}{t}\leq C w(B(\mathbf{x}_0,r))\| f\|^2_{\text{\rm BMO}  (\mathbf X)}.$$
 \end{corollary}
 \begin{proof}
     Set $q(\mathbf{x})=\partial_t p_t(\mathbf{x})_{\big|t=1}=c_{N,k} (1+ \|\mathbf{x}\|^2)^{-(\mathbf N+3)/2}(\|\mathbf{x}\|^2-\mathbf N)$. 

 It is obvious that 
 $$ 0=\int_{\mathbb{R}^N} q(\mathbf{x})\, dw(\mathbf{x})=\int_{\mathbb{R}^N} T_j p(\mathbf{x})\, dw(\mathbf{x}).$$
 Since $\mathcal Fp(\xi)={ c_k^{-1}}\exp(-\|\xi\|)$, we have 
  $\mathcal Fq(\xi)=-{c_k^{-1}} \|\xi\|\exp(-\|\xi\|)$, $ \mathcal F(T_jp)(\xi)=i{c_k^{-1}}\xi_j \exp(-\|\xi\|)$. 
 So \eqref{prop2} is satisfied for $q(\mathbf{x})$ and $T_jp(\mathbf{x}).$ For estimate \eqref{prop3} for $q(\mathbf{x})$, see~\cite[(3.12), Proposition 3.6]{DH-Studia}. To see \eqref{prop3}
 for $t\tau_{\mathbf x}(T_jp_t)(-\mathbf y)=tT_{j,\mathbf x} p_t(\mathbf x,\mathbf y)$, we set $g(\mathbf{x})=T_jp(\mathbf{x})=\partial_j p(\mathbf{x})$ (because $p(\mathbf x)$ is a radial function) and note that 
 $$ |\partial^I g(\mathbf{x})|\leq C_I (1+\|\mathbf{x}\|)^{-\mathbf N-2-|I|} \quad \text{for all } I\in\mathbb N_0^N. $$
 Now Theorem 4.1 of \cite{DH-JFAA} asserts that 
 $$ |tT_{j,\mathbf{x}}p_t(\mathbf{x},\mathbf{y})|=|g_t(\mathbf{x},\mathbf{y})|\leq C_I \Big(1+\frac{\|\mathbf{x}-\mathbf{y}\|}{t}\Big)^{-1} \Big(1+\frac{d(\mathbf{x},\mathbf{y})}{t}\Big)^{-1}\frac{1}{w(B(\mathbf{x},d(\mathbf{x},\mathbf{y})+t))}.$$
 The proof is complete.  \end{proof}

\begin{remark}\normalfont
    
It was proved in~\cite[Proposition 7.4]{JL} that if $b\in L^1_{\rm loc}(dw)$ is such that
\begin{equation*}
    \int_{\mathbb R^N} |b(\mathbf x)|(1+\|\mathbf x\|)^{-\mathbf N-1}<\infty \text{\ \  and\  \ }\| t\tilde \nabla_{k} \mathcal P_t b(\mathbf x)\|^2 dw(\mathbf x)\frac{dt}{t} \text{ is a Carleson measure},
\end{equation*}
then $b$ belongs to ${\rm BMO}(\mathbf X)$. Corollary~\ref{coro:Carleson_Poisson} asserts that the inverse inclusion holds. 
\end{remark}


\section{Chang--Fefferman decomposition} 

\subsection{Calder\'on reproducing formula}\label{sub:reproducing} 
\index[Other]{$\mathcal{Q}$@$\mathcal{Q}$}
\index[Other]{$\mathcal{Q}_{2^j}$@$\mathcal{Q}_{2^j}$}
For $j \in \mathbb{Z}$, let
\begin{align*}
    \mathcal{Q}_{2^j}=\left\{\prod_{\ell=1}^{N}[2^{j}n_\ell,2^{j}(n_\ell+1))\::\: (n_1,n_2,...,n_N) \in \mathbb{Z}^N\right\}
\end{align*}
denote the decomposition of $\mathbb R^N$ into the dyadic cubes of side-length $\ell (Q)=2^j$. We shall denote by $\mathbf{z}_Q$ the center of the cube $Q$. Then $\mathcal Q=\bigcup_{j\in\mathbb Z}\mathcal Q_{2^j}$ forms the set of all dyadic cubes. 

Assume that  $\psi \in \mathcal S(\mathbb R^N)$ is such that $\int_{\mathbb{R}^N} \psi(\mathbf{x})\, dw(\mathbf{x})=0$.  It follows from the Plancherel's equality~\eqref{eq:Plancherel} that if  $f\in L^2(dw)$, then  the function 
 $F(\mathbf{x},t)=\psi_t*f(\mathbf{x})$ belongs to the tent space  $T_2^2(\mathbf X)=L^2(\mathbb R^N\times (0,\infty), \frac{dw(\mathbf{x})\,dt}{t})$  and 
 \begin{equation}\label{eq:f_to_F}
     \|F\|_{T^2_2(\mathbf X)}\leq C_\psi\| f\|_{L^2(dw)}. 
 \end{equation} 
 Conversely, the mapping 
 \begin{equation}\label{eq:projection} T_2^2 \ni F\mapsto \lim_{\varepsilon \to 0} \int_\varepsilon^{\varepsilon^{-1} } \psi_t * F(\cdot , t)(\mathbf{x})\frac{dt}{t}=: \mathbf P_{\psi}(F)(\mathbf x)
 \end{equation}
 is a bounded operator from $T_2^2(\mathbf X)$ into $L^2(dw)$. The convergence is in $L^2(dw)$.

Set
 \begin{equation}\label{eq:M_1}
     M_1=8 \lceil \mathbf{N} +1\rceil.
 \end{equation}
 {For further purposes, we fix $N_0, N_1>0$ satisfying  
\begin{equation}\label{eq:constantsN}
2\mathbf N<N_0, \quad     N_0+1< N_1 \leq (4M_1+N-1)/2.
\end{equation} 
}
From now on, we fix  real--valued radial functions  ${\boldsymbol \phi}, {\boldsymbol \eta}\in C_c^\infty (B(0,1/4))$ such that 
$${\boldsymbol \phi}=\Delta_k^{M_1}{\boldsymbol \eta}, \quad |\mathcal F{\boldsymbol \eta} (\xi)|\leq C \| \xi\|^2\ \ \text{\rm  for all } \xi \in \mathbb{R}^N,$$ 
and  
$$ \int_0^\infty (\mathcal F{\boldsymbol \phi})(t\xi)^4\,\frac{dt}{t}=1, \text{  for all  } \xi\ne 0.$$ 
Then $L^2(dw)\ni f\mapsto {\boldsymbol \phi}_t*{\boldsymbol \phi}_t*f\in T_2^2(\mathbf X)$ is an isometric injection and the following  Calder\'on reproducing formula holds: 
\begin{equation*}
     f(\mathbf{x})=\lim_{n\to\infty} \int_{2^{-n}}^{2^n} ({\boldsymbol \phi}_t*{\boldsymbol \phi}_t)*({\boldsymbol \phi}_t*{\boldsymbol \phi}_t*f)(\mathbf{x})\,\frac{dt}{t},
\end{equation*}
where the convergence is in $L^2(dw)$. Furthermore, 
 for $f\in L^2(dw)$, we have  
 \begin{equation}\label{eq:decomp_f_Q}
     \begin{split}
         f(\mathbf{x})&=\lim_{n\to\infty} \int_{2^{-n}}^{2^n}{\boldsymbol \phi}_t^{*4} *f(\mathbf{x})\,\frac{dt}{t}\\& =\lim_{n \to \infty}\int_{2^{-n}}^{2^n}\int_{\mathbb{R}^N}({\boldsymbol \phi}_t*{\boldsymbol \phi}_t)(\mathbf{x},\mathbf{y})({\boldsymbol \phi}_t * {\boldsymbol \phi}_t*f)(\mathbf{y})\,dw(\mathbf{y})\,\frac{dt}{t}\\
         &=\sum_{j\in\mathbb Z} \sum_{Q^j\in \mathcal Q_{2^{-j}}} \int_{2^{-j}}^{2^{-j+1}} \int_{Q^j}({\boldsymbol \phi}_t*{\boldsymbol \phi}_t)(\mathbf{x},\mathbf{y})({\boldsymbol \phi}_t*{\boldsymbol \phi}_t*f)(\mathbf{y})\, dw(\mathbf{y})\,\frac{dt}{t}\\
         &=:\sum_{j\in\mathbb Z} \sum_{Q^j\in \mathcal Q_{2^{-j}}} f_{\{Q^j\}}(\mathbf x),
     \end{split}
 \end{equation} 
 where the convergence  is   unconditional in $L^2(dw)$. 

\subsection{Chang--Fefferman decomposition}\label{sec:Chang} 

 For $Q^j\in \mathcal Q_{2^{-j}}$, $j \in \mathbb{Z}$, we set \index[Other]{$\lambda_{Q^j}$@$\lambda_{Q^j}$}
 \begin{equation}\label{eq:lambda}
      \lambda_{Q^j}:=\Big(w(Q^j)^{-1}\int_{2^{-j}}^{2^{-j+1}}\int_{Q^j} |{\boldsymbol \phi}_t^{*2}*f(\mathbf{y})|^2\, dw(\mathbf{y})\,\frac{dt}{t}\Big)^{1/2},
 \end{equation}
 \index[Other]{$a_{Q^j}$@$a_{Q^j}$}
\begin{equation}\label{eq:A_Q}
     a_{Q^j}(\mathbf{x}):=
     \lambda_{Q^j}^{-1} f_{\{Q^j\}} (\mathbf x)=
     \lambda_{Q^j}^{-1}\int_{2^{-j}}^{2^{-j+1}}\int_{Q^j}{\boldsymbol \phi}_t^{*2}(\mathbf{x},\mathbf{y})({\boldsymbol \phi}_t^{*2}*f)(\mathbf{y})\, dw(\mathbf{y})\,\frac{dt}{t}
 \end{equation}
 provided $\lambda_{Q^j}\ne 0$, otherwise we put $a_{Q^j} \equiv 0$. 
 
  Combining ~\eqref{eq:decomp_f_Q}--\eqref{eq:A_Q}, we obtain the generalization to the Dunkl setting of the  Chang--Fefferman decomposition:
 \begin{equation*}
     f(\mathbf{x})=\sum_{j\in\mathbb Z} \sum_{Q^j\in \mathcal Q_{2^{-j}}} \lambda_{Q^j} a_{Q^j}(\mathbf{x}),
 \end{equation*}
 where the convergence is unconditional in $L^2(dw)$.  Moreover, it follows from \eqref{eq:f_to_F} and the boundedness of $\mathbf P_{{\boldsymbol \phi}^{*2}}$ from $T^2_2(\mathbf X)$ to $L^2(dw)$ (see \eqref{eq:projection}) that  there is a constant $C_{\boldsymbol \phi}>0$ such that for any sub-collection  $\mathcal Q'\subseteq \mathcal Q$ and any $f\in L^2(dw)$, we have 
 \begin{equation}
     \label{eq:subcollection}
     \Big\| \sum_{Q\in \mathcal Q'} \lambda_Q a_Q\Big\|_{L^2(dw)} \leq C_{\boldsymbol \phi} \| f\|_{L^2(dw)}. 
 \end{equation}

 Observe that 
 $$ a_{Q^j}=\Delta_k^{2M_1}\tilde a_{Q^j},$$
 where 
 \begin{equation}\label{eq:atom_decomp}
     \widetilde a_{Q^j}(\mathbf x)=\lambda^{-1}_{Q^j}\int_{2^{-j}}^{2^{-j+1}}\int_{Q^j}t^{4M_1} ({\boldsymbol \eta}_t* {\boldsymbol \eta}_t)(\mathbf{x},\mathbf{y}) ({\boldsymbol \phi}_t^{*2}*f)(\mathbf{y})\, dw(\mathbf{y})\frac{dt}{t}.
 \end{equation}

The remaining part of the section is devoted for studying properties of the Chang--Fefferman decomposition in the case where $f$ is a compactly supported $\text{\rm BMO}(\mathbf X)$-function. Clearly, by the John-Niremberg inequality~\eqref{eq:John}, such a function belongs to $L^2(dw)$.

\subsection{Support properties}

{ For a cube $Q$ (not necessarily  dyadic) and $s>0$, let $sQ$ \index[Other]{$s{Q}$@$s{Q}$}denote the cube with the same center $\mathbf z_Q$  as $Q$ and sides parallel to the axes of the length  $s\ell  (Q)$. We set }\index[Other]{$Q^{\diamond}$@$Q^{\diamond}$}
\begin{align}\label{def:diamond}
    Q^{\diamond}:=\mathcal O(B(\mathbf z_Q, 2\sqrt{N}\ell (Q))).
\end{align} 

\begin{proposition}\label{prop:changg}
    Consider the Chang--Fefferman decomposition of a compactly supported ${\rm BMO}(\mathbf X)$-function $f$ (see Subsection~\ref{sec:Chang}). Then for all $Q\in\mathcal Q$, $a_Q$, and $\widetilde{a}_Q$, we have 

    \begin{equation*}
        \text{\rm supp}\, a_Q\subseteq Q^{\diamond}, \quad  \text{\rm supp}\, \widetilde a_Q\subseteq Q^{\diamond}. 
    \end{equation*} 
    Moreover, {there is a constant {$C_{10}\geq 1$} such that for all  compactly supported $f \in {\rm BMO}(\mathbf{X})$ and $Q \in \mathcal{Q}$, we have}
\begin{equation}\label{eq:sum_lambda2}
     \sum_{\substack{  P^{\diamond} \cap Q^{\diamond}\ne \emptyset, \\  \ell (P)\leq \ell (Q)} }
 \lambda_P^2w(P)\leq C_{10}w(Q)\| f\|^2_{{\rm BMO}(\mathbf X)}.
\end{equation}
     In particular, 
      \begin{equation}
          \label{eq:lambda_bound}
        0\leq   {\lambda_Q}\leq C_{10}\| f\|_{{\rm BMO}(\mathbf X)}. 
      \end{equation}
\end{proposition}
\begin{proof}
    The first assertion follows from~\eqref{eq:A_Q} and the fact that ${\boldsymbol \phi}_t^{*2}(\mathbf x,\mathbf y)=0$ for $\mathbf x\notin Q^{\diamond}$, ${\mathbf{y}}\in Q$ and $t\leq 2\ell (Q)$ (see~\eqref{eq:supp_transl}). 

 In order to prove~\eqref{eq:sum_lambda2}, we fix $Q\in \mathcal Q$ and  consider all the cubes $P\in\mathcal Q$ such that $\ell (P)\leq \ell (Q)$, $P^{\diamond}\cap Q^{\diamond}\ne \emptyset$. Then,  using \eqref{eq:lambda}, we get
 \begin{equation*}
     \begin{split}
          \sum_{ \substack{   P^{\diamond} \cap Q^{\diamond}\ne \emptyset,\\  \ell (P)\leq \ell (Q)}}
 \lambda_P^2w(P)&
 \leq \sum_{2^{-j}\leq \ell (Q)} \ \sum_{ \substack{P \in\mathcal Q_{2^{-j}},\\ P^{\diamond}\cap Q^{\diamond}\ne \emptyset}} \int_{2^{-j}}^{2^{-j+1} }\int_{P} |{\boldsymbol \phi}_t^{*2}*f(\mathbf{y})|^2\, dw(\mathbf{y})\,\frac{dt}{t}\\
 &\leq C' \int_0^{2\ell (Q)}\int_{(6Q)^{\diamond}}|{\boldsymbol \phi}_t^{*2}*f(\mathbf{y})|^2\, dw(\mathbf{y})\,\frac{dt}{t}\\
 &\leq C''|G|w(Q)\| f\|_{{\rm BMO}(\mathbf X)}^2,
     \end{split}
 \end{equation*} 
 where in the last inequality we have used Corollary \ref{coro:Carleson} and the doubling property~\eqref{eq:doubling}. 
\end{proof}

\subsection{Size and regularity properties}
 {In this section we derive estimates for the Dunkl derivatives of the functions $a_Q$ and $\tilde a_Q$ from the Chang--Fefferman decomposition of any $L^2$ functions $f$ (see Subsection \ref{sec:Chang}). We also prove  Lipschitz estimates and the cancellation property. We want to emphasize   that the constants $C, C_I, C_I'$ in Propositions \ref{prop:size_a}  and \ref{propo:regularity_a} do not depend on $f$. }

\begin{proposition}\label{prop:size_a}
Let $I \in \mathbb{N}_0^N$. There are  constants $C_I,C_I'>0$ such that for all
 {$a_Q$ and $\widetilde a_Q$ - functions from the Chang--Fefferman decomposition of an $L^2(dw)$-function $f$  (see Subsection \ref{sec:Chang}), and  for all $\mathbf{x} \in \mathbb{R}^N$} we have
 \begin{equation}\label{eq:a_Qbound_1}  
     |T^I \tilde a_{Q}(\mathbf x)| \leq C_I \ell(Q)^{4M_1-|I|}\Big(1+\frac{\| \mathbf x-\mathbf z_{Q}\|}{\ell(Q)}\Big)^{-1}\chi_{[0,2\sqrt{N}\ell (Q)]}(d(\mathbf x,\mathbf z_Q)),
 \end{equation}
 \begin{equation}\label{eq:a_Qbound}        
     |T^Ia_{Q}(\mathbf x)| \leq C_I'\ell (Q)^{-|I|} \Big(1+\frac{\| \mathbf x-\mathbf z_{Q}\|}{\ell(Q)}\Big)^{-1}\chi_{[0,2\sqrt{N}\ell (Q)]}(d(\mathbf x,\mathbf z_Q)).
     \end{equation}
\end{proposition}

\begin{proof}
     Using~\eqref{eq:atom_decomp} and part~\eqref{translations_kernels:a} of Theorem~\ref{teo:translations_kernels} and~\eqref{eq:lambda}, we get 
 \begin{equation*}
     \begin{split}
         |T^I\tilde a_{Q}(\mathbf{x})|&\leq C_I\lambda_{Q}^{-1}\int_{\ell(Q)}^{2\ell(Q)}\int_{Q}t^{4M_1 -|I|}w(B(\mathbf{y},t))^{-1} \Big(1+\frac{\|\mathbf{x}-\mathbf{y}\|}{t}\Big)^{-1} |{\boldsymbol \phi}_t^{*2}*f(\mathbf{y})|\, dw(\mathbf{y})\frac{dt}{t}\\
         &\ \ \ \ \times \chi_{[0,2\sqrt{N} \ell (Q)]}(d(\mathbf x, \mathbf z_Q))\\
         & \leq C_I\lambda_{Q}^{-1}\ell(Q)^{4 M_1-|I|}\Big(1+\frac{\| \mathbf{x}-\mathbf{z}_{Q}\|}{\ell(Q)}\Big)^{-1} \chi_{[0,2\sqrt{N}\ell (Q)]}(d(\mathbf x,\mathbf z_Q))\\
         &\quad \times \Big( w(B(Q))^{-1}\int_{\ell(Q)}^{2\ell(Q)}\int_{Q}|{\boldsymbol \phi}_t^{*2}*f(\mathbf{y})|^2\, dw(\mathbf{y})\frac{dt}{t}\Big)^{1/2}\\
         &= C_I \ell(Q)^{4M_1-|I|}\Big(1+\frac{\| \mathbf{x}-\mathbf{z}_{Q}\|}{\ell(Q)}\Big)^{-1}\chi_{[0,2\sqrt{N} \ell(Q)]}(d(\mathbf{x},\mathbf{z}_{Q})),
     \end{split}
 \end{equation*}
{where in the last inequality we have used the Cauchy--Schwarz inequality.} This proves~\eqref{eq:a_Qbound_1}. The proof of~\eqref{eq:a_Qbound} follows the same pattern. 
\end{proof}

\begin{proposition}\label{propo:regularity_a}
    There is a constant $C>0$ such that for all $a_Q$ - functions from the Chang--Fefferman decomposition of an $L^2(dw)$-function $f$ (see Subsection \ref{sec:Chang}) and  for all $\mathbf x,\mathbf{x}' \in \mathbb{R}^N$ are such that $d(\mathbf x,\mathbf x')\leq 2\sqrt{N}\ell (Q)$  we have
   \begin{equation}\label{eq:a_Qminus}
      |a_{Q}(\mathbf x)-a_Q(\mathbf x')| \leq C \frac{\|\mathbf x-\mathbf x'\|}{\ell(Q)}\Big(1+\frac{\| \mathbf x-\mathbf z_{Q}\|}{\ell(Q)}\Big)^{-1}\chi_{[0,8\sqrt{N}\ell (Q)]}(d(\mathbf x,\mathbf z_Q)).
      \end{equation}
\end{proposition}

\begin{proof}
     Let us note that if $d(\mathbf x,\mathbf x')\leq 2\sqrt{N}\ell (Q)$, $\mathbf y\in Q$, $\ell (Q)\leq t \leq 2\ell (Q)$, then ${\boldsymbol \phi}^{*2}_t(\mathbf x,\mathbf y)=0={\boldsymbol \phi}^{*2}_t(\mathbf x',\mathbf y)$ for $d(\mathbf x,\mathbf z_Q)>8\sqrt{N}\ell(Q)$ (see~\eqref{eq:supp_transl}). Hence, applying part~\eqref{translations_kernels:c} of Theorem \ref{teo:translations_kernels}, we conclude that 
 \begin{equation*}\label{eq:a_Q_lip}
    \begin{split}
        |a_{Q}&(\mathbf{x})-a_{Q}(\mathbf{x}')| \leq \lambda_{Q}^{-1} \int_{\ell (Q)}^{2\ell Q)}\int_{Q}|{\boldsymbol \phi}_t^{*2}(\mathbf{x},\mathbf{y})-{\boldsymbol \phi}_t^{*2}(\mathbf{x}',\mathbf{y})||{\boldsymbol \phi}_t^{*2}*f(\mathbf{y})|\, dw(\mathbf{y})\frac{dt}{t}\\
        &\leq C\frac{\lambda_{Q}^{-1}}{w(Q)}
        \int_{\ell (Q)}^{2\ell (Q)}\int_{Q}\frac{\|\mathbf{x}-\mathbf{x}'\|}{\ell (Q)}\Big(1+\frac{\|\mathbf{x}-\mathbf{y}\|}{\ell (Q)}\Big)^{-1} |{\boldsymbol \phi}_t^{*2}*f(\mathbf{y})|\, dw(\mathbf{y})\frac{dt}{t}.\\
    \end{split}
\end{equation*}
Since $1+\|\mathbf x-\mathbf y\|/\ell (Q)\sim 1+\|\mathbf x-\mathbf z_Q\|/\ell (Q)$ for $\mathbf y\in Q$, we obtain~\eqref{eq:a_Qminus} by applying the Cauchy-Schwarz inequality and~\eqref{eq:lambda}. 
\end{proof}

\begin{proposition}\label{propo:integral_0_a}
 Let
 {$a_Q$ and $\widetilde a_Q$ be the functions from the Chang--Fefferman decomposition of an $L^2(dw)$-function $f$}. Then for all  $I \in \mathbb{N}_0^N$, we have 
\begin{equation}\label{eq:integ_zero} \int_{\mathbb{R}^N} T^I\widetilde a_Q(\mathbf y)\, dw(\mathbf y)=\int_{\mathbb{R}^N} T^I a_Q(\mathbf y)\, dw(\mathbf y)=0.
\end{equation} 
\end{proposition}

\begin{proof}
    We note that for each multi-index $I \in \mathbb{N}_0^N$,
 \begin{equation*}\begin{split}
 \lambda^{-1}_{Q}&\int_{\mathbb R^N}\int_{\ell (Q)}^{2\ell (Q)} \int_{Q} t^{4M_1}\Big|T_{\mathbf{x}}^I({\boldsymbol \eta}_t* {\boldsymbol \eta}_t)(\mathbf{x},\mathbf{y}) ({\boldsymbol \phi}_t^{*2}*f)(\mathbf{y})\Big|\, dw(\mathbf{y})\,\frac{dt}{t}\,dw(\mathbf{x})\\
 &\leq C_I \frac{\ell (Q)^{4M_1-|I|}}{ \lambda_{Q}} 
 \int_{\ell (Q)}^{2 \ell (Q)} \int_{Q} \Big| ({\boldsymbol \phi}_t^{*2}*f)(\mathbf{y})\Big|\, dw(\mathbf{y})\frac{dt}{t}\leq C'_Iw(Q) \ell (Q)^{4M_1-|I|}<\infty.
 \end{split}
 \end{equation*}
 Hence   from the Fubini theorem and   the fact that $\int_{\mathbb{R}^N} T^I_{\mathbf x}{\boldsymbol \eta}_t(\mathbf x,\mathbf y)\, dw({\mathbf{x}})=0$ for all $I \in \mathbb{N}_0^{N}$ and $\mathbf y\in \mathbb R^N$, we conclude
 $$\int_{\mathbb{R}^N} T^I \widetilde a_Q(\mathbf x)\, dw(\mathbf x)=0.$$ 
 The proof of the second equality in~\eqref{eq:integ_zero} can be handled in much  the same way. 
\end{proof}

\subsection{Actions of singular integrals on \texorpdfstring{$a_Q$}{aQ}  }

\begin{proposition}\label{prop:action} 
     Let $\kappa$ be an even positive integer, $\kappa >\mathbf N+4M_1+1$. There exists $C>0$ such that for all  $0<M\leq 4M_1+N-1$, all the functions $a_Q$ from the Chang--Fefferman decomposition of any  $L^2(dw)$ function (see Subsection~\ref{sec:Chang}),  and all regular kernels $\boldsymbol S$ of order zero, we have 

     \begin{equation}\label{eq:size_a_Q}
     |\mathbf Sa_Q(\mathbf x)|\leq C \big(\boldsymbol c_1+\|S_0\|_{C^{\kappa}(\mathbb{R}^N)}\big)  \Big(1+\frac{\|\mathbf x-\mathbf z_Q\|}{\ell (Q)}\Big)^{-1}\Big(1+\frac{d(\mathbf x,\mathbf z_Q)}{\ell (Q)}\Big)^{-M}, 
\end{equation}
      \begin{equation}\begin{split}\label{eq:Holder_Sa_Q}
     |\mathbf S      & a_Q(\mathbf x) -\mathbf Sa_Q(\mathbf x')|\\
     &\leq C \Big(\boldsymbol c_1+\|S_0\|_{C^{\kappa}(\mathbb{R}^N)}\Big) \frac{\|\mathbf x -\mathbf x'\|}{\ell (Q)}  \Big(1+\frac{\|\mathbf x-\mathbf z_Q\|}{\ell (Q)}\Big)^{-1}\Big(1+\frac{d(\mathbf x,\mathbf x')}{\ell (Q)}\Big)^M\Big(1+\frac{d(\mathbf x,\mathbf z_Q)}{\ell (Q)}\Big)^{-M}.
 \end{split}\end{equation} 
 Moreover, 
  \begin{equation}\label{eq:a_Q_zero}
     \int_{\mathbb{R}^N} \mathbf Sa_Q(\mathbf x)\, dw(\mathbf x)=0. 
 \end{equation}
\end{proposition}

\begin{proof}
    By scaling we may assume that $\ell (Q)=1$. Further, thanks to Proposition~\ref{prop:size_a} and~\eqref{eq:distr_S}, it suffices to consider 
    \begin{equation*}
        \begin{split}
            \sum_{j\in \mathbb Z} \int_{\mathbb{R}^N} S_j(\mathbf x,\mathbf y)a_Q(\mathbf y)\, dw(\mathbf y)=\sum_{j\in \mathbb Z} \mathbf S_ja_Q(\mathbf x).
        \end{split}
    \end{equation*}
    From the definition of $a_Q$ (see~\eqref{eq:A_Q}), we get 
    \begin{equation}\label{eq:SjaQ}
        \mathbf S_ja_Q(\mathbf x)=\lambda_Q^{-1} \int_1^2\int_Q (S_j*{\boldsymbol \phi}_t^{*2})(\mathbf x,\mathbf y)({\boldsymbol \phi}_t^{*2} *f)(\mathbf y)\, dw(\mathbf y)\,\frac{dt}{t}. 
    \end{equation}
    In order to prove~\eqref{eq:size_a_Q}, we split the summation  into two parts. 
    
    Part 1: $j\leq 0$. Then, by part~\eqref{teo_a} of Theorem \ref{teo:bounds_conv}, we conclude that for $1\leq t\leq 2$ and $\mathbf y\in Q$, one has 
    \begin{align*}
    |(S_j*{\boldsymbol \phi}_t^{*2})(\mathbf x,\mathbf y)|\leq C\| S_0\|_{C^\kappa(\mathbb{R}^N)} w(B(\mathbf x,1))^{-1}2^{j} \Big(1+\| \mathbf x-\mathbf z_Q\|\Big)^{-1}\chi_{[0,4\sqrt{N}]}(d(\mathbf x,\mathbf z_Q)).    
    \end{align*}
    
Recall that $w(B(\mathbf x,1))\sim w(B(\mathbf z_Q,1))$ for $d(\mathbf x,\mathbf z_Q)\leq 4\sqrt{N}$ (see~\eqref{eq:measure} and~\eqref{eq:growth}). Hence, using~\eqref{eq:SjaQ}, the Cauchy-Schwarz inequality,  and~\eqref{eq:lambda}, we obtain 
 \begin{equation}\label{eq:Sa_j_minus} |\mathbf S_ja_Q(\mathbf x)|\leq   C\| S_0\|_{C^\kappa(\mathbb{R}^N)}2^{j} \Big(1+\| \mathbf x-\mathbf z_Q\|\Big)^{-1}\chi_{[0,4\sqrt{N}]}(d(\mathbf x,\mathbf z_Q)).
 \end{equation}

  Part 2: $j>  0$. We remark that $w(B(\mathbf x,2^j))\sim w(B(\mathbf z_Q,2^j))$ if $d(\mathbf x,\mathbf z_Q)\leq 4\sqrt{N}2^j$ (see~\eqref{eq:measure} and~\eqref{eq:growth}). So, from the part~\eqref{teo_b} of Theorem~\ref{teo:bounds_conv} and ~\eqref{eq:growth}, we deduce  that for $1\leq t\leq 2$ and $\mathbf y\in Q$, one has 
    \begin{equation*}
    \begin{split}
        |(S_j*{\boldsymbol \phi}_t^{*2})(\mathbf x,\mathbf y)|&\leq C\| S_0\|_{C^\kappa(\mathbb{R}^N)} w(B(\mathbf z_Q,2^j))^{-1}2^{-4M_1j} \Big(1+\frac{\| \mathbf x-\mathbf z_Q\|}{2^j}\Big)^{-1}\chi_{[0,4\sqrt{N}2^j]}(d(\mathbf x,\mathbf z_Q))\\
        &\leq C \| S_0\|_{C^\kappa(\mathbb{R}^N)} w(B(\mathbf z_Q,1))^{-1}2^{-(4M_1+N)j} \Big(1+\frac{\| \mathbf x-\mathbf z_Q\|}{2^j}\Big)^{-1}\chi_{[0,4\sqrt{N}2^j]}(d(\mathbf x,\mathbf z_Q)).
    \end{split}\end{equation*}
    Consequently, by~\eqref{eq:SjaQ},  \eqref{eq:lambda}, and the Cauchy-Schwarz inequality, we get  
    \begin{equation}\label{eq:Saj_plus}
        \begin{split}
             |\mathbf S_ja_Q(\mathbf x)|&\leq   C\| S_0\|_{C^\kappa(\mathbb{R}^N)} 2^{-(4M_1+N)j} \Big(1+\frac{\| \mathbf x-\mathbf z_Q\|}{2^j}\Big)^{-1}\chi_{[0,4\sqrt{N}2^j]}(d(\mathbf x,\mathbf z_Q))\\
            & \leq   C\| S_0\|_{C^\kappa(\mathbb{R}^N)} 2^{-(4M_1+N-1)j} \big(1+\| \mathbf x-\mathbf z_Q\|\big)^{-1}\chi_{[0,4\sqrt{N}2^j]}(d(\mathbf x,\mathbf z_Q)).
             \\
        \end{split}
    \end{equation}
    Summing up~\eqref{eq:Sa_j_minus} and~\eqref{eq:Saj_plus}, we get the desired estimate~\eqref{eq:size_a_Q}. 

    We now turn for proving~\eqref{eq:Holder_Sa_Q}. We may assume that $\| \mathbf x-\mathbf x'\|\leq 2$, otherwise~\eqref{eq:Holder_Sa_Q} follows from~\eqref{eq:size_a_Q}.  Thanks to Proposition~\ref{propo:regularity_a}, it suffices to deal with the sum
   \begin{equation*}
        \begin{split}
         \sum_{j\in \mathbb Z} \big(\mathbf S_ja_Q(\mathbf x)-\mathbf S_ja_Q(\mathbf x')\big)=   \sum_{j\in \mathbb Z} \int_{\mathbb{R}^N} \big(S_j(\mathbf x,\mathbf y)a_Q(\mathbf y)-S_j(\mathbf x',\mathbf y)a_Q(\mathbf y)\big)\, dw(\mathbf y). 
        \end{split}
    \end{equation*}
    Part 1: $j\leq 0$. Recall that $\int_{\mathbb{R}^N} S_0(\mathbf x)\, dw(\mathbf x)=0$.  Using part~\eqref{teo_c} of Theorem~\ref{teo:bounds_conv}, for $1\leq t\leq 2$ and $\mathbf y\in Q$, we have 
   \begin{equation*}
       \begin{split}
          |S_j*{\boldsymbol \phi}_t^{*2}(\mathbf x,\mathbf y)&- S_j*{\boldsymbol \phi}_t^{*2}(\mathbf x',\mathbf y)|\\ 
          &\leq C\|S_0\|_{C^\kappa(\mathbb{R}^N)} w(B(\mathbf x,1))^{-1}2^j\|\mathbf x-\mathbf x'\|\Big(1+\|\mathbf x-\mathbf z_Q\|\Big)^{-1} \chi_{[0,8\sqrt{N}]}(d(\mathbf x,\mathbf z_Q)). 
       \end{split}
   \end{equation*}
   Utilizing~\eqref{eq:SjaQ} together with the Cauchy-Schwarz inequality, we obtain  
   \begin{equation}\label{eq:Sa_minus1}
       |\mathbf S_ja_Q(\mathbf x)-\mathbf S_ja_Q(\mathbf x')|\leq C \| S_0\|_{C^\kappa(\mathbb{R}^N)}2^j \| \mathbf x-\mathbf x'\|(1+\|\mathbf x-\mathbf z_Q\|)^{-1} \chi_{[0,8\sqrt{N} ]}(d(\mathbf x,\mathbf z_Q)). 
   \end{equation}

   Part 2: $j>0$. Recall that ${\boldsymbol \phi}*{\boldsymbol \phi}=\Delta_k^{2M_1}({\boldsymbol \eta}*{\boldsymbol \eta})$ and $\text{supp}\, ({\boldsymbol \eta}*{\boldsymbol \eta})\subseteq B(0,1/2)$. Applying the part~\eqref{teo_d} of Theorem~\ref{teo:bounds_conv}, we conclude that for $1\leq t\leq 2$ and $\mathbf y\in Q$, one has 
   \begin{equation*}
       \begin{split}
            |S_j*& {\boldsymbol \phi}_t^{*2}(\mathbf x,\mathbf y)- S_j*{\boldsymbol \phi}_t^{*2}(\mathbf{x}',\mathbf{y})|\\
            &\leq C \|S_0\|_{C^\kappa(\mathbb{R}^N)}2^{-4M_1j}\frac{\|\mathbf x-\mathbf x'\|}{2^j} \Big(1+\frac{\|\mathbf x-\mathbf z_Q\|}{2^j}\Big)^{-1}w(B(\mathbf x,2^{j}))^{-1} \chi_{[0,2^{j+3}\sqrt{N}]}(d(\mathbf x,\mathbf z_Q))\\
            &\leq  C \|S_0\|_{C^\kappa(\mathbb{R}^N)}2^{-(4M_1+N)j}\frac{\|\mathbf x-\mathbf x'\|}{2^j} \Big(1+\frac{\|\mathbf x-\mathbf z_Q\|}{2^j}\Big)^{-1}w(B(\mathbf x,1))^{-1} \chi_{[0,2^{j+3}\sqrt{N}]}(d(\mathbf x,\mathbf z_Q)),\\
       \end{split}
   \end{equation*}
where in the last inequality we have used~\eqref{eq:growth}. Hence, for any $0<M\leq 4M_1+N$, we have 
\begin{equation}\label{eq:Sa_minus2}
    \begin{split}
        |\mathbf S_j a_Q(\mathbf x)-\mathbf S_ja_Q(\mathbf x')|\leq C  \|S_0\|_{C^\kappa(\mathbb{R}^N)}2^{-Mj}\|\mathbf x-\mathbf x'\| (1+\|\mathbf x-\mathbf z_Q\|)^{-1} \chi_{[0,2^{j+4}\sqrt{N}]}(d(\mathbf x,\mathbf z_Q)).
    \end{split}
\end{equation}
Now using~\eqref{eq:a_Qminus}  and  summing  the inequalities~\eqref{eq:Sa_minus1} and~\eqref{eq:Sa_minus2}, we get 
\begin{equation*}
    \begin{split}
          |\mathbf Sa_Q(\mathbf x)-\mathbf Sa_Q({\mathbf{x}}')|& \leq C\boldsymbol c_1\|\mathbf x-\mathbf x'\|(1+\|\mathbf x-\mathbf z_Q\|)^{-1} \chi_{[0,8\sqrt{N}]}(d(\mathbf x, \mathbf z_Q))\\
          & \hskip 1cm + C\| S_0\|_{C^\kappa(\mathbb{R}^N)}\sum_{j\leq 0} 2^j\|\mathbf x-\mathbf x'\|(1+\| \mathbf x-\mathbf z_Q\|)^{-1} \chi_{[0, 8\sqrt{N}]} (d(\mathbf x,\mathbf z_Q))\\
          &\hskip 1cm + C  \| S_0\|_{C^\kappa(\mathbb{R}^N)}\sum_{2^j> d(\mathbf x,\mathbf z_Q)/128\sqrt{N}} 2^{-jM}\|\mathbf x-\mathbf x'\|(1+\| \mathbf x-\mathbf z_Q\|)^{-1} \\
         & \leq  C \big(\boldsymbol c_1+\|S_0\|_{C^{\kappa}(\mathbb{R}^N)}\big)\|\mathbf x-{\mathbf x'}\|(1+\|\mathbf x-\mathbf z_Q\|)^{-1} (1+d(\mathbf x,\mathbf z_Q))^{-M},
    \end{split}
\end{equation*}
which finishes the proof of~\eqref{eq:Holder_Sa_Q}. 

To get~\eqref{eq:a_Q_zero}, we note that by Propositions  \ref {prop:size_a} and \ref{propo:integral_0_a} the functions $a_Q$ belong to the Hardy space $H_{\rm Dunkl}^1$. Hence  \cite[Theorem 1.2]{DzHe-JFA} asserts that  $\mathbf Sa_Q\in H^1_{\rm Dunkl}$.   So~\eqref{eq:a_Q_zero} follows, since all the functions from $H^1_{\rm Dunkl}$ have integral zero (see Section \ref{sec:atomic}). 
\end{proof}

\begin{corollary}\label{cor:real_im} 
    Let $a_Q'(\mathbf x)=\text{\rm Re}\, a_Q(\mathbf x)$, $a_Q''(\mathbf x)=\text{\rm Im}\, a_Q(\mathbf x)$. Then  Proposition~\ref{prop:action} holds if we replace $a_Q$ either by $a_Q'$ or {by} $a_Q''$.  
\end{corollary}

\subsection{Main lemma}
Following Uchiyama~\cite{Uchiyama} we treat {$\mathbb C^{d+1}$ as a $2(d+1)$-dimensional } vector space over $\mathbb R$ with the identification { 
\begin{align*}
     \mathbb C^{d+1}\ni \overrightarrow {\boldsymbol v}\mapsto V( \overrightarrow {\boldsymbol v}) =(\text{Re}\,v_0, \text{Im}\,v_0,  \text{Re}\,v_1,\text{Im}\,v_1,  \text{Re}\,v_2,\text{Im}\,v_2,\dots , \text{Re}\,v_d,\text{Im}\,v_d)\in\mathbb R^{2d+2},
\end{align*}
\begin{align*}
    \overrightarrow {\boldsymbol v}=(v_0, v_1,v_2,\dots ,v_d). 
\end{align*}
Then $\| \overrightarrow {\boldsymbol v}\|=\| V(\overrightarrow {\boldsymbol v})\|$ and $\langle V(\overrightarrow {\boldsymbol v}), V(\overrightarrow {\boldsymbol w })\rangle_{\mathbb R^{2d+2}} =\text{Re} \, \langle {\overrightarrow {\boldsymbol v}} , {\overrightarrow {\boldsymbol w}}\rangle_{\mathbb C^{d+1}} =\text{Re} \, \langle  {\overrightarrow {\boldsymbol w}} , {\overrightarrow {\boldsymbol v}}\rangle_{\mathbb C^{d+1}}$. Here $ \langle {\overrightarrow {\boldsymbol v}} ,  {\overrightarrow {\boldsymbol w}}\rangle_{\mathbb C^{d+1}}=\sum_{j=0}^d v_j\overline{ w_j}$ denotes the inner product in the unitary space $\mathbb C^{d+1}$.  
}
\

{
\begin{lemma}(cf. \cite[Lemma 2.3]{Uchiyama}\label{cor:b_Q_ort})
    Assume that a system  $\overrightarrow S={ (\boldsymbol S^{\{0\}}}, \boldsymbol S^{\{1\}},\boldsymbol S^{\{2\}},\dots,\boldsymbol S^{\{d\}})$ of regular kernels of order zero satisfies the condition~\eqref{eq:triangle_new}. Then there is a constant $C_{11}\geq 1$ such that for any $a_Q$ from the Chang--Fefferman decomposition (see Subsection~\ref{sec:Chang}) and for any unit vector $\overrightarrow{\boldsymbol \nu}\in\mathbb C^{d+1}$ there is a function   $\overrightarrow{b_Q} =({b_Q^{\{0\}}, } b_Q^{\{1\}},b_Q^{\{2\}},\dots , b_Q^{\{d\}})$  (which takes values in $\mathbb C^{d+1}$) such that for all $0<M\leq  (4M_1+N-1)/2$ and $\mathbf{x},\mathbf{x}' \in \mathbb{R}^N$, we have
    \begin{equation}\label{eq:b_Q_bound}
        \|\overrightarrow{b_Q}(\mathbf x)\|\leq C_{11} \left(1+\frac{\|\mathbf{x}-\mathbf{z}_Q\|}{\ell(Q)}\right)^{-1}\left(1+\frac{d(\mathbf{x},\mathbf{z}_Q)}{\ell(Q)}\right)^{-2M};
\end{equation}

\begin{equation*}
\begin{split}
&\|\overrightarrow{b_Q}(\mathbf x)  -\overrightarrow{b_Q}(\mathbf{x}')\| \\
&\leq C_{11} \frac{\|\mathbf{x}-\mathbf{x}'\|}{\ell(Q)}\left(1+\frac{\|\mathbf{x}-\mathbf{z}_Q\|}{\ell(Q)}\right)^{-1}\left(1+\frac{d(\mathbf{x},\mathbf{x}')}{\ell(Q)}\right)^{2M}\left(1+\frac{d(\mathbf{x},\mathbf{z}_Q)}{\ell(Q)}\right)^{-2M}; 
\end{split}
\end{equation*} 
\begin{equation}
    \label{eq:int_b_Q_zero}
    \int_{\mathbb{R}^N} \overrightarrow{b_Q}(\mathbf y)\, dw(\mathbf y)=0; 
\end{equation}
\begin{equation}\label{eq:b_Q_id}
    \overrightarrow{S^*}\circ \overrightarrow{b_Q}(\mathbf x){ :=\sum_{j=0}^d \mathbf S^{\{j\} *} b_Q^{\{j\}} } =a_Q(\mathbf x), \quad \langle V(\overrightarrow{b_Q}(\mathbf x)),V(\overrightarrow{\boldsymbol \nu})\rangle_{\mathbb R^{2d+2}} =0. \index[Other]{$\overrightarrow{S^*}\circ $@$\overrightarrow{S^*}\circ $}
\end{equation}

\end{lemma}

}

\begin{proof}[Proof of Lemma \ref{cor:b_Q_ort}] The proof follows \cite[proof of Lemma 2.3]{Uchiyama}.
For any unit vector $\overrightarrow{\boldsymbol \nu}\in\mathbb C^{d+1}$, let $\Theta_j(\xi,\overrightarrow{\boldsymbol \nu})$ be as in Proposition \ref{propo:real_complex}. 
Each $\Theta_j(\xi,\overrightarrow{\boldsymbol \nu})$ uniquely corresponds to a regular kernel $\boldsymbol Z^{\{j\}}_{\overrightarrow{\boldsymbol \nu}}$ of order zero. 
For $a_Q$ from the Chang--Fefferman decomposition, set   $a_Q'(\mathbf x)=\text{\rm Re}\, a_Q(\mathbf x)$, $a_Q''(\mathbf x)=\text{\rm Im}\, a_Q(\mathbf x)$. Put  
\begin{equation}\label{eq:ZaQ}
    \begin{split}
        b_Q^{\{j\}} & = \mathbf Z_{\overrightarrow{\boldsymbol \nu}}^{\{j\}} a'_Q+ i \mathbf Z_{i\overrightarrow{\boldsymbol \nu}}^{\{j\}} a''_Q
        = \mathcal F^{-1} \Big( \Theta_j(\xi, \overrightarrow{\boldsymbol \nu} )\mathcal Fa'_Q(\xi) + i  \Theta_j(\xi,i \overrightarrow{\boldsymbol \nu} )\mathcal Fa''_Q(\xi)\Big).
    \end{split}
\end{equation}
It follows from Corollary \ref{cor:real_im} and Proposition \ref{prop:action} that~\eqref{eq:b_Q_bound} --~\eqref{eq:int_b_Q_zero} are satisfied. 

We turn to verify  the first equality in~\eqref{eq:b_Q_id}. To this end, using \eqref{eq:ZaQ}, we write 
\begin{equation*}
    \begin{split}
        \mathcal F\Big(\sum_{j=0}^d\mathbf S^{\{j\}*}b_Q^{\{j\}}\Big)(\xi )&=  \sum_{j=0}^d \overline{\theta_j(\xi)} \Theta_j(\xi,\overrightarrow{\boldsymbol \nu})\mathcal F a_Q'(\xi)+i\sum_{j=0}^d \overline{\theta_j(\xi)} \Theta_j(\xi,i\overrightarrow{\boldsymbol \nu})\mathcal F a_Q''(\xi)\\
        &= \mathcal Fa_Q'(\xi) +i \mathcal Fa_Q''(\xi)=\mathcal Fa_Q(\xi),
    \end{split}
\end{equation*}
where in the second equality we have used~\eqref{eq:Uchi1} for $\overrightarrow{\boldsymbol \nu}$ and $i\overrightarrow{\boldsymbol \nu}$. 

In order to verify the second equality  in~\eqref{eq:b_Q_id} we note that $2\text{\rm Re} (\mathcal F^{-1}g)=\mathcal F^{-1} (g+\overline{\reallywidecheck g})$ for any reasonable function/kernel $g$, where, according to our notation, $\reallywidecheck g(\mathbf x)=g(-\mathbf x)$.  We also recall that $g$ is real-valued if and only if $\overline{\mathcal Fa}={ \reallywidecheck{\mathcal F a}}$. Hence, by~\eqref{eq:ZaQ}, 
\begin{equation*}
    \begin{split}
      &  2 \langle V(\overrightarrow{b_Q}(\mathbf x)),V(\overrightarrow{\boldsymbol \nu})\rangle_{\mathbb R^{2d+2}}  =2 \text{\rm Re} \,\langle \overrightarrow{b_Q}(\mathbf x) , \overrightarrow{\boldsymbol \nu} \rangle_{\mathbb C^{d+1}}
      = 2\text{\rm Re} \Big(\sum_{j=0}^d \overline{\nu_j} b_Q^{\{j\}}
       (\mathbf x)\Big)\\
       &= 2\text{\rm Re }\,\mathcal F^{-1} \Big( \sum_{j=0}^d \overline{\nu_j} \Theta_j (\xi, \overrightarrow{\boldsymbol \nu})\mathcal Fa'_Q(\xi) + i \overline{\nu_j} \Theta_j (\xi, i\overrightarrow{ \boldsymbol \nu})\mathcal Fa''_Q(\xi)\Big)(\mathbf x)\\
       &= \mathcal F^{-1} \Big( \sum_{j=0}^d \overline{\nu_j} \Theta_j (\xi, \overrightarrow{\boldsymbol \nu})\mathcal Fa'_Q(\xi) +  \nu_j\overline{ \Theta_j (-\xi,  \overrightarrow{\boldsymbol \nu})\mathcal Fa'_Q(-\xi)}\Big)(\mathbf x)\\
       & \ \ +  \mathcal F^{-1} \Big( \sum_{j=0}^d i \overline{\nu_j} \Theta_j (\xi,i\overrightarrow{ \boldsymbol \nu})\mathcal Fa''_Q(\xi) -  i\nu_j\overline{ \Theta_j (-\xi, i \overrightarrow{\boldsymbol \nu})\mathcal Fa''_Q(-\xi)}\Big)(\mathbf x)\\
       &= \mathcal F^{-1} \Big( \sum_{j=0}^d \Big(\overline{\nu_j} \Theta_j (\xi, \overrightarrow{\boldsymbol \nu}) +  \nu_j\overline{ \Theta_j (-\xi,  \overrightarrow{\boldsymbol \nu})} \Big)\mathcal Fa'_Q(\xi)\Big)(\mathbf x)\\
       & \ \ +  \mathcal F^{-1} \Big( \sum_{j=0}^d \Big( i \overline{\nu_j} \Theta_j (\xi,i \overrightarrow{\boldsymbol \nu}) -  i\nu_j\overline{ \Theta_j (-\xi, i \overrightarrow{\boldsymbol \nu})} \Big)\mathcal F a''_Q(\xi)\Big)(\mathbf x).
    \end{split}
\end{equation*}
It easily follows from~\eqref{eq:Uchi2} that 
$$  \sum_{j=0}^d \Big(\overline{\nu_j} \Theta_j (\xi, \overrightarrow{\boldsymbol \nu}) +  \nu_j\overline{ \Theta_j (-\xi,  \overrightarrow{\boldsymbol \nu})} \Big) = 0 = \sum_{j=0}^d \Big( i \overline{\nu_j} \Theta_j (\xi,i \overrightarrow{\boldsymbol \nu}) -  i\nu_j\overline{ \Theta_j (-\xi, i \overrightarrow{\boldsymbol \nu})} \Big).$$
Thus the second inequality in~\eqref{eq:b_Q_id} is established. 
\end{proof}


\subsection{Auxiliary functions}\label{sec:auxiliary}
{ Recall that   $N_0>2\mathbf N$ (see~\eqref{eq:constantsN}).} For a compactly supported ${\rm BMO}(\mathbf X)$-function $f$ and for $j,m\in\mathbb Z$, we define 
\begin{equation}\label{eq:tau_def}
    \boldsymbol \tau_j(\mathbf x) =\boldsymbol \tau_j(f)(\mathbf x) :=\sum_{Q\in \mathcal Q_{2^{-j}}} \lambda_Q \Big(1+\frac{d(\mathbf z_Q,\mathbf x)}{2^{-j}}\Big)^{-N_0},\index[Other]{$ \boldsymbol \tau_j(\mathbf{x})$@$ \boldsymbol \tau_j(\mathbf x)=\boldsymbol \tau_j(f)(\mathbf{x})$}
\end{equation}
\begin{equation}\label{eq:sigma_def} \boldsymbol\sigma_m (\mathbf x)=\boldsymbol\sigma_m(f)(\mathbf x)=\sum_{j\leq m} \Big(\frac{9}{10}\Big)^{m -j}\boldsymbol\tau_{j}(\mathbf x) =\sum_{j=0}^\infty \Big(\frac{9}{10}\Big)^j\boldsymbol \tau_{m-j}(\mathbf x).\index[Other]{$\boldsymbol\sigma_m(\mathbf{x})$@$\boldsymbol\sigma_m(\mathbf{x})=\boldsymbol\sigma_m(f)(\mathbf{x})$}\end{equation}
where $\lambda_Q$ are the coefficients from the Chang--Fefferman decomposition of $f$ (see~\eqref{eq:lambda}).

Then 
\begin{equation}\label{eq:sigSum} \boldsymbol\sigma_m (\mathbf x)=\frac{9}{10}\boldsymbol\sigma_{m-1}(\mathbf{x})+\boldsymbol\tau_m(\mathbf x).
\end{equation} 

It can be easily proved using~\eqref{eq:finite_integral} and~\eqref{eq:doubling} that for ${K} >N$ there is a constant $C_K>0$ such that for all $j\in\mathbb Z$ and $\mathbf x\in\mathbb R^N$, one has 
\begin{equation} \label{eq:PQ_2}
\sum_{Q\in \mathcal Q_{2^{-j}}}  \Big(1+\frac{d(\mathbf z_Q,\mathbf x)}{2^{-j}}\Big)^{-{K}}\leq C_K.
\end{equation}

\begin{lemma}\label{lem:sum_square}
{There is a constant $C>0$ such that for any compactly  supported  ${\rm BMO}(\mathbf{X})$-function $f$ and all $\mathbf{x} \in \mathbb{R}^N$, we have}
    $$\boldsymbol\tau_j(\mathbf x)^2\leq C\sum_{P\in \mathcal Q_{2^{-j}}} \lambda_P^2\Big(1+\frac{d(\mathbf x, \mathbf z_P)}{2^{-j}}\Big)^{-N_0},$$
where $\lambda_P$ are the coefficients from the Chang-Fefferman decomposition of $f$.  \end{lemma}
\begin{proof}
    Because $ab \leq (a^2+b^2)/{2}$ for all $a,b \in \mathbb{R}$, we have 
    \begin{equation*}
    \begin{split}
        \boldsymbol\tau_j(\mathbf x)^2&=\sum_{(P, Q) \in \mathcal Q_{2^{-j}}\times \mathcal Q_{2^{-j}}} \lambda_P \lambda_Q \Big(1+\frac{d(\mathbf z_P,\mathbf x)}{2^{-j}}\Big)^{-N_0}\Big(1+\frac{d(\mathbf z_Q,\mathbf x)}{2^{-j}}\Big)^{-N_0}\\&\leq \frac{1}{2}
        \sum_{(P, Q) \in \mathcal Q_{2^{-j}}\times \mathcal Q_{2^{-j}}}
        \lambda_P^2 \Big(1+\frac{d(\mathbf z_P,\mathbf x)}{2^{-j}}\Big)^{-N_0}\Big(1+\frac{d(\mathbf z_Q,\mathbf x)}{2^{-j}}\Big)^{-N_0}\\ 
        &\ \ +\frac{1}{2}\sum_{(P, Q) \in \mathcal Q_{2^{-j}}\times \mathcal Q_{2^{-j}}}
        \lambda_Q^2 \Big(1+\frac{d(\mathbf z_P,\mathbf x)}{2^{-j}}\Big)^{-N_0}\Big(1+\frac{d(\mathbf z_Q,\mathbf x)}{2^{-j}}\Big)^{-N_0}\\
        & = 
    \sum_{P\in \mathcal Q_{2^{-j}}}\lambda_P^2 \Big(1+\frac{d(\mathbf z_P,\mathbf x)}{2^{-j}}\Big)^{-N_0}\sum_{Q \in\mathcal Q_{2^{-j}}}\Big(1+\frac{d(\mathbf z_Q,\mathbf x)}{2^{-j}}\Big)^{-N_0}.
    \end{split}\end{equation*}
 Now, thanks to ~\eqref{eq:PQ_2}, the proof is finished.
\end{proof}

\begin{lemma}[Christ-Geller\texorpdfstring{~\cite{CHG}}{CHG}]\label{lem:Ch-Gell12}
There is a constant $C_{12} \geq 1$ such that for all compactly supported $f\in {\rm BMO}(\mathbf X)$ {and for all $j,m\in\mathbb Z$,} one has 
    \begin{equation}\label{eq:tau_infty}\|\boldsymbol\tau_j\|_\infty \leq C_{12}\| f\|_{{\rm BMO}(\mathbf X)},
    \end{equation}
    \begin{equation}\label{eq:sigma_infty}\|\boldsymbol\sigma_m \|_\infty \leq C_{12}\| f\|_{{\rm BMO}(\mathbf X)},
    \end{equation}
    \begin{equation}\label{eq:sigma_tau}\int_{{ Q^{\diamond}}} \sum_{2^{j}\geq \ell (Q)^{-1}} \boldsymbol\sigma_j(\mathbf x)\boldsymbol\tau_j(\mathbf x)dw(\mathbf x)\leq C_{12} w(Q)\| f\|_{{\rm BMO}(\mathbf X)}^2 \quad \text{for all } Q\in\mathcal Q.
    \end{equation}
\end{lemma}
\begin{proof} {The inequality~\eqref{eq:lambda_bound} combined with~\eqref{eq:PQ_2} and the definition of $\boldsymbol \tau_j$ imply~\eqref{eq:tau_infty}.} Further, to prove~\eqref{eq:sigma_infty}  we use~\eqref{eq:sigma_def} and  ~\eqref{eq:tau_infty} and  we get  
   \begin{equation*}
       \begin{split}
           \boldsymbol\sigma_m(\mathbf x)\lesssim \sum_{j\leq m} \Big(\frac{9}{10}\Big)^{m-j}\| f\|_{{\rm BMO}(\mathbf X)}\lesssim\| f\|_{{\rm BMO}(\mathbf X)}.
       \end{split}
   \end{equation*}
 {We now turn to prove~\eqref{eq:sigma_tau}. We adapt  the proof of Christ-Geller \cite{CHG} p. 558. 
  Fix $Q\in\mathcal Q$. Let $\ell\in\mathbb Z$ be such that $\ell (Q)=2^{-\ell} $. For $P\in\mathcal Q$ such that $\ell(P)\leq \ell (Q)$, let $\tilde Q(P)\in\mathcal Q_{2^{-\ell}}$ be a unique cube which contains $P$, that is, $P\subset \tilde Q(P)$.  
  If $\mathbf x\in {  Q^{\diamond}}$, then by  the triangle inequality for the distance $d(\mathbf x,\mathbf y)$, we have 
   $$\Big(1+\frac{d(\mathbf z_P,\mathbf x)}{\ell (P)}\Big)^{-1} \leq C \Big(1+\frac{d(\mathbf z_Q,\mathbf z_{\tilde Q(P)})}{\ell (Q)}\Big)^{-1}. $$
 Hence,  applying Lemma \ref{lem:sum_square} and using the inequality 
 $$ \int_{\mathbb{R}^N} \Big(1+\frac{d(\mathbf z_P, \mathbf x)}{\ell (P)}\Big)^{-N_0/2}\, dw(\mathbf x) \leq C w(P), $$
we get 
   \begin{equation*}
   \begin{split}
      & \int_{Q^{\diamond} } \sum_{j= \ell}^{\infty} 
      \boldsymbol\tau_j(\mathbf x)^2\, dw(\mathbf x)\\
      &\lesssim \int_{Q^{\diamond}} \sum_{j= \ell}^\infty  \sum_{P\in \mathcal Q_{2^{-j}}}  \lambda_P^2  \Big(1+\frac{d(\mathbf z_Q, \mathbf z_{\tilde Q(P)})}{\ell (Q)}\Big)^{-N_0/2}\Big(1+\frac{d(\mathbf z_P, \mathbf x)}{\ell (P)}\Big)^{-N_0/2}\, dw(\mathbf x)\\
      &\lesssim \sum_{j= \ell}^\infty \sum_{P\in \mathcal Q_{2^{-j}}}  \lambda_P^2 w(P)  \Big(1+\frac{d(\mathbf z_Q, \mathbf z_{\tilde Q(P)})}{\ell (Q)}\Big)^{-N_0/2}\\
     & = \sum_{Q'\in\mathcal Q_{2^{-\ell}}} \sum_{P\in\mathcal Q, \, P\subseteq Q'} \lambda_P^2 w(P)  \Big(1+\frac{d(\mathbf z_Q, \mathbf z_{ Q'})}{\ell (Q)}\Big)^{-N_0/2}.\\
      \end{split}
      \end{equation*}
 Now we utilize ~\eqref{eq:sum_lambda2} and obtain  
 \begin{equation}\label{eq:tau_sum3}
   \begin{split}
       \int_{Q^{\diamond}} \sum_{j=\ell}^\infty \boldsymbol\tau_j(\mathbf x)^2\, dw(\mathbf x) &\lesssim \sum_{Q'\in\mathcal Q_{2^{-\ell}}}  \Big(1+\frac{d(\mathbf z_Q, \mathbf z_{Q'})}{\ell (Q')}\Big)^{-N_0/2} w(Q')\| f\|_{{\rm BMO}(\mathbf{X})}^2\\ 
       &\lesssim \| f\|^2_{{{\rm BMO}(\mathbf X)}} \int_{\mathbb R^N} \Big(1+\frac{d(\mathbf z_Q, \mathbf x)}{\ell (Q)}\Big)^{-N_0/2}\, dw(\mathbf x)\\
       &\lesssim w(Q)\| f\|^2_{{{\rm BMO}(\mathbf X)}}.
       \end{split}
   \end{equation}

 }

   If $n\geq 0$, then by~\eqref{eq:tau_sum3} and~\eqref{eq:tau_infty}, we get  
   \begin{equation}\begin{split}\label{eq:tautau} \int_{Q^{\diamond}}\sum_{j= \ell -n}^{\infty} \boldsymbol\tau ^2_{j}(\mathbf x)\, dw (\mathbf x)
   & \leq \int_{Q^{\diamond}}\sum_{j=\ell}^{\infty} \boldsymbol\tau_{j}^2(\mathbf x)\, dw(\mathbf x)+ \int_{{Q^{\diamond}}}\sum_{j= \ell -n}^{\ell -1} \boldsymbol\tau_{j}^2(\mathbf x)\, dw(\mathbf x)\\
   & \leq Cw(Q)(1+n)\|f\|^2_{{\rm BMO}(\mathbf X)}.
   \end{split} \end{equation} 
Thus, using the Cauchy-Schwarz inequality and then~\eqref{eq:tau_sum3}  and~\eqref{eq:tautau}, we obtain 
 $$\int_{Q^{\diamond}}\sum_{j= \ell}^{\infty}  \boldsymbol\tau_{j}(\mathbf x)\boldsymbol\tau_{j-n}(\mathbf x)\, dw (\mathbf x)\leq C w(Q)(1+n)^{1/2} \| f\|_{{\rm BMO}(\mathbf X)}^2.$$ 
 Finally, 
 \begin{equation*}
     \begin{split}
         \int_{Q^{\diamond}} \sum_{j=\ell }^{\infty} \boldsymbol\tau_j(\mathbf x)\boldsymbol\sigma_j(\mathbf x)\, dw(\mathbf x)
         &= \int_{Q^{\diamond}} \sum_{j= \ell }^{\infty} \sum_{n=0}^{\infty} \boldsymbol\tau_j({\mathbf{x}}) \Big(\frac{9}{10}\Big)^n\boldsymbol\tau_{j-n}(\mathbf x)\, dw(\mathbf x)\\
         &\leq C\sum_{m\geq 0} \Big(\frac{9}{10}\Big)^n w(Q)(1+n)^{1/2}\| f\|_{{\rm BMO}(\mathbf X)}^2\\
         &\leq C' w(Q)\| f\|_{{\rm BMO}(\mathbf X)}^2. 
     \end{split}
 \end{equation*}
The proof of Lemma \ref{lem:Ch-Gell12} is complete  by taking as $C_{12}$ the largest constant in the proved inequalities. \end{proof}
  \begin{lemma}[cf. \cite{CHG}, Lemma 3.4]\label{lem:Ch-Gell}
     There is a constant $C>0$ such that for any function   $f\in {\rm BMO}(\mathbf X)$ which is  supported in a ball  $B(\mathbf{x}_0,r)$, $r <2^{-j}$,  one has 
      \begin{equation*}
          \Big|\sum_{Q\in\mathcal Q_{2^{-j}}}\lambda_Q a_Q({\mathbf{x}})\Big|\leq C \frac{r^{{N}}}{2^{-j{N}}}\| f\|_{{\rm BMO}(\mathbf X)}.
      \end{equation*}
 \end{lemma} 
 \begin{proof}  
By Lemma~\ref{lem:BMO_compactly_supported},  we have
\begin{equation}\label{eq:(4.17)_for_L1}
        \|f\|_{L^1(dw)} \leq Cw(B(\mathbf{x}_0,r))\| f\|_{{\rm BMO}(\mathbf X)}.
    \end{equation}

 Recall that $\supp {\boldsymbol \phi}_t^{*2}(\mathbf{y},\cdot) \subseteq \mathcal{O}(B(\mathbf{y},t/2))$ (see~\eqref{eq:supp_transl}).  Thus  for $Q\in\mathcal Q_{2^{-j}}$, 
      \begin{equation}\label{eq:lambda_1}
      \begin{split}
      \lambda_{Q} &=\Big(w(Q)^{-1}\int_{2^{-j}}^{2^{-j+1}}\int_{Q} \left|\int_{\mathcal{O}(B(\mathbf{y},t/2)) \cap B(\mathbf{x}_0,r)}{\boldsymbol \phi}_t^{*2}(\mathbf{y},\mathbf{x})f(\mathbf{x})\,dw(\mathbf{x})\right|^2\, dw(\mathbf{y})\frac{dt}{t}\Big)^{1/2}
      \end{split}
 \end{equation} 
(see~\eqref{eq:lambda} for the definition of $\lambda_Q$).  Note also that $\lambda_Q=0$ if  $Q^{\diamond} \cap B(\mathbf{x}_0,r)=\emptyset$, so we can assume that $Q^{\diamond} \cap B(\mathbf{x}_0,r) \neq \emptyset$. {Thus, in this case, thanks to~\eqref{eq:growth}  and the assumption $r\leq \ell(Q)$,} for all  $\ell (Q)\leq t\leq 2\ell (Q)$, and $\mathbf y\in Q$,  we have 
$$w(B(\mathbf{y},t/2)) \sim w(B(\mathbf{x}_0,\ell(Q)))\sim w(Q).$$ 
     Consequently, by part~\eqref{translations_kernels:a} of Theorem~\ref{teo:translations_kernels}, for 
     $ \mathbf y\in Q$, $\mathbf x\in Q^{\diamond} \cap B(\mathbf{x}_0,r)$,  we get   
     \begin{equation*}
      |{\boldsymbol \phi}_t^{*2}(\mathbf{y},\mathbf{x})| \leq \frac{C}{w(B(\mathbf{y},4t))} \leq \frac{C}{w(B(\mathbf{x}_0,\ell(Q)))} .  
     \end{equation*}
     Therefore, from~\eqref{eq:lambda_1},~\eqref{eq:(4.17)_for_L1},  and~\eqref{eq:growth}, we conclude that 
     \begin{align*}
         \lambda_Q  &\leq C w(B(\mathbf x_0, 2^{-j}))^{-1} \| f\|_{L^1(dw)}\leq C \frac{w(B(\mathbf{x}_0,r))}{w(B(\mathbf{x}_0,\ell(Q)))}\|f\|_{{\rm BMO}(\mathbf X)}  \leq C\frac{r^{{N}}}{2^{-j{N}}}\|f\|_{{\rm BMO}(\mathbf X)}.
     \end{align*}
   { Recall that $|a_Q(\mathbf x)|\leq C$ (see~\eqref{eq:a_Qbound})}. Since $\text{supp}\, a_Q\subseteq Q^{\diamond}$ and the sets $Q^{\diamond}$, where $Q$ runs over  $\mathcal Q_{2^{-j}}$,  have bounded overlapping property with an overlapping constant independent of $j$, we obtain the lemma. 
 \end{proof}


 \section{Constructive Fefferman-Stein decomposition of \texorpdfstring{${\rm BMO}(\mathbf X$)}{BMO(X)}  functions} 

This section is devoted for proving Theorem \ref{UchiyamaDecomposition0}. 
 The main step of the proof  is the following theorem.
\begin{theorem}\label{KeyTheorem}
 Assume that $\overrightarrow S=( \delta_0, \boldsymbol S^{\{1\}},\boldsymbol S^{\{2\}},...,\boldsymbol S^{\{d\}})$ is a system of regular kernels of order 0 satisfying ~\eqref{eq:triangle}.
 Then there are constants  $A_0, A\geq 1$ and $  0<\varepsilon_0 <1$ such that {for all $r>0$}, any   ${\rm BMO}(\mathbf X)$-function $f$
 supported in  the ball $B(0, r)$ with  $\| f\|_{{\rm BMO}(\mathbf X)}=\varepsilon<\varepsilon_0$,  can be decomposed into 
\begin{equation}\label{eq:decom_need} f=\Big(\sum_{j=1}^d \mathbf S^{\{j\}*} \widetilde g_j \Big)+\widetilde g_0+ f_1,
\end{equation}
{ 
\begin{equation}\label{eq:g_infty}
\sum_{j=0}^d \| \widetilde g_j\|_{L^\infty} \leq { 2{ \sqrt{d+1}}+A\| f\|_{{\rm BMO}(\mathbf X)} }, 
\end{equation} }
\begin{equation}\label{eq:f_1BMO}
\| f_1\|_{{\rm BMO}}\leq A \varepsilon^2, \ \ \text{\rm supp}\, f_1\subseteq B(0,A_0 r),
\end{equation}
 Moreover,
\begin{equation}\label{eq:g_L2}
\sum_{j=0}^d \| \widetilde g_j\|_{L^2(dw)}\leq A w( B(0,r))^{1\slash 2} \| f\|_{{\rm BMO}(\mathbf X)}.
\end{equation}
\end{theorem}
\subsection{Proof of Theorem \ref{KeyTheorem}}
\begin{proof}

    The proof follows the ideas of Uchiyama~\cite{JFAA} and Christ and Geller~\cite{CHG}. 

   {For  $\overrightarrow S$, let $C_{11}$ be the constant from Lemma \ref{cor:b_Q_ort}. We shall prove Theorem \ref{KeyTheorem} with 
   \begin{equation*}
       A_0=32\sqrt{N} \quad \text{ and } \varepsilon_0=(100 C_{11}C_{12})^{-1}, 
   \end{equation*}
     where $C_{12}$ is the constant from Lemma \ref{lem:Ch-Gell12}. } {Recall that we fixed constants   $N_0>2\mathbf N$ and    $N_0+1< N_1 \leq (4M_1+N-1)/2$ (see~ \eqref{eq:M_1},  \eqref{eq:constantsN}, and  Section \ref{sec:auxiliary}).}

We may assume without loss of generality that $f$ is supported in $B(0,1)$, because the dilations $f^{r}(\mathbf x)=f(r\mathbf x)$ are isometries on ${\rm BMO}(\mathbf X)$ and $L^\infty(\mathbb R^N)$, 
$$ \| f\|_{L^2(dw)}^2= r^{\mathbf N}\| f^{r}\|_{L^2(dw)}^2 =\frac{w(B(0,r))}{w(B(0,1))} \| f^{r}\|_{L^2(dw)}^2,$$ (see~\eqref{eq:measure}), and $\mathbf S (f^{r})=(\mathbf Sf)^{r}$ for any regular kernel $\boldsymbol S$ of order zero. 

By the Chang--Fefferman decomposition given in Section \ref{sec:Chang}, we have
\begin{equation}\label{eq:f_deco_1} f=\sum_{Q \in \mathcal{Q}}\lambda_Q a_Q=\sum_{Q\in\mathcal Q,\, \ell  (Q)\leq 1} \lambda_Q a_Q+\sum_{Q\in\mathcal Q, \, \ell (Q)> 1} \lambda_Q a_Q=:f_0+g_0.
\end{equation}

For $\ell \in \mathbb{Z}$, let  $\boldsymbol \tau_\ell (\mathbf x)=\boldsymbol \tau_{\ell} (f)(\mathbf x)$ and $\boldsymbol \sigma_\ell (\mathbf x)=\boldsymbol \sigma_\ell (f)(\mathbf x)$ be the auxiliary functions associated with the decomposition (see ~\eqref{eq:tau_def} and~\eqref{eq:sigma_def}). 

It follows from Lemma \ref{lem:Ch-Gell}, {Lemma \ref{lem:BMO_compactly_supported}, and~\eqref{eq:subcollection}}   that
\begin{equation}\label{eq:g_0infty} \| g_0\|_{L^\infty} \leq C_{13}\| f\|_{{\rm BMO}(\mathbf X)}, \quad {\| g_0\|_{L^2(dw)}\leq C_{13} w(B(0,1))^{1/2} \| f\|_{{\rm BMO}(\mathbf X)}.}
\end{equation}
Thus, in farther consideration we shall deal with  the function
$f_0=\sum_{Q\in\mathcal Q, \, \ell (Q)\leq 1} \lambda_Qa_Q$.

For $\ell=-1$ we  define $\overrightarrow{g_{-1}}(\mathbf x)\equiv (1,0,...,0)\in { \mathbb  C^{d+1}}$, $\overrightarrow{E_{-1}}(\mathbf x)=\overrightarrow{h_{-1}}(\mathbf x)\equiv (0,...,0)\in { \mathbb C^{d+1}}$. 
Following~\cite{CHG}, our task is to construct, by induction, for each integer $\ell\geq 0$ functions $\overrightarrow{h_\ell}$, $\overrightarrow{g_\ell}$, and $\overrightarrow{E_\ell}$ on $\mathbb R^N$, taking values in { $\mathbb C^{d+1}$}, such that
\begin{equation}\label{c1}
\overrightarrow{S^*}\circ \overrightarrow{h_\ell} =\sum_{Q\in\mathcal Q_{2^{-\ell}}} \lambda_Qa_Q,
\end{equation}
\begin{equation}\label{c7}
\overrightarrow{h_\ell}=\sum_{Q\in \mathcal Q_{2^{-\ell}} }\lambda_Q\overrightarrow{b_{Q}}, 
\end{equation}
 with some functions $\overrightarrow{b_Q}$ which take values in { $\mathbb C^{d+1}$ } and satisfy 
\begin{equation}\label{c8}
\|\overrightarrow{b_Q}(\mathbf x)\|\leq C_{11} \Big(1+\frac{\|\mathbf x-\mathbf z_Q\|}{\ell (Q)}\Big)^{-1}\Big(1+\frac{d(\mathbf z_Q, \mathbf x)}{\ell(Q)}\Big)^{- 2N_1} \quad \text{ for all } \ \mathbf x\in\mathbb R^N,
\end{equation}
\begin{equation}\label{c9}
\|\overrightarrow{b_Q}(\mathbf x)-\overrightarrow{b_Q}(\mathbf y)\|\leq C_{11} \frac{\|\mathbf x-\mathbf y\|}{\ell (Q)} \Big(1+
\frac{d(\mathbf x,\mathbf y)}{\ell (Q)}\Big)^{N_1-1}\Big(1+\frac{\|\mathbf x-\mathbf z_Q\|}{\ell (Q)}\Big)^{-1} \Big(1+\frac{d(\mathbf z_Q, \mathbf x)}{\ell (Q)}\Big)^{-N_1+1} \ \ \text{\rm for all }
\mathbf x,\mathbf y\in\mathbb R^N,
\end{equation} 
\begin{equation}\label{c9'}
\|\overrightarrow{b_Q}(\mathbf x)-\overrightarrow{b_Q}(\mathbf y)\|\leq C_{11} \frac{\|\mathbf x-\mathbf y\|}{\ell (Q)} \Big(1+\frac{\| \mathbf x-\mathbf z_Q\|}{\ell (Q)}\Big)^{-1}\Big(1+\frac{d(\mathbf z_Q, \mathbf x)}{\ell (Q)}\Big)^{-2N_1+1} \ \ \text{\rm for }
\|\mathbf x-\mathbf y\|\leq A_0\ell (Q),
\end{equation}
\begin{equation}\label{c10}
\langle V( \overrightarrow{b_Q}(\mathbf x)), V(\overrightarrow{g_{\ell-1}}(\mathbf z_Q))\rangle_{ \mathbb R^{2d+2}}\equiv 0 \ \ \text{\rm for all } Q\in\mathcal Q_{2^{-\ell}},
\end{equation}
\begin{equation}\label{eq:b_Q_int0}
    \int_{\mathbb{R}^N} \overrightarrow{b_Q}(\mathbf x)\, dw(\mathbf x)=\overrightarrow{0},
\end{equation}
\begin{equation}\label{c2}
\|\overrightarrow{g_\ell}(\mathbf x)\|= 1 \text{ for all }\mathbf{x} \in \mathbb{R}^N,
\end{equation}
\begin{equation}\label{c3}
\overrightarrow{g_\ell}=\overrightarrow{g_{\ell-1}}+\overrightarrow{h_\ell}+\overrightarrow{E_\ell} \ \ \text{ for all }\ \ell\geq 0,
\end{equation} 
\begin{equation}\label{c11}
\|\overrightarrow{g_\ell}(\mathbf x)-\overrightarrow{g_\ell}(\mathbf y)\|
\leq A_1 \frac{\|\mathbf x-\mathbf y\|}{2^{-\ell}} \Big(1+
\frac{d(\mathbf x,\mathbf y)}{2^{-\ell}}\Big)^{N_1-1}\boldsymbol\sigma_\ell(\mathbf x)\ \ \text{\rm for all }
\mathbf x,\mathbf y\in\mathbb R^N,
\end{equation}
\begin{equation}\label{c12}
\|\overrightarrow{E_\ell}(\mathbf x)\|\leq  A_2\boldsymbol\tau_\ell(\mathbf x)\boldsymbol\sigma_\ell(\mathbf x), \quad  
\end{equation}
\begin{equation}\label{c13}
\|\overrightarrow{E_\ell}(\mathbf x)-\overrightarrow{E_\ell}(\mathbf y)\|\leq { A_3 } \frac{\|\mathbf x-\mathbf y\|}{2^{-\ell}} \| f\|_{{\rm BMO}(\mathbf X)}^2 \ \ \text{\rm for } \|\mathbf x-\mathbf y\|\leq A_02^{-\ell},
\end{equation}

\begin{equation}\label{c5}
\sum_{\ell=0}^\infty \overrightarrow{E_\ell} \ \text{ converges in } L^1_{\rm loc}(dw) \ \text{ to } \overrightarrow{E}^0 \in {\rm BMO}(\mathbf X),
\end{equation}
\begin{equation}\label{c6}
\| \overrightarrow{E}^0\|_{{\rm BMO}(\mathbf X)} \leq {A_4}\| f\|_{{\rm BMO}(\mathbf X)}^2, 
\end{equation}
\begin{equation}\label{c4}
\{\overrightarrow{g_\ell}\}_{\ell \in \mathbb{N}} \ \text{ converges in } \ L^1_{\rm loc}(dw) \ \text{ to } \overrightarrow{g} \in L^\infty, 
\end{equation}

 where the constants $A_1,A_2,A_3,A_4$ are indicated in the Table~\ref{tab:constants} at the end of the article.

\

{In the proofs of \eqref{c1}--\eqref{c4} several inequalities with constants will occur. The constants will not  depend on $f$, provided $\|f\|_{\text{BMO}(\mathbf X)}=\varepsilon <\varepsilon_0$, $\text{supp}\, f\subset B(0,1)$.  For this purpose we will control  their dependencies. }
We shall provide the inductive step. The proof  { for  $\ell =0$ }is essentially the same as the inductive step. 

Assume that~\eqref{c1}--\eqref{c13} hold  for all $j \in \mathbb{Z}$ such that $0\leq j<\ell$.

{\bf Step 1: constructions of }$\overrightarrow{h_\ell}$ and $\overrightarrow{b_Q}$, {\bf  proofs of} \eqref{c1}--\eqref{eq:b_Q_int0}.  For $Q\in\mathcal Q_{2^{-\ell}}$, we apply   Lemma \ref{cor:b_Q_ort} with $\overrightarrow{\boldsymbol \nu}=\overrightarrow{g_{\ell -1}}(\mathbf z_Q)$ and obtain { a vector-valued function }   
$\overrightarrow{b_Q}(\mathbf x)$ satisfying~\eqref{c8}--\eqref{eq:b_Q_int0}  such that $a_Q(\mathbf x)=\overrightarrow{S^*}\circ\overrightarrow{b_Q}(\mathbf x)$ (see \eqref{eq:b_Q_id}). 
{Observe that $$\#\{Q\in \mathcal Q_{2^{-\ell}} : b_Q\not\equiv 0\}\ \ \text{ is finite},$$ since ${\text{supp}}\, f\subset B(0,1)$ (see \eqref{eq:number_Q}). }
Let  $\overrightarrow{h_\ell}(\mathbf x)$ be given by~\eqref{c7}. Then {$\overrightarrow{h_\ell}\in L^1(dw)\cap L^\infty (\mathbb R^N)$ } and  \eqref{c1} follows from \eqref{eq:b_Q_id}.
{In Step 5 of the proof we provide estimates for the $L^2(dw)$-norm of sums of $\overrightarrow{h_\ell}$'s,}

{\bf Step 2: constructions of } $\overrightarrow{g_{\ell}}$ and $\overrightarrow{E_{\ell}}$, {\bf proofs of }\eqref{c2} and \eqref{c3}. 
Thanks to~\eqref{c8} and ~\eqref{eq:tau_infty}, we have  
\begin{equation}\label{h_l}
\|\overrightarrow{h_\ell}(\mathbf x)\|\leq C_{11}\boldsymbol\tau_\ell(\mathbf x)\leq  C_{11}C_{12}\|f\|_{{\rm BMO}(\mathbf X)}= C_{11}C_{12}\varepsilon. 
\end{equation}
 Define
 \begin{equation}\label{eq:defG_L} \overrightarrow{G_\ell}(\mathbf x):=\overrightarrow{g_{\ell-1}}(\mathbf x)+ \overrightarrow{h_\ell}(\mathbf x).
 \end{equation}
 Set $C_{14}= C_{11}C_{12}$.
 Since $\|\overrightarrow{g_{\ell-1}}(\mathbf x)\|=1$ { (see~\eqref{c2}),} from~\eqref{h_l} we conclude that 
 \begin{equation}\label{G1}1-C_{14}\varepsilon \leq \|\overrightarrow{G_\ell}(\mathbf x)\|\leq 1+C_{14}\varepsilon.
 \end{equation}
In other words, $\|\overrightarrow{G_\ell}(\mathbf x)\|$ is close to 1. Thanks to the orthogonality~\eqref{c10} 
the following better estimates are true:  
\begin{equation}\label{c14}
\big|1-\|\overrightarrow{G_\ell}(\mathbf x)\|\big|\leq A_2\boldsymbol\tau_\ell(\mathbf x)\boldsymbol\sigma_\ell(\mathbf x){\leq A_2C_{12}^2\varepsilon^2} \ \ \text{with }\ A_2 =\frac{20}{9} C_{11}^2A_1.
\end{equation}

{The second inequality in \eqref{c14} follows from Lemma \ref{lem:Ch-Gell12}. To see the first  one, } {having the induction hypotheses ~\eqref{c10} and \eqref{c11} in mind,} and using  ~\eqref{c8}, we have  
\begin{equation*}\begin{split}
\big| & \langle V( \overrightarrow{h_{\ell}}(\mathbf x))  ,V(\overrightarrow{g_{\ell -1}}(\mathbf x))\rangle_{\mathbb R^{2d+2}}\big| = \Big|\sum_{Q\in\mathcal Q_{2^{-\ell}}}  \lambda_Q \langle V( \overrightarrow{b_Q}(\mathbf x)),V(\overrightarrow{g_{\ell-1}}(\mathbf x))\rangle_{\mathbb R^{2d+2}} \Big|\\
& =\Big|\sum_{Q\in\mathcal Q_{2^{-\ell}}} \lambda_Q \langle V(\overrightarrow{b_Q}(\mathbf x)),V(\overrightarrow{g_{\ell-1}}(\mathbf x)-\overrightarrow{g_{\ell-1}}(\mathbf z_Q))\rangle_{\mathbb R^{2d+2}} \Big|\\
&\leq \sum_{Q\in\mathcal Q_{2^{-\ell}}} \lambda_Q C_{11}\Big(1+\frac{\|\mathbf x-\mathbf z_Q\|}{\ell (Q)}\Big)^{-1}\Big(1+\frac{d(\mathbf z_Q,\mathbf x)}{\ell (Q)}\Big)^{-2N_1}A_1 \frac{\|\mathbf x-\mathbf z_Q\|}{2^{-\ell+1}}\Big(1+\frac{d(\mathbf x,\mathbf z_Q)}{2^{-\ell+1}}\Big)^{N_1-1}\boldsymbol\sigma_{\ell-1}(\mathbf x)\\
&\leq C_{11}A_1\sum_{Q\in\mathcal Q_{2^{-\ell}}} \lambda_Q \Big(1+\frac{d(\mathbf x,\mathbf z_Q)}{2^{-\ell}}\Big)^{-N_1-1}\boldsymbol\sigma_{\ell-1}(\mathbf x)\\
&\leq C_{11}A_1 \boldsymbol\tau_\ell(\mathbf x)\boldsymbol\sigma_{\ell-1}(\mathbf x).
\end{split}\end{equation*}

Recall that  $\|\overrightarrow{g_{\ell-1}}(\mathbf x)\|= 1$. 
 If $0<\varepsilon<\varepsilon_0<(100C_{14})^{-1}$, then~\eqref{h_l} asserts that $  \|\overrightarrow{h_\ell}(\mathbf x)\| <10^{-2}$. Hence,  
\begin{equation*}
\begin{split}
\big|1-\|\overrightarrow{G_\ell}(\mathbf x)\|\big|&\leq \big|1-\|\overrightarrow{G_\ell}(\mathbf x)\|\big| \cdot \big|1+\|\overrightarrow{G_\ell}(\mathbf x)\|\big|\\
&\leq 2\big| \langle V(\overrightarrow{h_\ell}(\mathbf x)),V(\overrightarrow{g_{\ell-1}}(\mathbf x))\rangle_{\mathbb R^{2d+2}}\big| +\|\overrightarrow{h_\ell}(\mathbf x)\|^2\\
&\leq 2 C_{11}A_1\boldsymbol{\tau}_\ell(\mathbf x)\boldsymbol\sigma_{\ell-1}(\mathbf x) + C_{11}^2\boldsymbol{\tau}_\ell(\mathbf x)^2\\
&\leq 2 C_{11}^2A_1\boldsymbol{\tau}_\ell(\mathbf x) \big(\boldsymbol{\sigma}_{\ell-1}(\mathbf x)+\boldsymbol{\tau}_\ell(\mathbf x)\big)\\
&\leq \frac{20}{9}  C_{11}^2A_1\boldsymbol{\tau}_\ell(\mathbf x)\boldsymbol{\sigma}_{\ell}(\mathbf x),
\end{split}
\end{equation*}
where in the last inequality we have used~\eqref{eq:sigSum}. Thus~\eqref{c14} is established.

Put 
\begin{equation*}
\begin{split}
&\overrightarrow{g_{\ell}} (\mathbf x):=\frac{\overrightarrow{G_{\ell}}(\mathbf x)}{\|\overrightarrow{G_{\ell}}(\mathbf x)\|},\\
&\overrightarrow{E_{\ell}}(\mathbf x):=\overrightarrow{g_{\ell}}(\mathbf x)-\big(\overrightarrow{g_{\ell-1}}(\mathbf x)+\overrightarrow{h_{\ell}} (\mathbf x)\big)=\frac{\overrightarrow{G_{\ell}} (\mathbf x)}{\|\overrightarrow{G_{\ell}} (\mathbf x)\|}- \overrightarrow{G_{\ell}} (\mathbf x).
\end{split}\end{equation*}
{Then \eqref{c2} and \eqref{c3} hold for $\overrightarrow{g_{\ell}}$ and $\overrightarrow{E_{\ell}}$.} 

{\bf Step 3: proof of }\eqref{c11}.  Since 
\begin{equation*}
    \left\|\frac{\overrightarrow{u}}{\|\overrightarrow{u}\|}-\frac{\overrightarrow{v}}{\|\overrightarrow{v}\|}\right\| \leq \frac{1}{\|\overrightarrow{u}\|}\|\overrightarrow{u}-\overrightarrow{v}\| \quad \text{      for all      }0<\|\overrightarrow{u}\|\leq \|\overrightarrow{v}\|,
\end{equation*}
using~\eqref{G1}, we have 
\begin{equation}\label{eq:g_lLip}
\begin{split}\|\overrightarrow{g_{\ell}}(\mathbf x)-\overrightarrow{g_{\ell}}(\mathbf y)\|
&\leq (1-C_{14}\varepsilon)^{-1}\|\overrightarrow{G_{\ell}}(\mathbf x)-\overrightarrow{G_{\ell}}(\mathbf y)\|.\\
\end{split}\end{equation}
Further, by the definition of  $\overrightarrow{G_{\ell}}(\mathbf x)$ {(see \eqref{eq:defG_L})}, we get  
\begin{equation*}\begin{split}
\|\overrightarrow{G_{\ell}}(\mathbf x)-\overrightarrow{G_{\ell}}(\mathbf y)\|
&\leq
\|\overrightarrow{g_{\ell-1}}(\mathbf x) -
\overrightarrow{g_{\ell-1}}(\mathbf y)\|+
\|\overrightarrow{h_{\ell}}(\mathbf x) -
\overrightarrow{h_{\ell}}(\mathbf y)\|=:I_1+I_2.
\end{split}\end{equation*}
By the induction hypothesis~\eqref{c11}, 
\begin{equation}\begin{split}\label{eq:I1}
I_1&\leq  A_1\frac{\|\mathbf x -\mathbf y\|}{2^{-\ell +1}}\Big(1+\frac{d(\mathbf x,\mathbf y)}{2^{-\ell +1}}\Big)^{N_1-1}\boldsymbol\sigma_{\ell-1}(\mathbf x)
\leq  \frac{A_1}{2}\frac{\|\mathbf x-\mathbf y\|}{2^{-\ell}}\Big(1+\frac{d(\mathbf x,\mathbf y)}{2^{-\ell}}\Big)^{N_1-1}\boldsymbol\sigma_{\ell-1}(\mathbf x)\\
&=\frac{ 2^{-1}}{\frac{9}{10}} A_1 \frac{\|\mathbf x-\mathbf y\|}{2^{-\ell}}\Big(1+\frac{d(\mathbf x,\mathbf y)}{2^{-\ell}}\Big)^{N_1-1}\Big(\frac{9}{10}\Big)\boldsymbol\sigma_{\ell-1}(\mathbf x).
\end{split}
\end{equation}
Applying ~\eqref{c7} and ~\eqref{c9}, we obtain 
\begin{equation}\begin{split}\label{eq:I2}
I_2&\leq \sum_{Q\in\mathcal Q_{2^{-\ell}} }\lambda_Q\|\overrightarrow{b_{Q}}(\mathbf x) -
\overrightarrow{b_{Q}}(\mathbf y)\| \\ &\leq C_{11} \frac{\|\mathbf x-\mathbf y\|}{2^{-\ell}} \Big(1+
\frac{d(\mathbf x,\mathbf y)}{2^{-\ell}}\Big)^{N_1-1}\sum_{Q\in\mathcal Q_{2^{-\ell}} } \lambda_Q \Big(1+\frac{\|\mathbf x-\mathbf z_Q\|}{2^{-\ell}}\Big)^{-1}\Big(1+\frac{d(\mathbf z_Q, \mathbf x)}{2^{-\ell}}\Big)^{-N_1+1}\\
&\leq C_{11} \frac{\|\mathbf x-\mathbf y\|}{2^{-\ell}} \Big(1+
\frac{d(\mathbf x,\mathbf y)}{2^{-\ell}}\Big)^{N_1-1}\boldsymbol\tau_\ell(\mathbf x).
\end{split}\end{equation}

   Observe that  
$$\frac{ 2^{-1}}{\frac{9}{10}(1-C_{14}\varepsilon)} < 1 \quad \text{\rm for all } 0<\varepsilon\leq \varepsilon_0.$$ 

Recall also that
 $A_1=C_{11}(1-C_{14}\varepsilon_0)^{-1}=\frac{100}{99}C_{11}$,  see Table~\ref{tab:constants}.
Hence, from ~\eqref{eq:I1},~\eqref{eq:I2}, {and~\eqref{eq:sigSum}}, we conclude that 
\begin{equation}\label{eq91}
\begin{split}
&(1-C_{14}\varepsilon)^{-1}\|\overrightarrow{G_{\ell}}(\mathbf x)-\overrightarrow{G_{\ell}}(\mathbf y)\|\\
&\leq \frac{\|\mathbf x-\mathbf y\|}{2^{-\ell}}\Big(1+\frac{d(\mathbf x,\mathbf y)}{2^{-\ell}}\Big)^{N_1-1} \Big(\frac{2^{-1}}{\frac{9}{10}(1-C_{14}\varepsilon)}
 \frac{9}{10}A_1\boldsymbol\sigma_{\ell-1}(\mathbf x)
 +\frac{C_{11}}{1-C_{14}\varepsilon}\boldsymbol\tau_{\ell}(\mathbf x)\Big)\\
 &\leq A_1 \frac{\|\mathbf x-\mathbf y\|}{2^{-\ell} }\Big(1+
\frac{d(\mathbf x,\mathbf y)}{2^{-\ell}}\Big)^{N_1-1}\boldsymbol\sigma_\ell(\mathbf x),
\end{split}\end{equation}
where in the last inequality we have used~\eqref{eq:sigSum}. Thus~\eqref{c11} follows from~\eqref{eq:g_lLip} and~\eqref{eq91}. 

{\bf Step 4: proofs of } \eqref{c12} and \eqref{c13}. 
Observe ~\eqref{c12} is  easily obtained  from~\eqref{c14} and \eqref{c2}. Indeed, 
\begin{equation}\label{G3}\begin{split}
\|\overrightarrow{E_{\ell}}(\mathbf x)\|&=\|\overrightarrow{g_{\ell}}(\mathbf x)-\overrightarrow{G_{\ell}}(\mathbf x)\| =\big\|\overrightarrow{g_{\ell}}(\mathbf x)(1-\|\overrightarrow{G_{\ell}}(\mathbf x)\|)\big\|\\
&=\big|1-\|\overrightarrow{G_{\ell}}(\mathbf x)\|\big|\leq A_2\boldsymbol\tau_\ell(\mathbf x)\boldsymbol\sigma_\ell(\mathbf x).
\end{split}\end{equation}

We now turn to prove~\eqref{c13}.  {Set  $C_{15}=4C_{11}^2C_{12}^2A_1(1+A_0)^{N_1-1}$. We start by showing that }
\begin{equation}\label{G4}
\Big|\|\overrightarrow{G_{\ell}}(\mathbf x)\|-\|\overrightarrow{G_{\ell}}(\mathbf y)\|\Big|\leq C_{15} \frac{\|\mathbf x-\mathbf y\|}{2^{-\ell}}\varepsilon^2 \ \ \text{\rm for } \|\mathbf x-\mathbf y\|\leq A_02^{-\ell}.
\end{equation}
{Recall that $\|g_{\ell-1}(\mathbf x)\|=1$ (by the induction hypothesis \eqref{c2}). So, } from~\eqref{G1},~\eqref{eq:defG_L}, and ~\eqref{h_l}, we get
\begin{equation*}\begin{split}
\Big|\|\overrightarrow{G_{\ell}}(\mathbf x)\|&-\|\overrightarrow{G_{\ell}}(\mathbf y)\|\Big|
\leq  \Big|\|\overrightarrow{G_{\ell}}(\mathbf x)\|^2-\|\overrightarrow{G_{\ell}}(\mathbf y)\|^2\Big|\\
&\leq  \Big|\|\overrightarrow{h_{\ell}}(\mathbf x)\|^2-\|\overrightarrow{h_{\ell}}(\mathbf y)\|^2\Big|\\
&\ \ + 2\Big|\langle V(\overrightarrow{h_{\ell}}(\mathbf x)), V(\overrightarrow{g_{\ell-1}}(\mathbf x))\rangle_{\mathbb R^{2d+2}} - \langle V(\overrightarrow{h_{\ell}}(\mathbf y)), V(\overrightarrow{g_{\ell-1}}(\mathbf y))\rangle_{\mathbb R^{2d+2}}\Big|\\
& \leq  2\varepsilon C_{11} C_{12} \|\overrightarrow{h_{\ell}}(\mathbf x)-\overrightarrow{h_{\ell}}(\mathbf y)\|\\
& \ \ \
+ 2\Big|\langle V(\overrightarrow{h_{\ell}}(\mathbf x)), V(\overrightarrow{g_{\ell-1}}(\mathbf x)-\overrightarrow{g_{\ell-1}}(\mathbf y) )\rangle_{\mathbb R^{2d+2}} \Big|\\
&\ \ +2\Big|
\langle V(\overrightarrow{h_{\ell}}(\mathbf x) -   \overrightarrow{h_{\ell}}(\mathbf y)), V(\overrightarrow{g_{\ell-1}}(\mathbf y))\rangle_{\mathbb R^{2d+2}}\Big|\\
&=:J_1+J_2+J_3.
\end{split}\end{equation*}
Applying~\eqref{c7}, ~\eqref{c9}, and~\eqref{eq:tau_infty}, we obtain
\begin{equation*}\begin{split}
J_1&\leq  2\varepsilon C_{11} C_{12}  \sum_{Q\in \mathcal Q_{2^{-\ell}}} \lambda_Q \|\overrightarrow{b_{Q}}(\mathbf x)
-\overrightarrow{b_{Q}}(\mathbf y)\|\\
&\leq 2\varepsilon C_{11}^2 C_{12}  \sum_{Q\in\mathcal Q_{2^{-\ell}} } \lambda_Q \frac{\|\mathbf x-\mathbf y\|}{2^{-\ell}} \Big(1+
\frac{d(\mathbf x,\mathbf y)}{2^{-\ell}}\Big)^{N_1-1}
\Big(1+\frac{\|\mathbf z_Q- \mathbf x\|}{2^{-\ell}}\Big)^{-1}
\Big(1+\frac{d(\mathbf z_Q, \mathbf x)}{2^{-\ell}}\Big)^{-N_1+1}\\
&\leq  2\varepsilon C_{11}^2 C_{12}  (1+A_0)^{N_1-1} \frac{\|\mathbf x-\mathbf y\|}{2^{-\ell}} \boldsymbol\tau_\ell(\mathbf x)\leq  2 C_{11}^2 C_{12}^2
(1+A_0)^{N_1-1}  \frac{\|\mathbf x-\mathbf y\|}{2^{-\ell}} \varepsilon^2 ,
\end{split}
\end{equation*}
since $d(\mathbf x,\mathbf y)\leq \| \mathbf x-\mathbf y\|\leq A_0 2^{-\ell}$. 

To estimate $J_2$,  we use~\eqref{h_l}, the induction hypothesis~\eqref{c11}, and~\eqref{eq:tau_infty} to obtain
\begin{equation*}
\begin{split}
J_2\leq 2\varepsilon C_{11}C_{12} A_1 \frac{\|\mathbf x-\mathbf y\|}{2^{-\ell+1}} \Big(1+
\frac{d(\mathbf x,\mathbf y)}{2^{-\ell+1}}\Big)^{N_1-1}\boldsymbol\sigma_{\ell-1}(\mathbf x)
\leq C_{11}C_{12}^2(1+A_0)^{N_1-1} A_1 \frac{\|\mathbf x-\mathbf y\|}{2^{-\ell}}\varepsilon^2 .
\end{split}
\end{equation*}
In order to estimate $J_3$, we apply~\eqref{c7},  {together with~\eqref{c10}} and get 
\begin{equation*}\begin{split}
J_3&=2\Big|
\big\langle \sum_{Q\in\mathcal Q_{2^{-\ell}}}\lambda_Q V( \overrightarrow{b_{Q}}(\mathbf x) -   \overrightarrow{b_{Q}}(\mathbf y)), V(\overrightarrow{g_{\ell-1}}(\mathbf y)-\overrightarrow{g_{\ell-1}}(\mathbf z_Q))\big\rangle_{\mathbb R^{2d+2}}\Big|.\\
\end{split}\end{equation*}

Then utilizing ~\eqref{c9'}, {the induction hypothesis ~\eqref{c11}, and the definition of $\boldsymbol \tau_\ell$,} we have
\begin{equation*}\begin{split}
J_3&\leq 2
 \sum_{Q\in\mathcal Q_{2^{-\ell}}}\lambda_Q 
C_{11} \frac{\|\mathbf x-\mathbf y\|}{2^{-\ell}}
\Big(1+\frac{\|\mathbf z_Q- \mathbf y\|}{2^{-\ell}}\Big)^{-1}
\Big(1+\frac{d(\mathbf z_Q, \mathbf y)}{2^{-\ell}}\Big)^{-2 N_1+1}\\
& \hskip2cm \times A_1 \frac{\|\mathbf y-\mathbf z_Q\|}{2^{-\ell+1}}  \Big(1+\frac{d(\mathbf y,\mathbf z_Q)}{2^{-\ell}}\Big)^{N_1-1} \boldsymbol\sigma_{\ell-1}(\mathbf y)\\
&\leq C_{11}A_1\frac{\|\mathbf x-\mathbf y\|}{2^{-\ell}} \boldsymbol\tau_\ell(\mathbf y)\boldsymbol\sigma_{\ell-1}(\mathbf y)\\
&\leq   C_{11}C_{12}^2 A_1\frac{\|\mathbf x-\mathbf y\|}{2^{-\ell}} \varepsilon^2,
\end{split} \end{equation*}
where in the last inequality  we have used Lemma \ref{lem:Ch-Gell12}.
So the proof of~\eqref{G4} is complete.

We are now in a position to finish the proof of~\eqref{c13}. By the definition of $\overrightarrow{E_{\ell}}$ and~\eqref{G1},
\begin{equation*}
\begin{split}
\|\overrightarrow{E_{\ell}}(\mathbf x)-\overrightarrow{E_{\ell}}(\mathbf y)\|
&=\Big\| \Big(\frac{\overrightarrow{G_{\ell}}(\mathbf x)}{\|\overrightarrow{G_{\ell}}(\mathbf x)\|} -  \overrightarrow{G_{\ell}}(\mathbf x)\Big)
-\Big(\frac{\overrightarrow{G_{\ell}}(\mathbf y)}{\|\overrightarrow{G_{\ell}}(\mathbf y)\|} -  \overrightarrow{G_{\ell}}(\mathbf y)\Big)\Big\|\\
&= \Big\| \Big(\overrightarrow{G_{\ell}}(\mathbf y)-\overrightarrow{G_{\ell}}(\mathbf x)\Big)
\Big(1-\frac{1}{\| \overrightarrow{G_{\ell}}(\mathbf x)\|}\Big)+\overrightarrow{G_{\ell}}(\mathbf y)
\Big(\frac{1}{\| \overrightarrow{G_{\ell}}(\mathbf x)\|}-\frac{1}{\| \overrightarrow{G_{\ell}
}(\mathbf y)\|}\Big)
\Big\|\\
&\leq  (1-C_{14}\varepsilon)^{-1}\Big\|\overrightarrow{G_{\ell}}(\mathbf x) -\overrightarrow{G_{\ell}}(\mathbf y)\Big\|\cdot \Big|1-\|\overrightarrow{G_{\ell}}(\mathbf x)\|\Big| \\
& \ \ +(1-C_{14}\varepsilon)^{-1}\Big|\|\overrightarrow{G_{\ell}}(\mathbf x) \| -\|\overrightarrow{G_{\ell}}(\mathbf y)\|\Big|\\
&\leq A_1  (1+A_0)^{N_1-1} \frac{\|\mathbf x-\mathbf y\|}{2^{-\ell}}\boldsymbol\sigma_\ell(\mathbf x)C_{14}\varepsilon + C_{15}(1-C_{14}\varepsilon)^{-1}\frac{\|\mathbf x-\mathbf y\|}{2^{-\ell}}\varepsilon^2 ,
\end{split}
\end{equation*}
where in the last inequality for the first summand we have used~\eqref{eq91}, while for the second
one we have applied~\eqref{G4}.
Now  from Lemma \ref{lem:Ch-Gell12} we obtain ~\eqref{c13} with
\begin{equation*}
   A_3=A_1(1+A_0)^{N_1-1}C_{12}C_{14}+\frac{100}{99}C_{15}. 
\end{equation*}

Thus the construction of the functions $\overrightarrow{h_{\ell}}$, $\overrightarrow{g_{\ell}}$ and $\overrightarrow{E_{\ell}}$
satisfying~\eqref{c1}--\eqref{c13}  conducted in Steps 1--4  is complete.

{\bf Step 5: proofs of}~\eqref{c5}--\eqref{c4}. Set $N_2:=N_1-\mathbf N$. Then $\mathbf N+1 < N_2<N_1$  (see \eqref{eq:constantsN}).  First, we check  that there is $C_{16}>0$ such that for {all} $P,Q\in\mathcal Q$, $\ell(P)\leq \ell (Q)$, one has 
\begin{equation}\label{product}
\Big|\int_{\mathbb{R}^N}  \langle \overrightarrow{b_{Q}}(\mathbf x),\overrightarrow{b_{P}}(\mathbf x)\rangle_{\mathbb C^{d+1}}\, dw(\mathbf x)\Big|\leq C_{16} \frac{\ell (P)}{\ell (Q)} w(P)\Big(1+\frac{d(\mathbf z_Q,\mathbf z_P)}{\ell (Q)}\Big)^{{-N_2}} . 
\end{equation}
 Indeed,   using~\eqref{eq:b_Q_int0} and then~\eqref{c9} together with ~\eqref{c8} and the Cauchy-Schwarz inequality, we get  
\begin{equation*}
    \begin{split}
        \Big| &\int_{\mathbb{R}^N}  \langle \overrightarrow{b_{Q}}(\mathbf x), \overrightarrow{b_{P}}(\mathbf x)\rangle_{\mathbb C^{d+1}}\, dw(\mathbf x)\Big|=\Big|\int_{\mathbb{R}^N}  \langle \overrightarrow{b_{Q}}(\mathbf x)- \overrightarrow{b_{Q}}(\mathbf z_P),\overrightarrow{b_{P}}(\mathbf x)\rangle_{\mathbb C^{d+1}}\, dw(\mathbf x)\Big|\\
        &\leq C_{11}^2\int_{\mathbb{R}^N} \frac{\|\mathbf x-\mathbf z_P\|}{\ell (Q)} \Big(1+\frac{\| \mathbf x-\mathbf z_Q\|}{\ell(Q)}\Big)^{-1}  \Big(1+\frac{d( \mathbf x,\mathbf z_P)}{\ell(Q)}\Big)^{N_1-1}  \Big(1+\frac{d( \mathbf x, \mathbf z_Q)}{\ell(Q)}\Big)^{-N_1+1}  \\
        &\ \ \ \ \ \times  \Big(1+\frac{\|\mathbf x-\mathbf z_P\|}{\ell(P)}\Big)^{-1}  \Big(1+\frac{d(\mathbf x, \mathbf z_P)}{\ell(P)}\Big)^{-2N_1}\, dw(\mathbf x)\\
        &\leq C_{11}^2 \int_{\mathbb{R}^N} \frac{\ell(P)}{\ell (Q)} \Big(1+\frac{d( \mathbf x,\mathbf z_Q)}{\ell(Q)}\Big)^{-N_1}\Big(1+\frac{d(\mathbf x,\mathbf z_P)}{\ell(P)}\Big)^{-N_1-1}\, dw(\mathbf x)\\
        &\leq C_{11}^2  \int_{\mathbb{R}^N} \frac{\ell(P)}{\ell (Q)} \Big(1+\frac{d(\mathbf x,\mathbf z_Q)}{\ell(Q)}\Big)^{-N_2}\Big(1+\frac{d( \mathbf x,\mathbf z_P)}{\ell(Q)}\Big)^{-N_2}
        \Big(1+\frac{d(\mathbf x,\mathbf z_P)}{\ell(P)}\Big)^{-\mathbf N-1}
        \, dw(\mathbf x)\\
        &\leq C_{11}^2 \frac{\ell(P)}{\ell (Q)} \Big(1+\frac{d(\mathbf  z_P,\mathbf z_Q)}{\ell(Q)}\Big)^{-N_2}\int_{\mathbb{R}^N}  \Big(1+\frac{d(\mathbf x,\mathbf z_P)}{\ell(P)}\Big)^{-\mathbf N-1}
        \, dw(\mathbf x)\\
       &\leq  CC_{11}^2  \frac{\ell (P)}{\ell (Q)} w(P)\Big(1+\frac{d(\mathbf z_Q,\mathbf z_P)}{\ell (Q)}\Big)^{{-N_2}},
    \end{split}
\end{equation*}
where in the fourth inequality we have used the relation 
 $$\Big(1+\frac{d(\mathbf x,\mathbf z_Q)}{\ell(Q)}\Big)^{-N_2}\Big(1+\frac{d( \mathbf x,\mathbf z_P)}{\ell(Q)}\Big)^{-N_2} \leq \Big(1+\frac{d(\mathbf z_P,\mathbf z_Q)}{\ell(Q)}\Big)^{-N_2},$$ 
 while 
in the last inequality we have applied~\eqref{eq:finite_integral} with $\varepsilon=1$. Thus~\eqref{product} is verified.

Now for non-negative integers $s_1\leq s_2$, let $\mathcal Q_{s_1,s_2}=\bigcup_{\ell=s_1}^{s_2}\mathcal Q_{2^{-\ell}}$. Set $D_0=[0,1)$, $D_j=[2^{j-1},2^{j})$, $j\geq 1$. 
In virtue of~\eqref{product},
\begin{equation*}\begin{split}
\Big\|\sum_{\ell=s_1}^{s_2} \overrightarrow{h_{\ell}}\Big\|_{L^2(dw)}^2
&= \sum_{Q\in \mathcal Q_{s_1,s_2}}\sum_{P\in \mathcal Q_{s_1,s_2}} \lambda_Q{\lambda}_P
\int_{\mathbb{R}^N} \langle \overrightarrow{b_{Q}}(\mathbf x),\overrightarrow{b_{P}}(\mathbf x)\rangle_{\mathbb C^{d+1}}\,  dw(\mathbf x)\\
&\leq 2C_{16}\sum_{Q\in \mathcal Q_{s_1,s_2}} \sum_{\substack{P\in \mathcal Q_{s_1,s_2}\\ \ell (P)\leq \ell (Q)}}
\lambda_Q\lambda_P  \frac{\ell (P)}{\ell (Q)} w(P)\Big(1+\frac{d(\mathbf z_Q,\mathbf z_P)}{\ell (Q)}\Big)^{-N_2}\\
&\leq 2C_{16}\sum_{n=0}^{ s_2-s_1} \sum_{Q\in \mathcal Q_{s_1,s_2}} \sum_{\substack {P\in \mathcal Q_{s_1,s_2}\\ \ell (P)=2^{-n} \ell (Q)}}
\lambda_Q\lambda_P  \frac{\ell (P)}{\ell (Q)} w(P)\Big(1+\frac{d(\mathbf z_Q,\mathbf z_P)}{\ell (Q)}\Big)^{-N_2}\\
&\leq  2C_{16}
\sum_{n=0}^{{ s_2-s_1}} \sum_{j=0}^\infty \sum_{Q\in \mathcal Q_{s_1,s_2}} \sum_{\substack{P\in \mathcal Q_{s_1,s_2}\\ \ell (P)=2^{-n} \ell (Q) \\ \ell (Q)^{-1}d(\mathbf z_Q,\mathbf z_P) \in D_j}}
\lambda_Q\lambda_P  w(P) 2^{-n}2^{-jN_2}.\\
 \end{split}\end{equation*}
Applying  the Cauchy-Schwarz inequality two times, we obtain 
 \begin{equation*}\begin{split}
&\Big\|\sum_{\ell=s_1}^{s_2} \overrightarrow{h_{\ell}}\Big\|_{L^2(dw)}^2\\
&\leq  2C_{16}\sum_{n=0}^{ s_2-s_1} \sum_{j=0}^\infty  2^{-n}2^{-jN_2}
\Big\{\sum_{Q\in \mathcal Q_{s_1,s_2}}\lambda_Q^2w(Q)\Big\}^{1\slash 2}
\Bigg\{\sum_{Q\in\mathcal Q_{s_1, s_2}}\Big(\sum_{\substack{P\in\mathcal Q_{s_1,s_2}\\ \ell (P)=2^{-n} \ell (Q) \\
 \ell (Q)^{-1}d(\mathbf z_Q,\mathbf z_P)\in D_j}} \lambda_Pw(P)\Big)^{2}\frac{1}{w(Q)}\Bigg\}^{1\slash 2}\\
&\leq  2C_{16}\sum_{\substack{n\geq 0\\ j\geq 0}}
2^{-n}2^{-jN_2}
\Big\{\sum_{Q\in \mathcal Q_{s_1,s_2}} \lambda_Q ^2w(Q)\Big\}^{1\slash 2}\\
&\hskip2cm \times
\Bigg\{\sum_{Q\in \mathcal Q_{s_1,s_2}}\Bigg(\sum_{\substack{ P\in\mathcal Q_{s_1,s_2}\\ \ell (P)=2^{-n} \ell (Q) \\
 \ell (Q)^{-1} d(\mathbf z_Q,\mathbf z_P)\in D_j}}\lambda_P^2w(P)\Bigg)
 \Bigg(\sum_{\substack{ P\in\mathcal Q_{s_1,s_2}\\ \ell (P)=2^{-n} \ell (Q) \\
 \ell(Q)^{-1}d(\mathbf z_Q,\mathbf z_P)\in D_j}}w(P)\Bigg)\frac{1}{w(Q)}\Bigg\}^{1\slash 2}.
\end{split}\end{equation*}
Observe that there is a constant $C_{17}\geq 1$, which depends on $N, R, k$ such that for any integer $n\geq 0$ and any cube   $Q\in \mathcal Q$, one has 
\begin{equation}\label{eq:C_19_in}
    \sum_{\substack{ P\in \mathcal Q\\ \ell (P)=2^{-n} \ell (Q) \\
 d(\mathbf z_Q,\mathbf z_P)\leq 2^j\ell (Q)}}w(P)\leq C_{17} 2^{j\mathbf N}w(Q),
\end{equation}
$$$$
 so, with $C_{18}=2C_{16}\sqrt{C_{17}}$, we have 
 \begin{equation}\label{eq:C_20_in}\begin{split}
&\|\sum_{\ell=s_1}^{s_2} \overrightarrow{h_{\ell}}\|_{L^2(dw)}^2\\
&\leq  C_{18}\sum_{\substack{n\geq 0\\ j\geq 0}}
2^{-n}2^{-j(N_2-\mathbf N\slash 2)}
\Big\{\sum_{Q\in \mathcal Q_{s_1,s_2}} \lambda_Q^2w(Q)\Big\}^{1\slash 2}
\Bigg\{\sum_{Q\in \mathcal Q_{s_1,s_2}}\sum_{\substack{P\in\mathcal Q_{s_1,s_2}\\ \ell (P)=2^{-n} \ell (Q) \\
 \ell (Q)^{-1}d(\mathbf z_Q,\mathbf z_P)\in D_j }} \lambda_P^2w(P)\Bigg\}^{1\slash 2}.
\end{split}\end{equation}

 Note that there is  $C_{19}\geq 1$, which depends on $N, R, k$, such that 
for all $j,n\geq 0$ and every $P\in \mathcal Q$ the number of cubes $Q\in \mathcal Q$  such that
$\ell (P)=2^{-n}\ell (Q)$, $d(\mathbf z_Q,\mathbf z_P)\leq 2^{j}\ell (Q)$ is bounded by {$C_{19}2^{jN} \leq C_{19}2^{j\mathbf N}$.}  Therefore, { with $C_{20}=C_{18}\sqrt{C_{19}}$}
\begin{equation}\begin{split}\label{sumh}
&\Big\|\sum_{\ell=s_1}^{s_2} \overrightarrow{h_{\ell}}\Big\|_{L^2(dw)}^2\\
& \leq  C_{20}\sum_{\substack{n\geq 0\\ j\geq 0}}
2^{-n}2^{-j(N_2-\mathbf N)}
\Big\{\sum_{Q\in \mathcal Q_{s_1,s_2}}\lambda_Q^2w(Q)\Big\}^{1\slash 2}
\Big\{\sum_{P\in \mathcal Q_{s_1,s_2}}\lambda_P^2w(P) \Big\}^{1\slash 2}\\
&\leq  {4C_{20} } \sum_{Q\in \mathcal Q_{s_1,s_2}} \lambda_Q^2w(Q) \leq {4C_{20}} \int_{2^{-s_{2}}}^{2^{-s_1+1}} \int_{\mathbb{R}^N}  |{\boldsymbol \phi}_t*{\boldsymbol \phi}_t*f(\mathbf y)|^2\,dw(\mathbf y)\frac{dt}{t},
\end{split}\end{equation}
where in the last inequality we have used~\eqref{eq:lambda}. From the Plancherel's equality~\eqref{eq:Plancherel} for the  Dunkl transform we easily conclude that $\|\sum_{\ell=s_1}^{s_2} \overrightarrow{h_{\ell}}\|_{L^2(dw)}^2\to 0$ as $s_1,s_2\to\infty$, since $f\in L^2(dw)$. 
Thus~\eqref{sumh} implies that $\sum_{l\geq 0} \overrightarrow{h_l}(\mathbf x)$
converges in $L^2(dw)$ {and by  
 Lemma~\ref{lem:BMO_compactly_supported}, }
\begin{equation}\begin{split}\label{sumb}
\Big\|\sum_{l\geq 0} \overrightarrow{h_l}\Big\|_{L^2(dw)}&=\Big\| \sum_{Q\in\mathcal Q, \ \ell(Q)\leq 1} \lambda_Q \overrightarrow{b_Q}\Big\|_{L^2(dw)}\leq {2\sqrt{C_{20}} } \| f\|_{L^2(dw)}\\
&\leq {2\sqrt{C_{20}} C_{J{\text{-}}N,2}}\cdot  \| f\|_{{\rm BMO}(\mathbf X)}
w(B(0, 1))^{1\slash 2}.
\end{split}\end{equation}

Fix $Q\in \mathcal Q$. Let $m\in \mathbb Z$ be such that $2^{-m}=\ell (Q)$. From~\eqref{c12} and Lemma \ref{lem:Ch-Gell12}, we obtain 
\begin{equation}\begin{split}\label{BMO1}
\int_{{2}Q} \sum_{j\geq {\max (m,0)}} \|\overrightarrow{E_j}(\mathbf x)\|\, dw(\mathbf x)
& \leq A_2 \int_{{2}Q}\sum_{ j\geq { \max (m,0)}} \boldsymbol\sigma_j(\mathbf x )\boldsymbol\tau_j(\mathbf x)\, dw(\mathbf x)\\
& \leq A_2C_{12}
w(Q)\| f\|^2_{{\rm BMO}(\mathbf{X})}.
\end{split}\end{equation}

Applying ~\eqref{BMO1} to $Q\in\mathcal Q$ such that $\ell (Q)=2^0=1$, we conclude that the series $\sum_{j\geq 0} \overrightarrow{E_j}(\mathbf x)$ converges locally in $L^1(dw)$ to a function $\overrightarrow E^0(\mathbf x)$.

On the other hand, we deduce  from~\eqref{c13} that for $\mathbf x,\mathbf y\in {2} Q$, $\ell (Q)=2^{-m}$, { $m\in\mathbb Z$, }
we have the following estimate  on the finite sum:
\begin{equation}\label{BMO2}
\sum_{\min(m,0)\leq j<m}  \|\overrightarrow{E_j}(\mathbf x)-\overrightarrow{E_j}(\mathbf y)\|\leq {A_3} \|\mathbf x-\mathbf y\| \| f\|_{{\rm BMO}(\mathbf X)}^2
 \sum_{\min(m,0)\leq j<m}2^{j}\leq {A_3\sqrt{N}}\| f\|^2_{{\rm BMO}(\mathbf X)},
\end{equation}
where the left  side of~\eqref{BMO2} is  understood to be zero if $m\leq 0$.  
The estimates~\eqref{BMO1} together with~\eqref{BMO2} imply that 
$$ \frac{1}{w({ 2} Q)} \int_{{2}Q}  \Big\| \overrightarrow E^0(\mathbf x)-\sum_{\min(m,0)\leq j<m} \overrightarrow{E_j}(\mathbf z_Q)\Big\|\, dw(\mathbf x)\leq {(A_3\sqrt{N} +C_{12}A_2)} \| f\|_{{\rm BMO}(\mathbf X)}^2,$$
which together with the doubling property give  ~\eqref{c5} and~\eqref{c6} with 
$$A_4= A_3\sqrt{N} +C_{12}A_2.$$

Since $\sum_{j=0}^{\ell} (\overrightarrow{h_j}+\overrightarrow{E_j})=\overrightarrow {g_{\ell}}-\overrightarrow{g_{-1}}$,
we conclude~\eqref{c4} from~\eqref{c5},~\eqref{sumb},  and~\eqref{c2}.

\ 

{\bf Step 6: completion of the proof of Theorem~\ref{KeyTheorem}.}
Having~\eqref{c1}--\eqref{c4} already proved, we are in a position to finish the proof of Theorem~\ref{KeyTheorem}.

It follows from~\eqref{G3},~\eqref{eq:sigma_infty}, and the definition of $\boldsymbol\tau_\ell(\mathbf x)$ (see Section~\ref{sec:auxiliary})  that
\begin{equation}\label{eq:ELL}\begin{split}
\|\overrightarrow{E_{\ell}}(\mathbf x)\| &
\leq C_{12}A_2\| f\|_{{\rm BMO}(\mathbf X)} \sum_{Q\in\mathcal Q_{2^{-\ell}} } \lambda_Q \Big(1+\frac{d(\mathbf x,\mathbf z_Q) }{\ell (Q)}\Big)^{-N_0}.
\end{split}
\end{equation}
Recall that $\text{\rm supp}\, f\subseteq B(0,1)$.  Hence $\lambda_Q=0$, if $d(\mathbf z_Q, 0)>4\sqrt{N}{=A_0/8}$ and {$\ell (Q)\leq 1$.} Moreover, {there is a constant $C_{21} \geq 1$, which depends only on $N,R$}, such that 
\begin{equation}\label{eq:number_Q}
    \#\{Q\in\mathcal Q_{2^{-\ell}} :  \  \lambda_Q\ne 0\}\leq {C_{21}}2^{\ell N}, \quad { \text{\rm for } \ell\geq 0.}
\end{equation}
 Hence, if $\|\mathbf x\|=d(\mathbf x,0) >A_0/2$ and $\lambda_Q\ne 0$, then $2^{-1}d(\mathbf x,0)\leq d(\mathbf x,\mathbf z_Q)\leq 2 d(\mathbf x,0)$ and   {from \eqref{eq:ELL},   \eqref{eq:lambda_bound}, and \eqref{eq:constantsN}},  we conclude that 
 \begin{equation}\label{sumE}\begin{split}
  \|\overrightarrow{E^{0}}(\mathbf x)\|&\leq 
\sum_{\ell\geq 0}\|\overrightarrow{E_{\ell}}(\mathbf x)\| 
\leq {C_{12}A_2 C_{10} C_{21}2^{N_0}}\| f\|_{{\rm BMO}(\mathbf X)}^2 \sum_{l\geq 0}2^{\ell N}\Big(1+\frac{d(\mathbf x, 0)}{ 2{^{-\ell}}}\Big)^{-N_0}\\
&\leq   { 2C_{12}A_2 C_{10} C_{21}2^{N_0}}\| f\|_{{\rm BMO}(\mathbf X)}^2d(\mathbf x,0)^{-N_0} {= C_{22} \|\mathbf x\|^{-N_0}\| f\|_{{\rm BMO}(\mathbf X)}^2}.
\end{split}\end{equation}

Using~\eqref{c3}, \eqref{c4}, and \eqref{c5},  we write 
\begin{equation}\label{eq:sum_h_ELL}
\begin{split}
\sum_{\ell\geq 0} \overrightarrow{h_{\ell}}
& =\sum_{l\geq 0} (\overrightarrow{g_{\ell}}-\overrightarrow{g_{\ell-1}}) -\sum_{l\geq 0} \overrightarrow{E_{\ell}}=  \overrightarrow g-\overrightarrow{g_{-1}}- \overrightarrow{E^{0}}\\
&= \overrightarrow g-\overrightarrow{g_{-1}} -\overrightarrow{E^{0}} \chi_{B(0, A_0/2)^c} - \overrightarrow{E^{0}}\chi_{B(0, A_0/2)}.
\end{split}
\end{equation}
From~\eqref{c2} and~\eqref{sumE} we have
 \begin{equation}\label{norm1}\| \overrightarrow g-\overrightarrow{g_{-1}}- \ \overrightarrow{E^{0}}\chi_{B(0, A_0/2)^c}\|_{L^\infty} \leq 2 + { C_{22}} \| f\|_{{\rm BMO}(\mathbf X)}^2.
 \end{equation}
Now using~\eqref{sumE} combined with~\eqref{c6}, we get
\begin{equation}\label{norm2} \Big\|\overrightarrow{E^{0}}\chi_{B(0, A_0/2)}\Big\|_{{\rm BMO}(\mathbf X)}{ \leq 
\|\overrightarrow{E^{0}}\|_{{\rm BMO}(\mathbf X)} + \| \overrightarrow{E^{0}}\chi_{B(0, A_0/2)^c}\|_{L^\infty} 
}\leq { (A_4+C_{22}) } \| f\|_{{\rm BMO}(\mathbf X)}^2
\end{equation}
and, consequently, {by Lemma~\ref{lem:BMO_compactly_supported}, }
\begin{equation}\label{norm3}
\Big\| \overrightarrow{E^{0}}\chi_{B(0,A_0/2)}\Big\|_{L^2(dw)}\leq { C_{23}} w(B(0,1))^{1\slash 2} \| f\|_{{\rm BMO}(\mathbf X)}^2, 
\end{equation}
{with $C_{23}=C_{J{\text{-}}N,2}(A_0/2)^{\mathbf N/2}(A_4+C_{22})$. }
Finally, from ~\eqref{eq:sum_h_ELL}, ~\eqref{sumb}, and~\eqref{norm3}, we obtain  
\begin{equation}\label{norm4}\begin{split}
 \Big\| \overrightarrow g-\overrightarrow{g_{-1}} &- \overrightarrow{E^{0}}\chi_{B(0, A_0/2)^c} \Big\|_{L^2(dw)} = \Big\| \sum_{\ell\geq 0}  \overrightarrow{h_{\ell}}+ \overrightarrow{E^{0}}\chi_{B(0, A_0/2)}\Big\|_{L^2(dw)} \\
&\leq {(2\sqrt{C_{20}} C_{J{\text{-}}N,2}+C_{23})}w(B(0,1))^{1\slash 2}\Big(\| f\|_{{\rm BMO}(\mathbf X)}+\| f\|_{{\rm BMO}(\mathbf X)}^2\Big).
\end{split}\end{equation}

Recall that for {$\overrightarrow f =(f_0, f_1,f_2,...,f_d)$, we denote $\overrightarrow {S^*} \circ \overrightarrow f=\sum_{j=0}^d \mathbf S^{\{j\}*}f_j$.} Let
\begin{equation}\label{eq:function_F} F=\overrightarrow {S^*} \circ \Big\{  \overrightarrow{E^{0}}\chi_{B(0, A_0/2)}\Big\}. 
\end{equation}
  The H\"ormander multiplier theorem~\cite[Theorem 1.2]{DzHe-JFA}  asserts that the operators $\mathbf S^{\{j\}}$ are bounded on $H^1_{\rm Dunkl}$. Hence, by ~\eqref{norm2} and duality arguments, {there is a constant $C_{24} \geq 1$ which depends on the system $\overrightarrow{S}$ such that for all $f$ satisfying $\| f\|_{{\rm BMO}(\mathbf X)}=\varepsilon<\varepsilon_0$, $\text{supp}\, f\subset B(0,1)$), we have } 
  \begin{equation}\label{eq:FinBMO}
      \| F\|_{{\rm BMO}(\mathbf X)}\leq C_{24}\| 
\overrightarrow{E^{0}}\chi_{B(0, A_0/2)} \|_{{\rm BMO}(\mathbf X)}\leq { C_{24}(A_4+C_{22})w(B(0,1))^{1/2}} \| f\|_{{\rm BMO}(\mathbf X)}^2. 
  \end{equation}
   { Moreover,  { using \eqref{eq:S-bound}, \eqref{eq:doubling},  \eqref{norm3}, and the Cauchy-Schwarz inequality,   we get the following bound on $|F(\mathbf x)|$ for $ \|\mathbf x\|=d(\mathbf x,0)>A_0$: 
  \begin{equation}\label{eq:bound_F}\begin{split}
 |F(\mathbf x)|&\leq C_{0,0} \sup_{\|\mathbf y\|\leq A_0/2} w(B(\mathbf x,d(\mathbf x,\mathbf y)))^{-1} \| \overrightarrow{E^{0}}\chi_{B(0,A_0/2)}\|_{L^1(dw)}
  \leq C_{25} w(B(0,1))\|\mathbf x\|^{-\mathbf N} \| f\|^2_{{\rm BMO}(\mathbf X)}.
  \end{split}  \end{equation}
 with $C_{25}\geq 1$ depending on $N,R,k,\overrightarrow{S}$ but   independent of $f$, provided $\| f\|_{{\rm BMO}(\mathbf X)}=\varepsilon$, $\text{\rm supp}\, f\subset B(0,1)$. }
{Thus from \eqref{eq:bound_F} and \eqref{eq:FinBMO}, we conclude that  there is a constant $C_{26} \geq 1$ independent of $f\in {\rm BMO}(\mathbf X)$ satisfying $\text{supp}\, f\subset B(0,1)$, $\|f\|_{{\rm BMO}}=\varepsilon$, such that}

{\begin{equation}
    \label{eq:L2Ac}
    \| F\chi_{B(0,A_0)^c}\|_{L^2(dw)}\leq  C_{26} \| f\|_{{\rm BMO}(\mathbf X)}^2=C_{26}\varepsilon^2, 
\end{equation} }

{\begin{equation}
    \label{eq:L_inftyAc}
    \| F\chi_{B(0,A_0)^c}\|_{L^\infty}\leq  C_{26} \| f\|_{{\rm BMO}(\mathbf X)}^2=C_{26}\varepsilon^2. 
\end{equation} }

\begin{equation}\label{eq:FA_0} 
\begin{split}
\| F\chi_{B(0, A_0)}\|_{{\rm BMO}(\mathbf X)}&\leq   \| F\chi_{B(0,A_0)^c}\|_{L^\infty} +\|F\|_{{\rm BMO}(\mathbf X)} \\
&\leq (C_{26}+C_{24}(A_4+C_{22})w(B(0,1))^{1/2})\|f\|_{{\rm BMO}(\mathbf X)}^2\\
& =(C_{26}+C_{24}(A_4+C_{22})w(B(0,1))^{1/2})\varepsilon^2,
\end{split}
\end{equation}

{ Let us denote  $\overrightarrow{g}'=( g_0-F\chi_{B(0, A_0)^c},0,\dots,0)$.  {Let us recall that $\mathbf S^{\{0\}}={\rm Id}$.}  }Put 
 \begin{equation}\label{eq:functions}\begin{cases}
 & \overrightarrow{\widetilde g}=\overrightarrow{g}'+\overrightarrow g-\overrightarrow{g_{-1}}- \overrightarrow{E^0}\chi_{B(0, A_0/2)^c}= \overrightarrow{g}'+\sum_{\ell\geq 0}  \overrightarrow{h_{\ell}}+ \overrightarrow{E^0}\chi_{B(0, A_0/2)}, \\
& f_1=-F\chi_{B(0, A_0)}.
\end{cases}\end{equation}
}
{ 
Then, by \eqref{norm1}, { \eqref{eq:g_0infty},  and \eqref{eq:L_inftyAc} we obtain  
\begin{equation*}\label{eq:ggg} 
    \begin{split}
        \sum_{j=0}^d \|\widetilde g_j\|_{L^\infty} &\leq \| g_0\|_{L^{\infty}}+\| F\chi_{B(0,A_0)^c}\|_{L^\infty}  +  \sqrt{d+1} \| \overrightarrow g-\overrightarrow{g_{-1}}- \ \overrightarrow{E^{0}}\chi_{B(0, A_0/2)^c}\|_{L^\infty} \\
        &\leq  C_{13}\| f\|_{{\rm BMO}(\mathbf X)}+ C_{26}\| f\|_{{\rm BMO}(\mathbf X)}^2+  \sqrt{d+1}(2+C_{22} \|f\|_{{\rm BMO}(\mathbf X)}^2)
    \end{split}
\end{equation*}
which gives 
  \eqref{eq:g_infty}.}  For relations  \eqref{eq:f_1BMO}, see   \eqref{eq:FA_0}. Now   \eqref{eq:g_L2} is exactly  \eqref{norm4} { combined with \eqref{eq:L2Ac} and \eqref{eq:g_0infty}}. Finally, for the decomposition  \eqref{eq:decom_need}, we use \eqref{eq:functions} and  write 
  { 
 \begin{equation*}
     \begin{split}
         \overrightarrow {S^*} \circ \overrightarrow{\widetilde g} +f_1 &= 
(g_0-F\chi_{B(0, A_0)^c})+  \overrightarrow {S^*} \circ \Big(\sum_{\ell \geq 0}\overrightarrow h_{\ell}\Big) + \overrightarrow {S^*} \circ \Big(\overrightarrow{E^0}\chi_{B(0, A_0/2)}\Big) +f_1\\
         &=(g_0-F\chi_{B(0, A_0)^c}) + f_0+F -F\chi_{B(0, A_0)}\\
         &= g_0 +f_0=f,
     \end{split}
 \end{equation*} }
 where in the second equality we have applied   \eqref{c1}, \eqref{eq:f_deco_1}, and \eqref{eq:function_F}.}
\end{proof}
\subsection{Proof of Theorem \ref{UchiyamaDecomposition0}}
\begin{proof}[Proof of Theorem \ref{UchiyamaDecomposition0}]
Let  ${A,A_0\geq 1}$ and $0<\varepsilon_0<1$ be as in {Theorem}~\ref{KeyTheorem}.    Fix  $0<\varepsilon <\varepsilon_0$ such that $A\varepsilon <1$ and $AA_0^{\mathbf N\slash 2} \varepsilon <1$. Let $0\not\equiv f$ be a compactly supported ${\rm BMO}(\mathbf X)$-function. Without loss of generality we may assume that  $\| f\|_{{\rm BMO}(\mathbf X)}=\varepsilon$. Let $r>0$ be  such that  $\text{\rm supp}\, f\subseteq  B(0,r)$. {Recall that $\mathbf S^{\{0\}}={\rm Id}$.}
Decompose $f$ according to Theorem~\ref{KeyTheorem}, i.e.,
\begin{align*}
    f=\sum_{j=1}^d \mathbf S^{\{j\}*} \widetilde g_j +\widetilde g_0+ f_1.
\end{align*}
Set $\widetilde g_j^{\{0\}}:=\widetilde g_j$ for $j=0,1,...,d$. If $f_1 \equiv 0$, we are done. Otherwise we apply
Theorem~\ref{KeyTheorem}
to the function $\varepsilon{f_1}/{\| f_1\|_{{\rm BMO}(\mathbf{X})}}$ and obtain functions
$f_2$, $\widetilde g_j^{\{1\}}$, $j=0,1,...,d$, such that
\begin{equation*}
f_1=\sum_{j=1}^d \mathbf S^{\{j\}*}\widetilde g_j^{\{1\}} + \widetilde g_0^{\{1\}} + f_2,
\end{equation*}

\begin{equation*}
 \sum_{j=0}^d \| \widetilde g_j^{\{1\}}\|_{L^\infty} \leq (2{ \sqrt{d+1}}+A\varepsilon)\frac{\| f_1\|_{{\rm BMO(\mathbf X)}}}{\varepsilon}\leq (2{\sqrt{d+1}}+A\varepsilon)A\varepsilon, 
\end{equation*}

\begin{equation*}
 \| f_2\|_{{\rm BMO}(\mathbf X)}\leq \frac{\| f_1\|_{{\rm BMO(\mathbf X)}}}{\varepsilon} A\varepsilon^2\leq A^2\varepsilon^3, \ \ \text{\rm supp}\, f_2\subseteq B( 0 , A_0^2 r),
\end{equation*}
\begin{equation*}
 \sum_{j=0}^d \| \widetilde g_j^{\{1\}}\|_{L^2(dw)}\leq \frac{\| f_1\|_{{\rm BMO(\mathbf X)}}}{\varepsilon} A w(B(0,A_0r))^{1\slash 2} \varepsilon\leq A^2  w(B(0, A_0r))^{1\slash 2}\varepsilon^2.
\end{equation*}
Continuing this procedure we obtain
sequences of functions $\{\widetilde g^{\{n\}}_j\}_{n \in \mathbb{N}\cup \{0\}}$, $j=0,1,...,d$, and $\{f_n\}_{n \in \mathbb{N}}$ such that
\begin{align*}
    f=\sum_{j=1}^d \mathbf S^{\{j\}*}  \widetilde g_j^{\{0\}} +\widetilde g_0^{\{0\}}+ f_1;
\end{align*}

\begin{align*}
     f_n=\sum_{j=1}^d \mathbf S^{\{j\}*}  g_j^{\{n\}} + g_0^{\{n\}} + f_{n+1}, \quad n=1,2,\dots ; 
\end{align*}
\begin{equation}\label{eq:L_infty_g}
    \sum_{j=0}^d \| \widetilde g_j^{\{n\}}\|_{L^\infty} \leq (2{ \sqrt{d+1}}+A \varepsilon)A^n\varepsilon^n, \quad n=0,1,2,\dots ;
\end{equation}
\begin{equation}\label{eq:1}
    \sum_{j=0}^d \| \widetilde g_j^{\{n\}}\|_{L^2(dw)}\leq A^{n+1}  w(B(0, A_0^nr))^{1\slash 2}\varepsilon^{n+1},\quad n=0,1,2,\dots ;
\end{equation}
\begin{equation}\label{eq:2}
    \| f_{n+1}\|_{{\rm BMO(\mathbf X)}} \leq A^{n+1}\varepsilon^{n+2}, \ \ \text{\rm supp}\, f_{n+1}\subseteq B(0, A_0^{n+1}r), \quad n=0,1,2,\dots .
\end{equation}
Using Lemma~\ref{lem:BMO_compactly_supported} together with~\eqref{eq:measure},~\eqref{eq:1}, and~\eqref{eq:2} we get 
$$ \| f_n\|_{L^2(dw)} \leq C' w(B(0, A_0^n r))^{1\slash 2}
\| f_n\|_{{\rm BMO(\mathbf X)}} \leq
C'' A_0^{\mathbf N n\slash 2} w(B(0,  r))^{1\slash 2}A^n\varepsilon^{n+1}\to 0, \ \ \ \text{as }\ n\to\infty ,$$
$$ \sum_{n=0}^\infty \sum_{j=0}^d \|\widetilde g_j^{\{n\}}\|_{L^2(dw)}
\leq C \sum_{n=0}^\infty A^{n+1} A_0^{\mathbf Nn\slash 2} w(B(0, r))^{1\slash 2} \varepsilon^{n+1}<\infty. $$

Putting $\boldsymbol g_j=\sum_{n=0}^\infty \widetilde g_j^{\{n\}}$  
 for $j=0,1,\ldots,d$ and using~\eqref{eq:L_infty_g}, we complete the proof of Theorem \ref{UchiyamaDecomposition0}.
\end{proof}

\section{Proof of characterization of \texorpdfstring{$H^1_{\rm Dunkl}$}{H1} by systems of singular integrals}
{
\begin{proof}[Proof of Theorem \ref{teo:main_1}]
  Having Theorem \ref{UchiyamaDecomposition0} about the decompositions  of compactly supported ${\rm BMO}(\mathbf X)$-functions, the proof of Theorem~\ref{teo:main_1} goes by standard arguments. 
  For the sake of completeness we provide the details. Assume that $f$ is a complex valued  $L^1(dw)$-function such that $\mathbf S^{\{j\}}f\in L^1(dw)$ for all   $j=1,\dots,d$. Recall that  $\mathbf S^{\{0\}}=\text{\rm Id}$.   Consider the functional 
  \begin{equation}\label{functional:Phi}
      C_c (\mathbf X) \ni \varphi \mapsto \Phi(\varphi)=\int_{\mathbb{R}^N} \varphi (\mathbf x)f(\mathbf x)\, dw(\mathbf x)=  \lim_{t\to 0}\int_{\mathbb{R}^N} \varphi(\mathbf x) (f*h_t)(\mathbf x)\, dw(\mathbf x).
  \end{equation}
Applying Theorem~\ref{UchiyamaDecomposition0}, {there are  $\boldsymbol g_j\in L^2(dw)\cap L^\infty$, $j=0,1,\dots, d$, such that } 
\begin{equation*}
    \varphi =\sum_{j=1}^d \mathbf S^{\{j\}*} \boldsymbol g_j+\boldsymbol g_0=\sum_{j=0}^d \mathbf S^{\{j\}*} \boldsymbol g_j, 
\end{equation*}

\begin{equation}
    \label{eq:sum-g}
    \sum_{j=0}^\infty \| \boldsymbol g_j\|_{L^\infty }\leq C \|\varphi \|_{{\rm BMO}(\mathbf{X})}. 
\end{equation}  Note that $f*h_t\in L^2(dw)$ for all $t>0$, so 
\begin{equation*}\begin{split}
    \Phi(\varphi)&=\lim_{t\to 0} \int_{\mathbb{R}^N} \sum_{j=0}^d \mathbf S^{\{j\}*}\boldsymbol g_j (\mathbf{x})(f*h_t)(\mathbf{x})dw(\mathbf{x}) =\lim_{t\to 0} \sum_{j=0}^d \int_{\mathbb{R}^N} \boldsymbol g_j (\mathbf{x}) \mathbf S^{\{j\}} (f*h_t)(\mathbf{x})\, dw(\mathbf{x})\\
    &= \lim_{t\to 0} \sum_{j=0}^d \int_{\mathbb{R}^N} \boldsymbol g_j (\mathbf{x}) ((\mathbf S^{\{j\}} f)*h_t)(\mathbf{x})\, dw(\mathbf{x})
    =\sum_{j=0}^d \int_{\mathbb{R}^N} \boldsymbol g_j(\mathbf{x}) \mathbf S^{\{j\}}f(\mathbf{x})\, dw(\mathbf{x}).
    \end{split}
\end{equation*}
Hence, by~\eqref{eq:sum-g}, 
\begin{align*}
    |\Phi(\varphi)|\leq C \|\varphi\|_{{\rm BMO}(\mathbf{X})} \sum_{j=0}^d \|\mathbf S^{\{j\}}f\|_{L^1(dw)}.
\end{align*}
{So, $\Phi$ can be extended to a bounded functional on ${\rm VMO}(\mathbf X)$ {(see Section \ref{sec:BMO})} and its norm is controlled from above by $C \sum_{j=0}^d \|\mathbf S^{\{j\}}f\|_{L^1(dw)}$. By Coifman-Weiss \cite{CW} (see also Section \ref{sec:BMO}),  the functional $\Phi$  is represented by integration with a unique   $H^1_{\rm Dunkl}$-function (see \eqref{functional:Phi}).} Hence $f\in H^1_{\rm Dunkl}$ and 
$$\| f\|_{H^1_{\rm Dunkl}}\sim \| \Phi\|\leq C \sum_{j=0}^d \|\mathbf S^{\{j\}}f\|_{L^1(dw)}.$$

    The inverse estimate $\|\mathbf S^{\{j\}}f\|_{L^1(dw)}\leq C \| f\|_{H^1_{\rm Dunkl}}$ is Theorem 1.2 of~\cite{DzHe-JFA}.
\end{proof}

}
\appendix

\section{Proof of Theorem \texorpdfstring{\ref{teo:bounds_conv}}{3.10}}
\label{sec:Appendix}

\begin{proof}[Proof of Theorem \ref{teo:bounds_conv}] 
\ 
{ Thanks to the scaling, in the proof of the theorem we may assume that $0<s\leq t=1$.  }

{\it Proof of~\eqref{teo_a}}. Recall that  $\int_{\mathbb{R}^N} { \phi}_s(\mathbf{z},\mathbf{y})\, dw(\mathbf{z})=0$ for all $s>0$ and $\mathbf{y} \in \mathbb{R}^N$. Applying~\eqref{translations_kernels:a} and~\eqref{translations_kernels:b} of Theorem~\ref{teo:translations_kernels} and~\eqref{eq:supp_transl}, we get 
\begin{equation*}
\begin{split}
    |\psi_1 & *{\phi}_s (\mathbf x,\mathbf y)|=\Big| \int_{\mathbb{R}^N} (\psi_1(\mathbf x,\mathbf z)-\psi_1(\mathbf x,\mathbf y)){\phi}_s(\mathbf z,\mathbf y)\, dw(\mathbf z)\Big|\\
    &\leq C\|\psi\|_{C^{\kappa}}\| \phi\|_{C^{\kappa}}\int_{d(\mathbf z,\mathbf y)\leq s} \frac{\|\mathbf z-\mathbf y\|}{1+\|\mathbf x-\mathbf y\|}w((B(\mathbf x,1))^{-1/2}\big( w(B(\mathbf z,1))^{-1/2}+w(B(\mathbf y,1))^{-1/2}\big)\\
    &\hskip 1cm \times w(B(\mathbf z,s))^{-1} \Big(1+\frac{\|\mathbf z-\mathbf y\|}{s}\Big)^{-1}\, dw(\mathbf z)\\
    &\leq C\|\psi\|_{C^{\kappa}}\| \phi\|_{C^{\kappa}}(1+\|\mathbf{x}-\mathbf{y}\|)^{-1}sw(B(\mathbf x,1))^{-1/2}\\
    &\hskip 1cm \times \int_{d(\mathbf{z},\mathbf{y})\leq s}\big( w(B(\mathbf z,1))^{-1/2}+w(B(\mathbf y,1))^{-1/2}\big)w(B(\mathbf z,s))^{-1}\, dw(\mathbf z). 
\end{split}\end{equation*}
Let us note that, by~\eqref{eq:measure}, $w(B(\mathbf z,1))\sim w(B(\mathbf y,1))$ if $d(\mathbf z,\mathbf y)\leq s\leq 1$. Therefore, 
\begin{equation*}
     |\psi_1  *\phi_s (\mathbf x,\mathbf y)|\leq C\|\psi\|_{C^{\kappa}}\| \phi\|_{C^{\kappa}}(1+\|\mathbf x-\mathbf y\|)^{-1}sw(B(\mathbf x,1))^{-1/2}w(B(\mathbf y,1))^{-1/2}. 
\end{equation*}
Now~\eqref{teo_a} of the theorem follows from the facts that we  $w(B(\mathbf x,1))\sim w(B(\mathbf y,1))$, if $d(\mathbf x,\mathbf y)\leq 2$. Moreover, $ \psi*\phi_s(\mathbf x,\mathbf y)=0$ if $d(\mathbf x,\mathbf y)>2$ { and $0<s\leq 1$} (see \eqref{eq:supp_transl}). 

{\it Proof of~\eqref{teo_b}.}  Note that $\psi*\phi_s(\mathbf x,\mathbf y)=0$ if $d(\mathbf x,\mathbf y)\geq 2$ and $0<s\leq t=1$ (see~\eqref{eq:supp_transl}).  Further,  since $s^m \Delta_k^{m/2}({\eta}_s)(\mathbf x)=\phi_s(\mathbf x)$, 
\begin{equation}\label{eq:conv_psi_phi}
    \psi_1*\phi_s(\mathbf x)=s^m \psi_1*(\Delta_k^{m/2} (\eta_s))({\mathbf{x}})=s^m(\Delta_k^{m/2}\psi_1)*\eta_s(\mathbf x).
\end{equation}
Consequently, applying part~\eqref{translations_kernels:a} of Theorem \ref{teo:translations_kernels}, we conclude that 
\begin{equation*}
    \begin{split}
        |\psi_1&*\phi_s(\mathbf x,\mathbf y)|=s^m\Big| \int_{\mathbb{R}^N} (\Delta_k^{m/2} \psi_1)(\mathbf{x},\mathbf{z})\eta_s(\mathbf z,\mathbf y)\, dw({\mathbf{z}})\Big|\\
        &\leq C\| \psi\|_{C^{\kappa}}\|\eta\|_{C^\kappa}s^m\int_{\mathbb{R}^N} w(B(\mathbf x,1))^{-1} (1+\|\mathbf x-\mathbf z\|)^{-1}\\
        &\hskip 1cm \times w(B(\mathbf y,s))^{-1} (1+\|\mathbf z-\mathbf y\|/s)^{-1} \chi_{[0,s]}(d(\mathbf z,\mathbf y))\, dw(\mathbf z)\\
        &\leq  C\| \psi\|_{C^{\kappa}}\|\eta\|_{C^\kappa}  s^m w(B(\mathbf x,1))^{-1}(1+\| \mathbf x-\mathbf y\|)^{-1}.
    \end{split}
\end{equation*}

{\it Proof of~\eqref{teo_c}.} Applying part~\eqref{teo_a} of Theorem~\ref{teo:bounds_conv}, we have 
$$ |\psi_s*\phi_1(\mathbf z,\mathbf y)|\leq C \|\psi\|_{C^\kappa} \|\phi \|_{C^\kappa} s w(B(\mathbf y,1))^{-1} (1+\|\mathbf z-\mathbf y\|)^{-1} \chi_{[0,2]}(d(\mathbf z,\mathbf y)).$$
Recall that $\|\mathbf x-\mathbf x'\|\leq 1$ {and $0<s\leq 1$}. Now, using part~\eqref{translations_kernels:b} of Theorem \ref{teo:translations_kernels}, we obtain 
\begin{equation*}
\begin{split}
|\phi_1&*(\psi_s*\phi_1)(\mathbf x,\mathbf y)-\phi_1*(\psi_s*\phi_1)(\mathbf x',\mathbf y)| \\
& = \Big|\int_{\mathbb{R}^N} (\phi_1(\mathbf x,\mathbf z)-\phi_1(\mathbf x',\mathbf z))(\psi_s*\phi_1)(\mathbf z,\mathbf y)\, dw(\mathbf z)\Big|\\
& \leq C  \|\psi\|_{C^\kappa} \|\phi \|_{C^\kappa}^2 s \int_{\mathbb{R}^N} \frac{\|\mathbf x-\mathbf x'\|}{(1+\|\mathbf x-\mathbf z\|)} w(B(\mathbf z,1))^{-1/2}(w(B(\mathbf x,1))^{-1/2}+w(B(\mathbf x',1))^{-1/2} )\\
&\hskip 1cm \times w(B(\mathbf y,1))^{-1} (1+\|\mathbf z-\mathbf y\|)^{-1} \chi_{[0,2]}(d(\mathbf z,\mathbf y))\, dw(\mathbf z)\\
& \leq  C  \|\psi\|_{C^\kappa} \|\phi \|_{C^\kappa}^2 s w(B(\mathbf x,1))^{-1} \|\mathbf x-\mathbf x'\| (1+\|\mathbf x-\mathbf y\|)^{-1}. 
\end{split}
\end{equation*}
 The proof is finished, because 
$$ |\phi_1*(\psi_s*\phi_1)(\mathbf x,\mathbf y)-\phi_1*(\psi_s*\phi_1)(\mathbf x',\mathbf y)|=0\quad \text{if } \quad d(\mathbf x,\mathbf y)\geq 4 \quad \text{and} \ \|\mathbf x-\mathbf x'\|\leq 1.$$ 

{\it Proof of~\eqref{teo_d}.} We repeat arguments of the proof of part~\eqref{teo_b}. Using~\eqref{eq:conv_psi_phi}, we get 
\begin{equation*}
    \begin{split}
       {\psi}_1*\phi_s(\mathbf x,\mathbf y)&-\psi_1*\phi_s(\mathbf x',\mathbf y)=s^m\int_{\mathbb{R}^N} (\Delta_k^{m/2} \psi_1(\mathbf{x},\mathbf{z})-\Delta_k^{m/2} \psi_1(\mathbf x',\mathbf z))\eta_s(\mathbf z,\mathbf y)\, dw(\mathbf z).
    \end{split}
\end{equation*}
Applying part~\eqref{translations_kernels:c} of Theorem \ref{teo:translations_kernels} to the function $\psi$, part ~\eqref{teo_a} to the function $\phi$, and having in mind that $\|\mathbf x-\mathbf x'\|\leq 1$ {and $0<s\leq 1$}, we get 
\begin{equation*}
    \begin{split}
        |\psi_1*\phi_s(\mathbf x,\mathbf y)&-\psi_1*\phi_s(\mathbf x',\mathbf y)|\\
        &\leq C\|\psi\|_{C^\kappa} \|\phi\|_{C^\kappa} s^m\int_{\mathbb{R}^N} \|\mathbf x-\mathbf x'\| w(B(\mathbf x,1))^{-1} (1+\|\mathbf x-\mathbf z\|)^{-1}\chi_{[0,2]}(d(\mathbf x,\mathbf z))\\
        &\hskip 1cm \times w(B(\mathbf y,s))^{-1} \Big(1+\frac{\|\mathbf z-\mathbf y\|}{s}\Big)^{-1}\chi_{[0,s]}(d(\mathbf z,\mathbf y))\, dw(\mathbf z)\\
        &\leq C \|\psi\|_{C^\kappa} \|\phi\|_{C^\kappa} s^m \|\mathbf x-\mathbf x'\| (1+\|\mathbf x-\mathbf y\|)^{-1} w(B(\mathbf x,1))^{-1}\chi_{[0,4]}((d(\mathbf x,\mathbf y)) .
    \end{split}
\end{equation*}
The proof of Theorem \ref{teo:bounds_conv} is complete. 
\end{proof}

\begin{minipage}{\textwidth}
    \printindex[Other]
\end{minipage}
\vspace{1cm} 
\begin{minipage}{\textwidth}
    \printindex[Const]
\end{minipage}

\section{Summary of key mathematical constants, their first appearance, and relationships}
\label{appendix:B}

\renewcommand{\arraystretch}{1.3}

\begin{table}[H]
    \centering
    \rowcolors{2}{gray!10}{white} 
    \begin{tabular}{>{$}c<{$} l m{10cm}} 
        \toprule
        \rowcolor{gray!40}
        \textbf{Constant} & \textbf{First Appearance} & \textbf{Relations/Comments} \\
        \midrule
        \varepsilon_0 & Theorem~\ref{KeyTheorem} & $\varepsilon_0 = (100C_{11}C_{12})^{-1}$ \\ 
        A & Theorem~\ref{KeyTheorem} & Fixed at the end of the proof \\ 
        A_0 & Theorem~\ref{KeyTheorem} & $A_0 = 32\sqrt{N}$ \\ 
        A_1 & Equation~\eqref{c11} & $A_1 = C_{11} \cdot \frac{100}{99}$ \\ 
        A_2 & Equation~\eqref{c12} & $A_2 = \frac{20}{9} C_{11}^2 A_1$ \\ 
        A_3 & Equation~\eqref{c13} & $A_3 = A_1 (1 + A_0)^{N_1 -1} C_{11}C_{12}^2 + \frac{100}{99} 4 C_{11}^2 C_{12}^2 (1+A_0)^{N_1 -1} A_1$ \\ 
        A_4 & Equation~\eqref{c6} & $A_4 = A_3 \sqrt{N} + C_{12} A_2$ \\ 
        C_{10} & Proposition~\ref{prop:changg} &  \\ 
        C_{11} & Lemma~\ref{cor:b_Q_ort} &  \\ 
        C_{12} & Lemma~\ref{lem:Ch-Gell12} &  \\ 
        C_{13} & Equation~\eqref{eq:g_0infty} & Does not depend on $f \in \mathrm{BMO}$ of compact support \\ 
        C_{14} & Equation~\eqref{G1} & $C_{14} = C_{11}C_{12}$ \\ 
        C_{15} & Equation~\eqref{G4} & $C_{15} = 4C_{11}^2 C_{12}^2 (1+A_0)^{N_1 -1} A_1$ \\ 
        C_{16} & Equation~\eqref{product} & \begin{tabular}{@{}l@{}} 
    $C_{16} = C_{11}^2 C'$, where $C' > 0$ satisfies \\[4pt]
    \quad $ \displaystyle \frac{1}{w(P)} \int (1 + d(\mathbf{x},\mathbf{z}_P)/\ell(P))^{-\mathbf{N} -1} \, dw(\mathbf{x}) \leq C'.$ 
\end{tabular} \\ 
        C_{17} & Equation~\eqref{eq:C_19_in} & Depends only on $(N, R, k)$ \\ 
        C_{18} & Equation~\eqref{eq:C_20_in} & $C_{18} = 2C_{16} \sqrt{C_{17}}$ \\ 
        C_{19} & Below~\eqref{eq:C_20_in} & Depends only on $(N, R, k)$ \\ 
        C_{20} & Below~\eqref{eq:C_20_in} & $C_{20} = C_{18} \sqrt{C_{19}}$ \\ 
        C_{21} & Below~\eqref{eq:ELL} & Depends only on $(N, R, k)$ \\ 
        C_{22} & Equation~\eqref{sumE} & $C_{22} = C_{12} A_2 C_{10} C_{21} 2^{N_0}$ \\ 
        C_{23} & Equation~\eqref{norm3} & $C_{23} = C_{J{\text{-}}N,2}(A_0/2)^{\mathbf{N}/2}(A_4+C_{22})$ \\ 
        C_{24} & Equation~\eqref{eq:FinBMO} & Depends on $(N,R,k)$ and $\overrightarrow{S}$ \\ 
        C_{25} & Equation~\eqref{eq:bound_F} & Depends on $(N,R,k)$ and $\overrightarrow{S}$ \\ 
        C_{26} & Equation~\eqref{eq:L2Ac} & Depends on $(N,R,k)$ and $\overrightarrow{S}$ \\  
        C_{J{\text{-}}N,p} & Lemma~\ref{lem:BMO_compactly_supported} & Constant in John-Nirenberg inequality \\ 
        M_1 & Equation~\eqref{eq:M_1} & $M_1 = 8 \lceil \mathbf{N} +1\rceil$ \\ 
        N_0 & Equation~\eqref{eq:constantsN} & $N_0 > 2\mathbf{N}$ \\ 
        N_1 & Equation~\eqref{eq:constantsN} & $N_0 + 1 < N_1 \leq (4M_1 + N -1)/2$ \\ 
        N_2 & Above~\eqref{product} & $N_2 = N_1 - \mathbf{N}$ \\ 
        \bottomrule
    \end{tabular}
    \label{tab:constants}
\end{table}


\begin{thebibliography}{99}

\bibitem{AH}
B. Amri, A. Hammi,
\emph{Dunkl-Schr\"odinger operators\/},
{Complex Anal. Oper. Theory 113, (2019), 1033-1058.}

\bibitem{ADH-JFAA}
J.-Ph. Anker, J. Dziuba\'nski, A. Hejna,
	\emph{Harmonic functions, conjugate harmonic functions and the Hardy space $H^1$ in the rational Dunkl setting}, J. Fourier Anal. Appl. 25 (2019), 2356--2418.





\bibitem{Christ-inversion}
M. Christ, \emph{ On the regularity of inverses of singular integral operators}, 
Duke Math. J. 57 (1988),  459--484.


\bibitem{CHG} M. Christ, D. Geller,  
\emph{Singular integral characterizations of Hardy spaces on homogeneous groups},
Duke Math. J. \textbf{51} (1984),  547--598.

\bibitem{CW}
R. Coifman and G. Weiss,
\emph{Extensions of Hardy spaces and their use in analysis},
Bull. Amer.
Math. Soc. 83, 4 (1977), 569–645.


\bibitem{dJ1}
	M.F.E. de Jeu,
	\emph{The Dunkl transform\/},
	Invent. Math. 113 (1993), 147--162.


 
\bibitem{Dunkl}
	C.F. Dunkl,
	\emph{Differential-difference operators associated to reflection groups\/},
	Trans. {Amer}. Math. 311 {(1989), no. 1,} 167--183{.}


 
\bibitem{D5}
	C.F.~Dunkl,
	\emph{Hankel transforms associated to finite reflection groups},
in: {\it Proc. of the special session on hypergeometric functions on domains of positivity, Jack polynomials and applications,}
	 Proceedings, Tampa 1991, Contemp. Math. 138 (1989),
	 123--138.


 

\bibitem{DH-JFAA}
    J. Dziuba\'nski, A. Hejna, \emph{Remarks on Dunkl translations of non-radial kernels,}  J. Fourier Anal. Appl. 29:52 (2023), https://doi.org/10.1007/s00041-023-10034-2. 
     


\bibitem{DzHe-JFA} 
J. Dziuba\'nski and A. Hejna, 
    \emph{H\"ormander's multiplier theorem for the Dunkl transform},
     Journal of Functional Analysis 277 (2019), 2133-2159.

\bibitem{DH-Studia}
	J. Dziuba\'nski, A. Hejna,
	\emph{Remark on atomic decompositions for the Hardy space $H^1$ in the rational Dunkl setting},
    Studia Math. 251 (2020), no. 1, 89--110.     


\bibitem{DH-heat}
	J. Dziuba\'nski, A. Hejna,
	\emph{Upper and lower bounds for the Dunkl heat kernel,} Calc. Var. Partial Differential Equations 62 (2023), no.1, Paper No. 25, 18 pp. 




\bibitem{JFAA} J. Dziubański and K. Jotsaroop, 
\emph{On Hardy and BMO Spaces for Grushin Operator},
 J. Fourier Anal. Appl. 22 (2016), 954-995.

\bibitem{FeSt} C. Fefferman, E.M. Stein, \emph {$H^p$ spaces of several variables},  Acta Math. 129, 137--195 (1972).

\bibitem{HLLW}  Y. Han, M-Y. Lee, J. Li, B. D. Wick, \emph{Riesz transform and commutators in the Dunkl setting,} Analysis and Mathematical Physics,  Volume 14, article number 46, (2024).


\bibitem{Han_et_al} Y. Han,  M-Y Lee, J. Li, B.D. Wick, \emph{Lipschitz and Triebel-Lizorkin spaces, commutators in Dunkl setting}, Nonlinear Anal. 237 (2023), Paper No. 113365, 36 pp. 

\bibitem{Janson}
    S. Janson, 
\emph{Characterizations of $H^1$ by singular integral transforms on martingales and $\mathbb{R}^n$}.
Math. Scand. 41 (1977), no. 1, 140–152


\bibitem{JL}
    J. Jiu and Z. Li,
    \emph{The dual of the Hardy space associated with the Dunkl operators},
    Advances in Mathematics, 
    Volume 412, (2023).


\bibitem{Roesle99}
	M. R\"osler,
	\emph{Positivity of Dunkl's intertwining operator\/},
	Duke Math. J. 98 (1999), no. 3, 445--463.


\bibitem{Roesler2003}
	M. R\"osler,
	\emph{A positive radial product formula for the Dunkl kernel\/},
	Trans. Amer.Math. Soc. 355 (2003), no. 6, 2413--2438{.}

 \bibitem{Roesler3}
	M. R\"osler:
	\emph{Dunkl operators (theory and applications).
	In: Koelink, E., Van Assche, W. (eds.)
	Orthogonal polynomials and special functions} (Leuven, 2002), 93--135.
	Lect. Notes Math. 1817, Springer-Verlag (2003).
	
\bibitem{Roesler-Voit}
	M.~R\"osler, M. Voit,
	\emph{Dunkl theory, convolution algebras, and related Markov processes\/},
in \textit{Harmonic and stochastic analysis of Dunkl processes\/},
{P. Graczyk, M. R\"osler, M. Yor (eds.), 1--112, Travaux en cours 71,}
Hermann, Paris, 2008.

\bibitem{Stein-Harmonic} E.M. Stein, \emph{Harmonic Analysis (Real Variable Methods, Orthogonality and Oscillatory Integrals)}, Princeton Math. Ser. 43, Princeton Univ. Press, 1993. 

\bibitem{Stein-Weiss} E.M. Stein, G. Weiss, {\emph{On the theory of harmonic functions of several variables, I The theory of $H^p$ spaces,}} Acta Math. 103 (1960), 25--62. 

\bibitem{ThangaveluXu}
	S. Thangavelu, Y. Xu,
	\emph{Convolution operator and maximal function for the Dunkl transform},
	 J. Anal. Math. 97 (2005),
	 25--55.
     

\bibitem{Uchiyama} A. Uchiyama,  
\emph{  A constructive proof of the Fefferman-Stein decomposition of  ${\rm BMO}(\mathbf R^n )$  }, Acta Math.148(1982), 215--241
 
\end{thebibliography}
 \end{document}